\documentclass[11pt,reqno]{amsart}

\numberwithin{equation}{section}
\usepackage{times}
\usepackage{amsmath,amsfonts,amstext,amssymb,amsbsy,amsopn,amsthm,eucal}
\usepackage{txfonts}
\usepackage{dsfont}
\usepackage{graphicx}   
\usepackage{hyperref}
\usepackage{accents}
\usepackage{enumerate}

\newcommand{\Rm}{{\rm Rm}}
\newcommand{\Ric}{{\rm Ric}}

\newcommand{\Vol}{{\rm Vol}}
\newcommand{\diam}{{\rm diam}}

\newcommand{\rv}{{\rm v}}

\newcommand{\dN}{\mathds{N}}
\newcommand{\dP}{\mathds{P}}

\newcommand{\dR}{\mathds{R}}

\newcommand{\dZ}{\mathds{Z}}

\newcommand{\R}{\text{R}}

\newcommand{\cB}{\mathcal{B}}

\newcommand{\cC}{\mathcal{C}}

\newcommand{\cN}{\mathcal{N}}

\newcommand{\cS}{\mathcal{S}}

\newcommand{\cV}{\mathcal{V}}

\newtheorem{theorem}[equation]{Theorem}

\newtheorem{proposition}[equation]{Proposition}
\newtheorem{lemma}[equation]{Lemma}
\newtheorem{corollary}[equation]{Corollary}
\theoremstyle{definition}
\newtheorem{definition}[equation]{Definition}

\theoremstyle{remark}
\newtheorem{remark}[equation]{Remark}
\theoremstyle{remark}
\newtheorem{example}[equation]{Example}
\theoremstyle{remark}

\theoremstyle{remark}\newtheorem{conjecture}[equation]{Conjecture}
\theoremstyle{remark}\newtheorem{question}[equation]{Question}

\begin{document}

\author{Jeff Cheeger and Aaron Naber}
\thanks{The first author was supported by NSF Grant DMS-1406407 and a Simons Foundation Fellowship.}
\thanks{The second author was supported by NSF grant DMS-1406259}

\title[Regularity of Einstein Manifolds and
the Codimension $4$ Conjecture]{Regularity of Einstein Manifolds \\ and the Codimension $4$ Conjecture}

\address{J. Cheeger, Courant Institute, 251 Mercer St., New York, NY 10011}
\email{cheeger@cims.nyu.edu}
\address{A. Naber, Department of Mathematics, 2033 Sheridan Rd., Evanston, IL 60208-2370}
\email{anaber@math.northwestern.edu}
\date{\today}
\maketitle

\begin{abstract}
In this paper, we are concerned with the regularity of noncollapsed Riemannian manifolds $(M^n,g)$ 
with bounded Ricci curvature, as well as their Gromov-Hausdorff
 limit spaces 
$(M^n_j,d_j)\stackrel{d_{GH}}{\longrightarrow} (X,d)$, where $d_j$ denotes the Riemannian distance.
Our main result is a solution to the codimension $4$ conjecture, namely that $X$ is smooth away from a closed 
subset of codimension $4$. 
 We combine this result with the ideas of quantitative stratification to prove a priori 
$L^q$ estimates on the full curvature $|\Rm|$ for all $q<2$. In the case of Einstein manifolds, we improve this to estimates on the regularity scale.  
We apply this to prove a conjecture of Anderson that the collection of 
$4$-manifolds $(M^4,g)$ with $|\Ric_{M^4}|\leq 3$, 
$\Vol(M)>\rv>0$, and $\diam(M)\leq D$  contains at most a 
finite number of diffeomorphism classes.  A local version 
is used to show that noncollapsed $4$-manifolds with 
bounded Ricci curvature have a priori $L^2$ Riemannian 
curvature estimates.
\end{abstract}

\tableofcontents

\section{Introduction}\label{s:intro}

In this paper, we consider pointed Riemannian manifolds $(M^n,g,p)$ with bounded Ricci curvature 
\begin{align}\label{e:bddricci}
|\Ric_{M^n}|\leq n-1\, ,
\end{align}
 which satisfy the noncollapsing assumption 
\begin{align}
\label{e:noncollapsed}
\Vol(B_1(p))>\rv>0\, .
\end{align}
We will be particularly concerned with pointed Gromov-Hausdorff limits

\begin{align}
\label{e:limits}
(M^n_j,d_j,p_j)\stackrel{d_{GH}}{\longrightarrow} (X,d,p)
\end{align}
of sequences of such manifolds, where $d_j$ always denotes the Riemannian distance.  Our main result is that $X$ is smooth away from a closed subset of 
codimension $4$.\footnote{In the K\"ahler case, this  was shown in \cite{Cheeger}, 
and independently by Tian, by means of an $\epsilon$-regularity theorem
which exploits the first Chern form and its relation to Ricci curvature.} 
We will combine this with the previous work 
of the authors on quantitative stratification to show that $X$ satisfies a priori
$L^q$-estimates on the curvature $|\Rm|$ for all $q<2$; see Theorems \ref{t:main_codim4} and \ref{t:main_estimate}. 
 Finally, we will apply the results in the dimension $4$ setting in which there are various improvements, including a
  finiteness theorem up to diffeomorphism and an a priori
$L^2$ curvature bound, for noncollapsed manifolds with bounded Ricci  curvature; see Theorems \ref{t:main_finite_diffeo} 
and \ref{t:main_dimfour}.

The first major results on limit spaces satisfying (\ref{e:bddricci})--(\ref{e:limits}) were proved in 
 Einstein case; see \cite{A89}, \cite{BKN89},  \cite{T90}. 
 A basic assumption 
 is that the $L^{n/2}$ norm of the of the curvature tensor is bounded. 
 From  this, together with an
appropriate $\epsilon$-regularity
 theorem,
  it was shown that 
any limit space as above is smooth away from at most
a definite number of points at which the singularities are of orbifold type. 

  In dimension $4$, given (\ref{e:bddricci}),
 it follows directly from the Chern-Gauss-Bonnet formula that 
the $L^2$-norm of the curvature is bounded in terms of the Euler characteristic.
In \cite{Anderson_Einstein}, it is shown
 that the collection of noncollapsed $4$-manifolds with definite bounds on Ricci curvature, diameter and  Euler charactersistic, contains only finitely many diffeomorphism types.  
Assuming an $L^{n/2}$ bound on
curvature (in place of the Euler characteristic bound) the finiteness theorem was extended to arbitrary dimensions in
\cite{Anderson-Cheeger2}, a precursor of which was \cite{B90}.

It was conjectured in \cite{Anderson_ICM94} that for the finiteness theorem
in dimension $4$, the 
Euler characteristic 
bound 
is an unnecessary assumption. In Theorem \ref{t:main_finite_diffeo},
we prove this conjecture.

 
 The first step toward the study of such 
Gromov-Hausdorff limits as in (\ref{e:bddricci})--(\ref{e:limits}), without the need for  assumptions implying
integral curvature bounds, was taken in \cite{ChC1}. There, a stratification 
theory for noncollapsed limits with only lower Ricci curvature bounds was developed.  By combining 
this with the $\epsilon$-regularity results of \cite{Anderson_Einstein}, it was proved that a
noncollapsed Gromov-Hausdorff limit of manifolds  with bounded Ricci curvature is smooth outside a
closed  subset of codimension $2$.  More recently, it was shown in \cite{CheegerNaber_Ricci} that one 
can then prove a priori $L^q$-bounds on the curvature
for all $q<1$. 

Based on knowledge of the $4$-dimensional case, early workers conjectured that the 
singular set of a noncollapsed 
limit space satisfying (\ref{e:bddricci})--\eqref{e:limits} should have codimension $4$.
This was shown in \cite{CCTi_eps_reg} under the additional assumption of an
$L^q$ curvature bound  for all $q<2$. 
The following is the main result of this paper:

\begin{theorem}\label{t:main_codim4}
Let $(M^n_j,d_j,p_j)\stackrel{d_{GH}}{\longrightarrow} (X,d,p)$ be a Gromov-Hausdorff limit of manifolds 
with $|\Ric_{M^n_j}|\leq n-1$ and $\Vol(B_1(p_j))>\rv>0$.  Then the singular set $\cS$ satisfies
\begin{align}
\dim \cS\leq n-4.
\end{align}
The dimension can be taken to be the Hausdorff or Minkowski dimension.
\end{theorem}

We will outline the proof of Theorem \ref{t:main_codim4} in subsection \ref{ss:outline}. First we will discuss various applications.  Our first applications are to the regularity theory of Einstein manifolds.  To make this precise, let us begin with the following definition, see also \cite{CheegerNaber_Ricci}:

\begin{definition}\label{d:regularity_scale}
 For $x\in X$ we define the regularity scale $r_x$ by
\begin{align}
r_x\equiv \max_{0<r\leq 1}\big\{\sup_{B_r(x)}|\Rm|\leq r^{-2}\big\}\, .
\end{align}
If $x\in \cS$ is in the singular set of $X$, then $r_x\equiv 0$.
\end{definition}

Let $T_r(S)= \{x\in M: d(x,S)<r\}$ denote the $r$-tube around the set $S$. 
By combining Theorem \ref{t:main_codim4} with the quantitative stratification ideas of \cite{CheegerNaber_Ricci}, we 
show the following:

\begin{theorem}\label{t:main_estimate}
There exists $C=C(n,\rv,q)$ such that if $M^n$ satisfies $|\Ric_{M^n}|\leq n-1$ and $\Vol(B_1(p))>\rv>0$.  
then for each $q<2$,
\begin{align}\label{e:main_Rm_est}
\fint_{B_1(p)} |\Rm|^q\leq C\, .
\end{align}
If in addition, $M^n$ is  assumed to be Einstein, then 
for every $q<2$ we have that
\begin{align}\label{e:main_reg_scale}
\Vol(T_r(\{x\in B_1(p):r_{x}\leq r\}))\leq C\, r^{2q}
\end{align}
\end{theorem}

\begin{remark}
If we  replace the assumption that $M^n$ is Einstein with just a bound on $|\nabla\Ric_{M^n}|$ 
we obtain the same conclusion.  In fact, if we only assume a bound on the Ricci curvature $|\Ric_{M^n}|$, then  (\ref{e:main_reg_scale}) holds
with the regularity scale $r_x$ replaced by the harmonic radius $r_h$; see Definition \ref{d:harmonic_radius}.
  Note that estimates on the regularity scale are {\it much} stronger than corresponding 
$L^q$ estimates for the curvature given in (\ref{e:main_Rm_est}).
\end{remark}

The final theorems of the paper concern the $4$-dimensional case in which we can make some 
marked improvements on the results in the general case.  Let us begin with the following, which is a conjecture of 
Anderson \cite{Anderson_ICM94}.  

\begin{theorem}\label{t:main_finite_diffeo}
There exists $C=C(\rv,D)$ such that if  $M^4$ satisfies $|\Ric_{M^4}|\leq 3$, $\Vol(B_1(p))>\rv>0$ 
and $\diam (M^n)\leq D$, then  $M^4$ can have one of at
 most $C$ diffeomorphism types.
\end{theorem}

By proving a more local version of the above theorem, we can improve Theorem \ref{t:main_estimate} in
 the $4$-dimensional case and show that the $L^q$ bounds on the curvature for $q<2$ may be pushed 
all the way to an a priori $L^2$ bound in dimension $4$.  
We conjecture in Section \ref{s:conjectures} that this holds in all dimensions.

\begin{theorem}\label{t:main_dimfour}
There exists $C=C(\rv)$ such that if $M^4$ satisfies $|\Ric_{M^4}|\leq 3$ and $\Vol(B_1(p))>\rv>0$, then
\begin{align}\label{e:main_dimfour:1}
\fint_{B_1(p)} |\Rm|^2\leq C\, .
\end{align}
Furthermore, we have the sharp weak type $L^2$ estimate on the harmonic radius, 
\begin{align}\label{e:main_dimfour:2}
\Vol(T_r(\{x\in B_1(p):r_h\leq r\}))\leq Cr^4\, .
\end{align}
If we assume in addition that $M^4$ is Einstein,  then the same result holds with the harmonic radius 
$r_h$ replaced by the regularity scale $r_x$. \end{theorem}
\begin{remark}
If the assumption that $M^4$ is Einstein is weakened to assuming a bound on $|\nabla\Ric_{M^n}|$, then 
\eqref{e:main_dimfour:2} still holds with the harmonic radius $r_h$ replaced by the stronger regularity scale $r_x$.
\end{remark}

Next, we will give a brief outline of the paper.  We begin in subsection \ref{ss:outline} by outlining the proof of Theorem \ref{t:main_codim4}.  This  includes statements and explanations of some of the main technical theorems of the paper.  

In Section \ref{s:background} we go over some basic background and preliminary material.  This includes the basics of stratifications for limit spaces, the standard $\epsilon$-regularity theorem for spaces with bounded Ricci curvature, and some motivating examples.  

Sections \ref{s:transformation} and \ref{s:slicing} are the the most crucial sections of the paper. 
 There, we prove Theorems \ref{t:transformation_theorem_intro} and Theorem \ref{t:slicing_intro}, the 
Transformation and Slicing theorems which, 
roughly speaking, allow us to blow up along a collection of points which is large enough to see 
into the singular set; see Section \ref{ss:outline} for more on this.

Section \ref{s:codim4} is dedicated to proving the main result of the paper, Theorem \ref{t:main_codim4}.  
The argument is a blow up argument that exploits the Slicing Theorem of Section \ref{s:slicing}.  
In Section \ref{s:eps_reg},  based on  Theorem \ref{t:main_codim4}, we give a new $\epsilon$-regularity
 theorem.   Theorem \ref{t:eps_reg} states that if a ball in a space with bounded Ricci curvature
 is close  enough in the Gromov-Hausdorff sense to a ball in a metric cone, $\R^{n-3}\times C(Z)$, then the concentric ball 
of half the radius must be smooth. 

In Section \ref{s:qs_estimates}, the $\epsilon$-regularity theorem of Section \ref{s:eps_reg} 
is combined with the ideas of quantitative stratification to give effective improvements on all
 the results of the paper.  We show that the singular set has codimension $4$ in the Minkowski 
sense, and give effective estimates for tubes around the 
regions of curvature concentration.  
This culminates in the proof of Theorem \ref{t:main_estimate}.  In subsection \ref{ss:einstein_harmonic_estimates}, 
we use the effective estimates of Theorem \ref{t:main_estimate} to prove new estimates for harmonic
 functions on spaces with bounded Ricci curvature.  These estimates, which
  can fail on manifolds with 
only lower Ricci curvature bounds, give the first taste of how analysis on
  manifolds with bounded Ricci curvature improves over that on manifolds with only lower Ricci curvature bounds.

Finally, in Section \ref{s:dimension_four}, we discuss the $4$-dimensional case and prove 
the finiteness up diffeomorphism theorem, Theorem \ref{t:main_finite_diffeo}.   We also prove the 
$L^2$ curvature estimates of Theorem \ref{t:main_dimfour}.

\subsection{Outline of the proof Theorem \ref{t:main_codim4},
 the codimension 4 conjecture}\label{ss:outline}

Let $S^1_\beta$ denote the circle of circumference $\beta<2\pi$.
It has been understood since \cite{ChC1} that to prove Theorem \ref{t:main_codim4}, the key step
 is to show that the cone $\dR^{n-2}\times C(S^1_\beta)$ does not occur as the
(pointed)  Gromov-Hausdorff limit
of some sequence $M^n_j$ with $|\Ric_{M^n_j}|\to 0$. This was shown in
  \cite{CCTi_eps_reg}
assuming just a lower bound $\Ric_{M^n_i}\geq -(n-1)$, but with the additional
assumption that the $L^1$ norm of the curvature is sufficiently small. In \cite{Cheeger},
 it was proved for the K\"ahler-Einstein case, which was also done by Tian.  A common feature of 
  the proofs 
is an argument by contradiction, implemented by the use of harmonic almost splitting 
maps $u:B_2(p)\to \dR^{n-2}$,
 see Lemma \ref{l:harmonic_splitting_intro}.  In each case, it is shown that for most points 
$s\in \dR^{n-2}$ in the range, the slice $u^{-1}(s)$ has a certain good property which,
when combined with the assumed curvature bounds,
 enables one to deduce a contradiction.  In particular, in \cite{CCTi_eps_reg} it is shown
 that most slices $u^{-1}(s)$ have integral bounds on the second fundamental form, which 
when combined with the assumed integral curvature bounds, enables one apply the Gauss-Bonnet
 formula for $2$-dimensional manifolds with boundary, to derive a contradiction.



However,  prior to the present paper it was not known how, in the general case, to implement 
a version of the above strategy which would rule out the cones $\dR^{n-2}\times C(S^1_\beta)$, 
without assuming the integral curvature bounds.  In the remainder of this subsection, we will
 state the main results which are used in the present implementation,
 that enables us to prove 
Theorem \ref{t:main_codim4}. 

Thus, we consider a sequence of Riemannian manifolds $(M^n_j,d_j,p_j)$,
with $|\Ric_{M^n_j}|\to 0$ and $\Vol(B_1(p_j)>\rv>0$, such that
\begin{align}
(M^n_j,d_j,p_j)\stackrel{d_{GH}}{\longrightarrow} \dR^{n-2}\times C(S^1_\beta)\, . 
\end{align}
  As above, we have harmonic almost splitting maps 
\begin{align}\label{e:outline:almost_splitting_maps}
u_j:B_2(p_j)\to \dR^{n-2}\, ,
\end{align}
see Lemma \ref{l:harmonic_splitting_intro} below.  The key ingredient will be 
Theorem \ref{t:slicing_intro} (the Slicing Theorem), which states  that there exist
 $s_j\in \dR^{n-2}$ such that for all $x\in u^{-1}_j(s_j)$ and for all $r<1$, the ball
$B_r(x)$ is $\epsilon_jr$-close in the Gromov-Hausdorff sense to a ball in an isometric
product $\dR^{n-2}\times S_{j,x,r}$, where $\epsilon_j\to 0$ as $j\to \infty$.

Granted this, we can apply a blow up argument in the spirit of \cite{Anderson_Einstein}
 to obtain a contradiction.  Namely, it is easy to see that 
 since $\beta<2\pi$ then the minimum of the 
harmonic
radius $r_h$ at points of the slice $u^{-1}_j(s_j)$ is obtained at some $x_j\in u^{-1}_j(s_j)$ 
and is going to
zero as $j\to\infty$.  We rescale the metric by the inverse of the harmonic radius $r_j=r_h(x_j)$ 
and find a subsequence converging in the pointed Gromov-Hausdorff sense a smooth noncompact
 Ricci flat manifold, 
\begin{align}
(M^n_j, r_j^{-1}d_j,x_j)\to (X,d_X,x)\, ,
\end{align}
such that $X= \dR^{n-2}\times S$ splits off $\dR^{n-2}$ isometrically, with $S$ a 
smooth two dimensional surface.  It follows that $S$ is Ricci flat, and hence flat. From 
the noncollapsing assumption, it follows that $X$ has Euclidean volume growth. 
Thus, $X=\dR^n$ is Euclidean space.   However, the $2$-sided Ricci bound 
implies that  the harmonic radius behaves continuously in the limit. Hence, the harmonic 
radius at $x$ is $r_h(x_\infty)=1$; a contradiction.  See Section \ref{ss:codim2} for
 more details on the blow up argument.\\

Clearly then, the key issue is to show the existence of the points $s_j\in \dR^{n-2}$, such
 that at {\it all} points $x\in u_j^{-1}(s_j)$, we  have the above mentioned splitting property 
on $B_r(x)$ for {\it all} $r<1$. To indicate
the proof,
we now recall some known connections between isometric splittings, the Gromov-Haudorff distance
and harmonic maps to Euclidean spaces $\dR^k$.  We begin with a definition.
\begin{definition}
\label{d:intro}
A $\epsilon$-splitting map  $u=(u^1,\ldots,u^{k}):B_r(p)\to\dR^{k}$ is a harmonic map 
such that:

\begin{enumerate}

\item $|\nabla u|\leq 1+\epsilon$.
\vskip2mm

\item $\fint_{B_r(p)} |\langle \nabla u^\alpha,\nabla u^\beta\rangle-\delta^{\alpha\beta}|^2<\epsilon^2$.
\vskip2mm

\item $r^2\fint_{B_r(p)} |\nabla^2 u^\alpha|^2<\epsilon^2$.
\end{enumerate}
\end{definition}
\noindent
Note that the condition that $u$ is harmonic is equivalent to the harmonicity of the individual component functions $u^1,\dots,u^k$.

The following lemma summarizes the basic facts about splitting maps\footnote{
In \cite{ChC1}, only a uniform bound $|\nabla u|<C(n)$ is proved. This would actually suffice for our present
purposes.
The improved bound, $|\nabla u|<1+\epsilon$, in (1) above, is derived in (\ref{e:sharp1})--(\ref{e:sharp3}), in a context 
that passes over almost verbatim to the present one.}

\begin{lemma}[\cite{ChC1}]\label{l:harmonic_splitting_intro}
For every $\epsilon,R>0$ there exists $\delta =\delta(n,\epsilon,R)>0$ such that if 
$\Ric_{M^n}\geq -(n-1)\delta$ then:
\begin{enumerate}
\item
 If $u:B_{2R}(p)\to\dR^k$ is a $\delta$-splitting map, then there exists a 
map $f:B_{R}(p)\to u^{-1}(0)$ such that 
$$
(u,f):B_{R}(p)\to \dR^k\times u^{-1}(0)\, ,
$$
is an $\epsilon$-Gromov Hausdorff map, where $u^{-1}(0)$ is given the induced metric.
\vskip1mm

\item  If
\begin{align}
d_{GH}(B_{\delta^{-1}}(p),B_{\delta^{-1}}(0))<\delta,
\end{align}
where $0\in \dR^k\times Y$, then there exists an $\epsilon$-splitting map $u:B_{R}(p)\to \dR^k$.  
\end{enumerate}
\end{lemma}

Let us return to the consideration of 
the  maps $u_j$ from (\ref{e:outline:almost_splitting_maps}), which in our situation arise from (2) of Lemma \ref{l:harmonic_splitting_intro}.
We can thus assume that the $u_j$ are $\delta_j$-splitting maps, with $\delta_j\to 0$.  
We wish to find slices $u^{-1}_j(s_j)$ such that $B_r(x)$ continues to almost split for all 
$x\in u_j^{-1}(s_j)$ and all $r\leq 1$. 
One might hope that there always exist $s_j$ such that
 by restricting the map $u_j$ to each such ball $B_r(x)$, one obtains an
$\epsilon_j$-splitting map.
 However, it turns out that there are counterexamples to this statement; see Example \ref{sss:Example3_ConeSpace}.

The essential realization  is that for our purposes,  it actually suffices to show the existence of 
$s_j$ such that for all $x\in u_j^{-1}(s_j)$ and all $0<r\leq 1$, there exists a 
matrix $A=A(x,r)\in  GL(n-2)$,
such that the harmonic map
$A\circ u_j:B_r(x)\to \dR^{n-2}$ is our desired 
$\epsilon_j$-splitting map.
Thus, while $u_j$ might not itself be an $\epsilon_j$-splitting map on $B_r(x)$, it 
might only differ from one by 
a linear transformation of the image.
 This turns out to hold. In fact, we 
will show that $A$ can be chosen to be lower triangular with positive diagonal entries. Since
this condition  plays a role in the proof of Theorem \ref{t:transformation_theorem_intro} below,
we will incorporate it from now on.

\begin{theorem}\label{t:slicing_intro} (Slicing theorem)
For each $\epsilon>0$ there exists $\delta(n,\epsilon)>0$ such that if $M^n$ 
satisfies $\Ric_{M^n}\geq -(n-1)\delta$ and if $u:B_2(p)\to \dR^{n-2}$ is a harmonic 
$\delta$-splitting map, then there exists a subset $G_\epsilon\subseteq B_1(0^{n-2})$ 
which satisfies the following:
\begin{enumerate}
\item $\Vol(G_\epsilon)>\Vol(B_1(0^{n-2}))-\epsilon$.
\vskip1mm

\item If $s\in G_\epsilon$ then $u^{-1}(s)$ is nonempty.
\vskip1mm

\item For each $x\in u^{-1}(G_\epsilon)$ and $r\leq 1$ there exists a lower triangular matrix
$A\in GL(n-2)$ with positive diagonal entries such that $A\circ u:B_r(x)\to \dR^{n-2}$ is an $\epsilon$-splitting map.
\end{enumerate}
\end{theorem}

The proof of the Slicing Theorem is completed in Section \ref{s:slicing}.  In the next subsection,
we give the proof modulo the key technical results on which it depends. These will be indicated
in the remainder of the present subsection.

Given a harmonic function with values in $\R^k$,
for $1\leq\ell\leq k$, we put
\begin{align}
\label{e:omegadef}
\omega^\ell=:du^1\wedge\cdots\wedge du^\ell\, .
\end{align}
The forms, $\omega^\ell$, and in particular
the Laplacians, $\Delta |\omega^\ell|$, of their norms,
play a key role in the sequel.   We
point out that in general, $\Delta |\omega^\ell|$ is a distribution, not just a function.  Put
\begin{align}\label{e:nodal_sets}
&Z_{|\nabla u^a|}=:\{x\, :\, |\nabla u^a|(x)=0\}\, ,\notag\\
&Z_{|\omega^\ell |}=:\{x\, :\, |\omega^\ell |(x)=0\}\, .
\end{align}
Then the functions $|\nabla u^a|$, $|\omega^\ell|$, are Lipschitz on $B_2(p)$ 
and are smooth away from $Z_{|\nabla u^a|}$, $Z_{|\omega^\ell |}$, respectively.  
An important structural point which is contained in the next theorem, is that $\Delta|\nabla u^a|$ is
 in fact a function and 
$\Delta |\omega^\ell|$ is at least a Borel measure. As usual,
$|\Delta |\omega^\ell||$ denotes the absolute value of the measure
$\Delta |\omega^\ell|$. Thus,
$\int_U |\Delta |\omega^\ell||$ denotes the
mass of the restriction of $\Delta|\omega^\ell|$ to $U$ and
$\fint_U |\Delta |\omega^\ell||$ denotes
this mass divided by $\Vol(U)$.




\begin{theorem}\label{t:hessian_estimate_intro} (Higher order estimates)
For every $\epsilon>0$ there exists $\delta(n,\epsilon)>0$ such that if  
$\, \,\Ric_{M^n}\geq -(n-1)\delta$ and $u:B_2(p)\to \dR^k$ is a
$\delta$-splitting map, 
then the following hold:
\begin{enumerate}
\item There exists $\alpha(n)>0$ such that for each $1\leq a\leq k$,
\begin{align}
\label{kato_1}
\fint_{B_{3/2}(p)} \frac{|\nabla^2 u^a|^2}{|\nabla u^a|^{1+\alpha}} < \epsilon\, .
\end{align}

\item Let $\omega^\ell\equiv du^1\wedge\cdots\wedge du^\ell$, $1\leq\ell\leq k$.  The Laplacians $\Delta |\omega^\ell|$
 taken in the distributional sense, are Borel measures with singular part a nonnegative locally finite Borel measure 
supported on $\partial Z_{|\omega^\ell |}$. For $\ell=1$, the singular part vanishes.  The normalized mass of $\Delta|\omega^\ell|$  
satisfies
\begin{equation}\label{e:kato2}
\fint_{B_{3/2}(p)} |\Delta |\omega^\ell|| < \epsilon\, .
\end{equation}
\end{enumerate}
\end{theorem}

\begin{remark} 
In actuality, we will need only the case $\alpha=0$ of (\ref{kato_1}).
\end{remark}

As will be clear from Theorem \ref{t:transformation_theorem_intro} below (the Transformation theorem) 
 that the following definition is key.

\begin{definition}\label{d:gauged_singular_scale_into}
Let $u:B_2(p)\to\dR^k$ be a harmonic function.   
For  $x\in B_1(p)$ and $\delta>0$,
define the {\it singular 
scale} $s^\delta_x\geq 0$ to be the infimum of all radii $s$ such that for all $r$ with
$s\leq r<\frac{1}{2}$ and all $1\leq \ell\leq k$ we have 
\begin{align}\label{e:gauged_regularity1}
r^2\fint_{B_r(x)} |\Delta|\omega^\ell|| \leq \delta \fint_{B_r(x)}|\omega^\ell|\, .
\end{align}
\end{definition}
\vskip3mm


Note that there is an invariance property for (\ref{e:gauged_regularity1}).  Namely, if (\ref{e:gauged_regularity1}) holds for $u$ then it holds for 
$A\circ u$ for any lower triangular matrix $A\in GL(k)$. 
 That is, the singular scale of $u$ and the singular scale of $A\circ u$ are equal.  In view of (\ref{e:kato2}), 
this means essentially that (\ref{e:gauged_regularity1}) is
a necessary condition for the existence of $A$ as in the Slicing theorem. 
Our next result, which  is by far the most technically difficult
 of the paper, provides a sort of converse.  We will not attempt to summarize
the proof  except to say that it involves a contradiction argument, as
well as an induction on $\ell$.  It is proved  in Section \ref{s:transformation}.

\begin{theorem}\label{t:transformation_theorem_intro} (Transformation theorem)
For every $\epsilon>0$, there exists 
$\delta=\delta(n,\epsilon)>0$ such that if 
$\Ric_{M^n}\geq -(n-1)\delta$
 and $u:B_2(p)\to \dR^k$  is a $\delta$-splitting map,
 then for each $x\in B_1(p)$ and $1/2\geq r\geq s^\delta_x$ there exists a lower triangular matrix $A=A(x,r)$ with 
positive diagonal entries such
 that $A\circ u:B_r(x)\to \dR^k$ is an  $\epsilon$-splitting map.
\end{theorem}
\vskip3mm

\subsection{Proof the the Slicing theorem modulo technical results}
\label{ss:pst}
Granted
Theorem \ref{t:hessian_estimate_intro} and Theorem \ref{t:transformation_theorem_intro},
 the Transformation theorem, we now give the proof of the Slicing theorem, modulo
additional  two technical results,
(\ref{e:doubling_intro}), (\ref{e:intro4}). These  will be seen in Section \ref{s:slicing} to be easy consequences
of the Transformation theorem.

 Fix $\epsilon>0$ as in the Slicing theorem.
We must show that
 there exists $\delta=\delta(n,\epsilon)$ such that if $u:B_2(p)\to \dR^{n-2}$ denotes a $\delta$-splitting map, then
 the conclusions of the the Slicing theorem
hold.

Let us write $\delta_3=\delta_3(n,\epsilon)$ for what was denoted by 
 $\delta(n,\epsilon)$ in 
the Transformation
theorem. 
Put
\begin{align}
\label{cBdef}
\cB_{\delta_3}=:\, \bigcup_{x\in B_1(p)\, |\, s^{\delta_3}_x>0} B_{s^{\delta_3}_x}(x)\, ,
\end{align}
We can assume that $\delta$ of the Slicing theorem is small
enough that $s_x^{\delta_3}\leq 1/32$ (which will be used in (\ref{e:sum_intro})).

 Let $ |u(V)|$ denote the $(n-2)$-dimensional measure of $V\subset \dR^{n-2}$.
According to Theorem 2.37 of \cite{CCTi_eps_reg}, there exists $\delta_1=\delta_1(n,\epsilon/2)$ such that
if  $\delta(n,\epsilon)\leq \delta_1(n,\epsilon/2)$, then
\begin{equation}
\label{e:afv}
| B_1(0^{n-2}))\setminus u(B_1(p))|<\epsilon/2\, .
\end{equation}
It follows from the Transformation theorem and
(\ref{e:afv}), that if we choose $\delta$ to satisfy in addition $\delta\leq \delta_1$, then
to conclude the proof of the Slicing theorem, it suffices to show
that $\delta$ can be chosen so that we also have
\begin{equation}
\label{e:si}
|u(\cB_{\delta_3})|\leq \epsilon/2\, .
\end{equation}

To this end, we record two perhaps non-obvious, but easily verified
 consequences of Theorem \ref{t:transformation_theorem_intro}.

Denote by $\mu$, the measure such that for all open sets $U$
\begin{equation}
\label{mudef}
\mu(U)=\left(\int_{B_{3/2}(p)}|\omega|\right)^{-1}\cdot \int_{U}|\omega|\,  .
\end{equation}

The first consequence (see Lemma \ref{l:doubling}) is that for each $x\in B_1(p)$ and
$1/4\geq r\geq s^{\delta_3}_x$, we have the doubling condition
\begin{equation}\label{e:doubling_intro}
\mu(B_{2r}(x))\leq C(n)\cdot \mu(B_r(x))\, .
\end{equation}

The second consequence (see Lemma \ref{l:mu_vol_comparison}) is that if 
$x\in B_1(p)$ and $1/2\geq r\geq s^{\delta_3}_x$,
 then we have the volume estimate
\begin{equation}\label{e:intro4}
|u(B_{r}(x))|\leq  C(n)\cdot r^{-2}\mu(B_{r}(x)) \, .
\end{equation}

 The proof of these results exploits the fact that $A\circ u:B_r(x)\to \dR^{n-2}$ is an $\epsilon$-splitting map 
for some lower triangular matrix $A$ with positive diagonal entries.\\

By a standard covering lemma,
 there exists a collection of mutually disjoint balls, $\{B_{s_j}(x_j)\}$ with $s_j=s^{\delta_3}_{x_j}$, such that 
\begin{equation}
\label{e:5_intro}
\cB_{\delta_3}
\subset \bigcup_jB_{6s_j}(x_j)\, .
\end{equation}
 Since the balls $B_{s_j}(x_j)$ are mutually disjoint, we can apply Theorem \ref{t:hessian_estimate_intro} 
together with (\ref{e:intro4}) and the (three times iterated) doubling property \eqref{e:doubling_intro} of $\mu$ to obtain
\begin{align}\label{e:sum_intro}
|u(\cB_{\delta_3})|&\leq \sum_j |u(B_{6s_j}(x_j))|\leq \sum_j( {6s_{j}})^{-2}
\mu(B_{6s_{j}}(x_j))\notag\\
&\leq C(n)\sum_j s_{j}^{-2}\mu(B_{s_{j}}(x_j))
\leq C\delta_3^{-1}\sum_j  \left(\int_{B_{3/2}(p)}|\omega|\right)^{-1}\cdot \int_{B_{s_j}(x_j)}|\omega|\notag\\
&\leq C\delta_3^{-1}\cdot \fint_{B_{3/2}(p)}|\omega|\, .
\end{align}
(For the interated doubling property of $\mu$, we used $s^{\delta_3}_x\leq 1/32$.)

Write $\delta_2(n,\,\cdot\,)$ for what was denoted by $\delta(n,\,\cdot\,)$ in Theorem \ref{t:hessian_estimate_intro}.
If in addition we choose $\delta\leq \delta_2(n,\frac{1}{2}C^{-1}\delta_3\epsilon)$, where $C=C(n)$ is
the the constant on the last line in (\ref{e:sum_intro}), then by Theorem \ref{t:hessian_estimate_intro},
 the right hand side of (\ref{e:sum_intro}) is
$\leq \epsilon/2$; i.e. (\ref{e:si}) holds.
As we have noted, this suffices to complete the proof of the Slicing theorem.


\section{Background and Preliminaries}\label{s:background}

In this section we review from standard constructions and techniques, which will be used throughout the paper.

\subsection{Stratification of Limit Spaces}\label{ss:stratification}
In this subsection we  recall some basic properties of pointed Gromov-Hausdorff limit spaces
\begin{align}
(M^n_j,d_j,p_j)\stackrel{d_{GH}}{\longrightarrow} (X,d,p)\, ,
\end{align}
where the $\Ric_{M^n_j}\geq -(n-1)$  and the noncollapsing assumption $\Vol(B_1(p_j))\geq \rv>0$ holds.  
In particular, we recall the stratification of a noncollapsed limit space, which was first introduced in \cite{ChC1},
and which will play an important role in the proof of Theorem \ref{t:main_codim4}. 
 The effective version, called the {\it quantitative stratification}, which was first introduced
 in \cite{CheegerNaber_Ricci}, will be recalled in Section \ref{s:qs_estimates}. 
It will play an important role in the estimates of Theorem \ref{t:main_estimate}.

Given  $x\in X$,
 we call a metric space $X_x$ a {\it tangent cone} at $x$ if there exists a sequence $r_i\to 0$ such that 
\begin{align}
(X,r_i^{-1}d,x) \stackrel{d_{GH}}{\longrightarrow}X_x\, .
\end{align}
That tangent cones exist at every point is a consequence of  Gromov's compactness theorem;
see for instance the book \cite{Petersen_RiemannianGeometry}.  
A point is called {\it regular} if every tangent cone is isometric to $\dR^n$ and otherwise {\it singular}.
The set of singular points is denoted by $\cS$.
As explained below, for noncollapsed
limit spaces with a uniform lower Ricci bound,
the singular set has codimension $\geq 2$. At singular
points,
tangent cones may be highly nonunique, to the extent that neither the dimension of the singular set nor the homeomorphism 
type is uniquely defined; see for instance \cite{CoNa2}.  Easy examples show that the singular set need not be closed if 
one just assumes a uniform
a lower bound $\Ric_{M^n_j}\geq -(n-1)$.  However, under the assumption of a $2$-sided bound
$|\Ric_{M^n_j}|\leq (n-1)$, the singular set is indeed closed; see \cite{Anderson_Einstein}, \cite{ChC2}. 

 For noncollapsed limit spaces, as shown in \cite{ChC1}, every tangent cone is isometric to a {\it metric cone}, i.e.
\begin{align}
X_x=C(Z)\, ,
\end{align}
for some compact metric space $Z$, with $\diam(Z)\leq\pi$.  With this as our starting point, we introduce
the following notion of symmetry. 

\begin{definition}
A metric space $Y$ is called {\it $k$-symmetric} if $Y$ is isometric to 
$ \dR^k\times C(Z)$ for some compact metric
 space $Z$.  We  define the {\it closed $k$th-stratum} by
\begin{align}
\cS^k(X)=:\{x\in X: \text{ no tangent cone at $x$ is $(k+1)$-symmetric}\}
\end{align}
\end{definition}
Thus, in the noncollapsed case, every tangent cone is $0$-symmetric.

The key result of \cite{ChC1} is the following:
\begin{align}
\dim \cS^k\leq k\, ,
\end{align}
where dimension is taken in the Hausdorff sense.  Thus, away from a set of Hausdorff dimension 
$k$, every point has some tangent cone with $(k+1)$ degrees of symmetry. 
 For an effective refinement of this theorem see \cite{CheegerNaber_Ricci} and Section \ref{s:qs_estimates}.\\

\subsection{$\epsilon$-Regularity Theorems}\label{ss:eps_reg_review}
A central result of this paper is the $\epsilon$-regularity theorem,  Theorem \ref{t:eps_reg}.  The original $\epsilon$-regularity 
theorems for Einstein manifolds were given in \cite{BKN89}, \cite{A89},
 \cite{T90}.  They state that if $M^n$ is an 
Einstein manifold with $\Ric_{M^n}=\lambda g$,  $|\lambda|\leq n-1$,  
$\Vol(B_1(p))\geq \rv$ and 
\begin{align}
\int_{B_2(p)}|\Rm|^{n/2} <\epsilon(n,\rv)\, ,
\end{align}
then $\sup_{B_1(p)}|\Rm|\leq 1$. 

In \cite{CCTi_eps_reg}, \cite{Cheeger}, \cite{Chen_Donaldson},  $\epsilon$-regularity theorems were proved under the assumption
of $L^q$ curvature bounds, $1\leq q <n/2$, provided $B_2(p)$ is assumed sufficiently close to a ball in a cone which splits off
an isometric factor $\dR^{n-2q}$. 

On the other hand, the regularity theory of \cite{CheegerNaber_Ricci} for Einstein manifolds, depends on 
$\epsilon$-regularity theorems which do not assume $L^q$ curvature bounds.  In particular, it follows from the work of \cite{Anderson_Einstein} 
 that there exists $\epsilon(n)>0$ such that if $|\Ric_{M^n}|\leq \epsilon(n)$ and if
\begin{align}
d_{GH}(B_2(p),B_2(0^n))<\epsilon(n)\, ,
\end{align}
where $B_2(0^n)\subseteq \mathbb{R}^n$, then  $|\Rm|\leq 1$ on $B_1(p)$. 
 
 This result can be extended in several directions.  In order to state the extension in full generality, we first  recall the notion of the harmonic radius:

\begin{definition}\label{d:harmonic_radius}
For $x\in X$, we define the harmonic radius $r_h(x)$ so that $r_h(x)=0$ if no neighborhood of $x$ is a Riemannian manifold. 
 Otherwise, we define $r_h(x)$ to be the largest $r>0$ such that there exists a mapping $\Phi:B_r(0^n)\to X$ such that:
\begin{enumerate}
\item $\Phi(0)=x$ with $\Phi$ is a diffeomorphism onto its image.
\vskip1mm

\item $\Delta_g x^\ell = 0$, where $x^\ell$ are the coordinate functions and $\Delta_g$
 is the Laplace Beltrami operator.
\vskip1mm

\item If $g_{ij}=\Phi^*g$ is the pullback metric, then
\begin{align}
||g_{ij}-\delta_{ij}||_{C^0(B_r(0^n))}+r||\partial_k g_{ij}||_{C^0(B_r(0^n))} \leq 10^{-3}\, .
\end{align}
\end{enumerate}
\end{definition}

We call a mapping $\Phi:B_r(0^n)\to X$ as above a harmonic coordinate system.  
Harmonic coordinates have an abundance of good properties when it comes to regularity issues; see the book \cite{Petersen_RiemannianGeometry} for a nice introduction.  In particular, if the Ricci curvature is uniformly bounded then
in harmonic coordinates, the metric, $g_{ij}$  has a priori $C^{1,\alpha}\cap W^{2,q}$ bounds, for all $\alpha<1$ and 
$q<\infty$.  If in addition, there is a bound on $|\nabla \Ric_{M^n}|$, then in harmonic coordinates,
$g_{ij}$ has $C^{2,\alpha}$ bounds,
for all $\alpha <1$.

The primary theorem we wish to review in this subsection is the following:

\begin{theorem}[\cite{Anderson_Einstein}, \cite{ChC1}]\label{t:standard_epsilon_regularity}
There exists $\epsilon(n,\rv)>0$ such that if $M^n$ satisfies $|\Ric_{M^n}|\leq \epsilon$, $\Vol(B_1(p))>\rv>0$, and
\begin{align}
d_{GH}(B_2(p),B_2(0))<\epsilon(n)\, ,
\end{align}
where $0\in \dR^{n-1}\times C(Z)$, then the harmonic radius  $r_h(p)$ satisfies
\begin{align}
r_h(p)\geq 1\, .
\end{align}
If $M^n$ is further assumed to be Einstein, then the regularity scale $r_p$ satisfies $r_p\geq 1$.
\end{theorem}
 
By the results of the previous subsection, it is possible to find balls 
satisfying the above constraint off a subset of Hausdorff codimension $2$.  Moreover,
when combined with the quantitative stratification of \cite{CheegerNaber_Ricci}, 
see also Section \ref{s:qs_estimates}, this 
$\epsilon$-regularity theorem leads to a priori $L^p$ bounds
on the curvature.  The primary result of the present paper can be viewed as Theorem \ref{t:eps_reg}, 
 which states that the conclusions of Theorem \ref{t:standard_epsilon_regularity} continue to hold if
$\dR^{n-1}$ is replaced by
$0\in \dR^{n-3}\times C(Z)$.

\subsection{Examples}\label{ss:examples}
In this subsection, we  indicate some simple examples which play an important role in guiding the 
results of this paper.  

\begin{example} (The Cone Space $\dR^{n-2}\times C(S^1_\beta)$)
\label{sss:Example3_ConeSpace}
 The main result of this paper, Theorem \ref{t:main_codim4},   states that 
$\dR^{n-2}\times C(S^1_\beta)$, with $\beta<2\pi$, is not the 
noncollapsed Gromov-Hausdorff limit of a sequence of manifolds with bounded Ricci curvature.
However, it is clear that this space 
is the Gromov-Hausdorff limit of a a sequence
of noncollapsed manifolds with a uniform lower Ricci curvature bound.  Indeed,
 by rounding off $C(S^1_\beta)$, we see that $\dR^{n-2}\times C(S^1_\beta)$ 
can appear as a noncollapsed limit of manifolds with nonnegative sectional curvature.

 In this example, let us just consider the two dimensional cone  
$C(S^1_\beta)$ with $\beta<2\pi$. Regard $S^1_\beta$ as $0\leq\theta\leq 2\pi$, with the end 
points identified. Then the Laplacian on $S^1_\beta$ is $(\frac{2\pi}{\beta})^2\cdot\frac{\partial^2}{\partial\theta^2}$.
The eigenfunctions are of the form $e^{ik\theta}$, where $k$ is an integer. Written in polar coordinates,
a basis for the bounded
harmonic functions on $C(S^1_\beta)$ is
$\{r^{\frac{2\pi}{\beta}|k|}\cdot e^{ik\theta}\}$.
In particular, we see from this that if $\beta<2\pi$, then $|\nabla ( r^{\frac{2\pi}{\beta}|k|}\cdot e^{ik\theta})|\to 0$ 
 as $r\to 0$.  As a consequence,
{\it every} bounded harmonic function has vanishing gradient at the vertex,
which is a set of positive $(n-2)$-dimensional Hausdorff measure.  By considering examples with more 
vertices, we can construct limit spaces where bounded harmonic functions
 $h$ must have vanishing gradient on bounded subsets sets of arbitrarily large, or even infinite, $(n-2)$-dimensional
Hausdorff measure.  This set can even be taken to be dense.
\end{example}

\begin{example} (The Eguchi-Hanson manifold)
\label{sss:Example1_EH}
The Eguchi-Hanson metric $g$ is a complete Ricci flat metric on the cotangent bundle of $S^2$,
which at infinity,  becomes rapidly asymptotic to the metric cone on $\dR\dP(3)$ or equivalently to 
$\dR^4/\dZ_2$, where $\dZ_2$ acts on $\dR^4$ by $x\to-x$. When the metric $g$ is scaled down
by $g\to r^2g$, with $r\to 0$, one obtains a family of Ricci flat manifolds whose Gromov-Hausdorff
limit is $C(\dR\dP(3))= \dR^4/\dZ_2$.
This is the simplest example which shows that even under the assumption of 
 Ricci  flatness and noncollapsing, Gromov-Hausdorff limit spaces can contain
 codimension 4 singularities.
\end{example}

\begin{example} (Infinitely many topological types in dimension 4)
\label{sss:Example2_ITT}
Let   $T^3$ denote a flat $3$-torus. According to Anderson \cite{Anderson_Hausdorff}, there is  a 
collapsing sequence of manifolds 
$(M^4_j,d_j)\stackrel{d_{GH}}{\longrightarrow} T^3$ 
satisfying
\begin{align}
&\diam(M^4_j)\leq 1\, ,\notag\\
&|\Ric_{M^n_j}|\leq \epsilon_j\to 0\, ,\notag\\
&\Vol(M^4_j)\to 0\, ,\notag\\
&b_2(M^4_j)\to \infty\, ,
\end{align}
where $b_2(M^4_j)$ denotes the second Betti number of $M^4_j$.
In particular, 
Theorem \ref{t:main_finite_diffeo}, the finiteness theorem in dimension $4$,
does not extend to the case in which the lower volume bound is dropped.
\end{example}

\section{Proof of the Transformation Theorem}\label{s:transformation}


In this section we prove the Transformation theorem (Theorem \ref{t:transformation_theorem_intro}) which is the
main technical tool in the proof of the Slicing theorem (Theorem \ref{t:slicing_intro}). As motivation, let us mention the
following.  Given $\epsilon,\eta>0$ and a $\delta(\epsilon, \eta)$-splitting map $u:B_2(p)\to \dR^k$, one can use a weighted maximal function estimate for $|\nabla^2 u|$ to conclude there exists a set $B$ with small $(n-2+\eta)$-content, such  that for each $x\not\in B$ and every $0<r<1$,
 the restriction $u:B_r(x)\to \dR^k$ is an $\epsilon$-splitting map. However,
as we have observed in Example \ref{sss:Example3_ConeSpace}, we cannot take $\eta=0$, since 
$|\nabla u|$ can vanish on a set of large $(n-2)$-content.  For  purposes
of proving the Slicing theorem, this set is too large. 

Suppose instead, that we consider  the collection of balls $B_r(x)$ such that for no lower triangular matrix $A\in GL(n-2)$ 
with positive diagonal entries is $A\circ u$ an $\epsilon$-splitting map on $B_r(x)$.  Though we cannot show that this set has small $(n-2)$-content, we will prove that its image under $u$ has small $(n-2)$-dimensional measure.  This will be what is required for the Slicing Theorem. 

For the case of a single function, $k=1$, the basic
 idea can be explained as follows. In order to obtain an $\epsilon$-splitting
function on $B_r(x)$, it is not necessary that the Hessian of $u$ is small and the gradient is close to $1$.  Rather, we need only that the Hessian of $u$ is small relative to the gradient.  
That is, for $\epsilon>0$, $0<r\leq 1$, consider the condition
\begin{equation}\label{e:relative}
r\fint_{B_{2r}(x)}|\nabla^2u|\leq \delta(\epsilon)\cdot\fint_{B_{2r}(x)}|\nabla u|\, .
\end{equation}
Now if $\fint_{B_{2r}(x)}|\nabla u|$ is very small,
 then the restricted map $u:B_r(x)\to \dR$ will not define a splitting map. 
 However, if \eqref{e:relative} holds we may simply rescale $u$ so that $\fint_{B_{2r}(x)}|\nabla u|=1$,
 in which case, arguments as in the proof of Lemma \ref{l:harmonic_splitting_intro}
 tell us that after such a rescaling, $u:B_r(x)\to \dR$ becomes an $\epsilon$-splitting
 map.  

To control the collection of balls which do not satisfy
the inequality \eqref{e:relative}, we start with what is essentially (\ref{kato_1}):
$$
\label{e:svf1}
\fint_{B_{3/2}(p)} \frac{|\nabla^2 u|^2}{|\nabla u|} < \delta^2\, .
$$
By arguing as in subsection \ref{ss:pst},
this  enables us to control the set of balls $B_{2r}(x)$ which do not satisfy
\begin{equation}
\label{e:svf2}
r^2\fint_{B_{2r}(x)} \frac{|\nabla^2 u|^2}{|\nabla u|} 
< \delta\fint_{B_{2r}(x)} |\nabla u|\, ,
\end{equation}
and 
in particular, to show  that the image under $u$ of this collection of balls
has small $(n-2)$-dimensional measure.
On the other hand, (\ref{e:svf2}) implies (\ref{e:relative}), since 
\begin{equation}\label{e:svf3}
r\fint_{B_{2r}(x)}|\nabla^2 u| \leq \Big(r^2\fint_{B_{2r}(x)}\frac{|\nabla^2 u|^2}{|\nabla u|}\Big)^{1/2}
\cdot \Big(\fint_{B_{2r}(x)}|\nabla u|\Big)^{1/2}\leq \delta^{1/2}\fint_{B_{2r}(x)}|\nabla u|\, .
\end{equation}

For the case $k>1$ serious new issues arise.
 For one thing, even
if on some ball $B_r(x)$ the individual gradients, $\nabla u^1,\ldots,\nabla u^{n-2}$,
satisfy (\ref{e:relative}) and  we then  normalize them to have $L^2$
norm $1$, it still might be  the case that in the $L^2$ sense, this normalized collection looks close
to being linearly dependent.  Then $u$ would still be far from defining an $\epsilon$-splitting map.  This issue is related to the fact that for $\ell>1$ the distributional Laplacian $\Delta|\omega^\ell|$ may have a singular part.  Additionally, for $k>1$ we are unable to obtain a precise analog of (\ref{kato_1}), which was the tool
for handling the case $k=1$.  Instead, we have to proceed on the basis of (\ref{e:kato2}), the bound on the normalized mass of the distributional Laplacian $\Delta |\omega^\ell|$. These points make the proof of the Transformation theorem in the general case substantially more difficult.

\subsection{Higher Order Estimates}\label{ss:higher_order_estimates}
We begin by recalling the existence of a good cutoff function.
According to \cite{ChC1} if $\Ric_{M^n}\geq -\delta$, then for any $B_r(x)\subset M^n$ with $0<r\leq 1$ there exists a cutoff 
function, with $0\leq \varphi\leq 1$, such that 
\begin{align}\label{cutoff1}
\varphi(x) &\equiv1 \text{ if }x\in B_{9r/5}(x)\, ,\notag\\
{\rm supp}\, \varphi&\subset B_{2r}(x)\, ,
\end{align}
and such that 
\begin{align}\label{cutoff2}
r|\nabla\varphi|&\leq C(n)\, ,\notag\\
r^2|\Delta\varphi|&\leq C(n) \, . 
\end{align}

In preparation for proving part (2) of Theorem \ref{t:hessian_estimate_intro},
we state a general lemma on distributional Laplacians.  

Let $w$ be a smooth section of a Riemannian vector bundle with orthogonal
connection over
$B_2(p)$. Let $\Delta w$ denote the rough Laplacian of $w$. 
Note that $|w|$ is a Lipschitz function 
which is smooth off of the set $Z_{|w|}=:\{x\, |\, |w|(x)=0\}$.
We put
\begin{align}
U_r=\{x\, :\, |w|(x)\leq r\}\, .
\end{align}

\begin{lemma}
\label{l:bochner_general}
The distributional Laplacian $\Delta |w|$ is a locally finite Borel measure $\mu=\mu_{ac}+\mu_{sing}$. The measure $\mu$ is absolutely 
continuous on $B_2(p)\setminus Z_{|w|}$, with density
\begin{align}\label{e:dist_laplacian}
\mu_{ac}
=\frac{\langle \Delta w,w\rangle}{|w|}
+\frac{|\nabla w|^2-|\nabla|w||^2}{|w|}\, .
\end{align}
The singular part $\mu_{sing}$ is a nonnegative locally finite Borel measure supported on $B_2(p)
\cap\partial Z_{|w|}$. There exists $r_j\to 0$, such that for any nonnegative continuous function $\varphi$, with ${\rm supp}\, \varphi\subset B_2(p)$, we have
\begin{align}
\label{e:sing}
\mu_{sing}(\varphi)= \lim_{r_i\to 0}\int_{B_2(p)\cap\partial U_{r_i}} \varphi \cdot |\nabla|w||
 \geq 0\, .
\end{align}

If in addition $\Ric_{M^n}\geq -(n-1)\kappa$,
then on each ball $B_{2-s}(p)$,  the distributional Laplacian $\Delta |w|$ satisfies
the normalized mass bound
\begin{align}
\label{e:mass}
\fint_{B_{2-s}(p)}|\Delta |w|\,|\leq C(n,\kappa,s)\cdot \inf_c \fint_{B_{2-s/2}(p)}||w|-c| -2\fint_{B_{2-s/2}(p)}\frac{\langle \Delta w,w\rangle_{-}}{|w|}\, ,
\end{align}
where  $\frac{\langle \Delta w,w\rangle_{-}}{|w|}
=:\min(0,\frac{\langle \Delta w,w\rangle}{|w|})$
\end{lemma}
\begin{proof}
The computation
of the absolutely continuous part (\ref{e:dist_laplacian}) on
$B_2(p)\setminus Z_{|w|}$ is standard.  

Before continuing, let us mention the following  technical point.  Fix $2>s>0$.  Since
on $B_{2-s}(p)$, $|w|$ is Lipschtiz and
\begin{align}
\label{e:Kato}
\nabla  |w|=\frac{\langle\nabla w,w\rangle}{|w|}
\qquad ({\rm on}\,\, B_{2-s}(p)\setminus Z_{|w|})
\end{align}
has uniformly
bounded norm, $|\nabla |w||\leq |\nabla w|$, we have by the coarea formula,
 that as $r\to 0$
\begin{align}
\label{tozero}
 o(r)& =\int_{B_{2-s}(p)\cap (U_r\setminus U_{r/2})}|\nabla |w|\, |\notag\\
&=\int^r_{r/2} \mathcal H^{n-1}(B_{2-s}\cap\partial U_t)\, dt\, ,
\end{align}
where $\mathcal H^{n-1}$ denotes $(n-1)$-dimensional Hausdorff measure.
By combining this with Sard's theorem, it follows in particular that there exist 
decreasing sequences $r_i\searrow 0$, such that for any $2>s>0$, 
we have that $B_{2-s}(p)\cap\partial U_{r_i}$
is smooth and 
 \begin{align}
\label{e:sing}
\lim_{r_i\to 0}r_i\cdot\mathcal H^{n-1}(B_{2-s}(p)\cap\partial U_{r_i})=0\, .
\end{align}


Let $\varphi\geq 0$ denote a smooth function with 
${\rm supp}\,\varphi\subset B_2(p)$ and let $r_i\searrow 0$ be as in (\ref{e:sing}).  Then for any constant $c$, we have
\begin{align}\label{e:dist_bochner_1}
\int_{B_2(p)}\Delta\varphi \cdot \big(|w|-c\big) &=\lim_{r_i\to 0}\int_{B_2(p)\setminus U_{r_i}}\Delta\varphi \cdot |w|\notag\\
&= \lim_{r_i\to 0}
\int_{B_2(p)\setminus U_{r_i}} \varphi \cdot\Delta |w|+\lim_{r_i\to 0}
\int_{B_2(p)\cap\partial U_{r_i}}\varphi \cdot N(|w|)
-\lim_{r_i\to 0}\int_{B_2(p)\cap\partial U_{r_i}}N(\varphi) \cdot r_i \notag\\
&=\int_{B_2(p)\setminus Z_{|w|}}\varphi \cdot \Delta |w|+\lim_{r_i\to 0}\int_{B_R(p)\cap\partial U_{r_i}}\varphi\cdot N(|w|)
\notag\\
&=\int_{B_2(p)\setminus Z_{|w|}}\varphi\cdot\langle \Delta w,\frac{w}{|w|}\rangle
+\lim_{r_i\to 0}\int_{B_2(p)\setminus U_{r_i}}\varphi \cdot\frac{|\nabla w|^2-|\nabla|w||^2}{|w|}+\lim_{r_i\to 0}\int_{B_2(p)\cap\partial U_{r_i}}\varphi\cdot |\nabla|w||\, .
\end{align}
where the third term on the right-hand side of the second line of (\ref{e:dist_bochner_1})
vanishes because of (\ref{e:sing}). 
Note that since $w$ is smooth and
the second and third integrands on the last line are nonnegative, it follows that
all three limits on the  last line exist. 

For fixed $i$ each term on the last line above defines
a Borel measure.  
To see that the weak limits of these measures define Borel measures which
satisfy the mass bound in (\ref{e:mass}), we assume $\Ric_{M^n}\geq-(n-1)\kappa$ and
 choose $\varphi$ in (\ref{e:dist_bochner_1})  to be a cutoff
function as in \cite{ChC1} with $\varphi\equiv 1$ on $B_{2-s}(p)$, ${\rm supp}\, \varphi\subset B_{2-s/2}(p)$ and
$s|\nabla \varphi|,\, s^2|\Delta\varphi|\leq c(n,\kappa,s)$.
From the elementary fact that $a-b\geq 0$ implies $|a|\leq a+2b_{-}$, where $b_{-}=:\min(0,b)$, we get the mass bound
\begin{align}
\label{e:mb}
\min_c\int_{B_{2-s/2}(p)}|\Delta\varphi| \cdot |(|w|-c)| -
2\int_{B_{2-s/2}(p)}\varphi\cdot\frac{\langle \Delta w,w\rangle_{-}}{|w|}
\geq \int_{B_{2-s}(p)\setminus Z_{|w|} } |\Delta |w||
+\lim_{r_i\to 0}\int_{B_{2-s}(p)\cap\partial U_{r_i}}
|\nabla|w||
\end{align}
 which suffices to complete the proof. 

\end{proof}

\begin{remark} 
Note that for the proof of the mass bound in Lemma \ref{l:bochner_general}, on which the mass bound in Theorem 
\ref{t:hessian_estimate_intro} is based, it is crucial that the 
singular term has the correct sign: $\int_{B_2(p)\cap\partial U_{r_i}}\varphi \cdot |\nabla|\omega||\geq 0$, where
$\varphi\geq 0$.
\end{remark}

Now we can finish the proof of Theorem \ref{t:hessian_estimate_intro}.  
First, we recall the statement:

{\it For every $\epsilon>0$ there exists $\delta(n,\epsilon)>0$ such that if 
$\Ric_{M^n}\geq -\delta$, with $u:B_2(p)\to \dR^k$ a 
$\delta$-splitting map, 
then the following hold:}
\begin{enumerate}
\item {\it There exists $\alpha(n)>0$ such that for each $1\leq a\leq k$,}
\begin{align}
\label{e:higher_1}
\fint_{B_{3/2}(p)} \frac{|\nabla^2 u^a|^2}{|\nabla u^a|^{1+\alpha}} < \epsilon\, .
\end{align}

\item
{\it Let $\omega^\ell\equiv du^1\wedge\cdots\wedge du^\ell$, $1\leq\ell\leq k$.  The Laplacians $\Delta |\omega^\ell|$, taken in the distributional sense, are Borel measures with singular part a locally finite nonnegative Borel measure supported on $\partial Z_{|\omega^\ell |}$. For $\ell=1$, the singular part vanishes.  The normalized mass of $\Delta|\omega^\ell|$  satisfies}
\begin{equation}\label{kato2}
\fint_{B_{3/2}(p)} \big|\Delta |\omega^\ell|\big| < \epsilon\, .
\end{equation}
\end{enumerate}

\begin{proof}[Proof of Theorem \ref{t:hessian_estimate_intro}.]
We begin by proving (1).\footnote{
We remind the reader that in our subsequent applications, we encounter
only the case $\alpha=0$. Moreover, in view of Lemma \ref{l:bochner_general}, our subsequent
arguments would go through even without knowing
that the singular part of $\Delta|\nabla u|$ is absent. However, for the sake of completeness, we start
by considering all $0\leq \alpha<1$ and then specialize to the case $0\leq \alpha<\frac{1}{n-1}$, in which
we can give an effective estimate.}

The main observation is that for all $0\leq \alpha<1$ we have that the distributional Laplacian $\Delta |\nabla u|^{1-\alpha}$ satisfies
\begin{align}\label{e:dist_bochner_k=1}
\Delta |\nabla u|^{1-\alpha} = (1-\alpha)\frac{\Big(|\nabla^2 u|^2-(1+\alpha)|\nabla|\nabla u||^2 
+ \Ric(\nabla u,\nabla u)\Big)}{|\nabla u|^{1+\alpha}}\, .
\end{align}
In particular, unlike for $\omega^\ell$ with $\ell>1$, there is no possibility of a
singular contribution.  The proof
of this is similar to arguments in \cite{Dong_Nodal}, but for the sake of convenience,
we will outline it here.  There are two key facts which play a role in the vanishing of the singular part of $\Delta |\nabla u|$:
\begin{enumerate}
\item[(a)] The critical set $Z_{|\nabla u|}$ has Hausdorff dimension $\leq n-2$. 
\vskip2mm

\item[(b)]  $u$ vanishes to finite order at each point of $x\in Z_{|\nabla u|}$.  That is,  $u$ has a leading order Taylor expansion at $x$ of degree $k_x\geq 1$, with $k_x$ uniformly bounded on compact subsets. 
\end{enumerate}

The previous two properties are standard. They follow by working in a sufficiently smooth coordinate chart and using the monotonicity of the frequency, see \cite{HanLin_Book}, \cite{ChNaVa}, \cite{NaVa_CriticalSets}.  Note that  the frequency is not monotone until one restricts to a sufficiently
 regular coordinate chart.  Since in our situation, there is no a priori estimate on the size of such a coordinate chart,
 although there is finite vanishing order at each point, 
there is no {\it a priori} estimate on the size of the vanishing order.

Now let us finish outlining the proof of \eqref{e:dist_bochner_k=1}.  Let $\varphi$ be a smooth function with support contained in $B_2(p)$.  Put $S_r(\, \cdot\,) =\partial T_r(\, \cdot\,)$.  Now $Z_{|\nabla u|}$ is a closed set which satisfies the Hausdorff dimension estimate of (a).  While this is sufficient, to simplify the argument we use \cite{ChNaVa}, \cite{NaVa_CriticalSets}, to see that the following Minkowski estimate holds:
\begin{align}\label{e:nodal_volume_est}
\Vol(S_{r}(Z_{|\nabla u|}\cap \text{supp}\, \varphi))<C r\, .
\end{align}
Then we compute
 \begin{align}\label{e:bochner0}
\int_{B_2(p)} \Delta\varphi\cdot |\nabla u|^{1-\alpha} &= 
\lim_{r\to 0}
\int_{B_2(p)\setminus T_{r}(Z_{|\nabla u|})} \Delta\varphi\cdot |\nabla u|^{1-\alpha} 
\notag\\
& = 
-\lim_{r\to 0}
\int_{B_2(p)\setminus T_{r}(Z_{|\nabla u|})}
\langle\nabla\varphi,\nabla|\nabla u|^{1-\alpha}\rangle 
+\lim_{r\to 0}\int_{S_{r}(Z_{|\nabla u|})}
N(\varphi)\cdot |\nabla u|^{1-\alpha}\, ,
\notag\\
& = -\lim_{r\to 0}\int_{B_2(p)\setminus T_{r}(Z_{|\nabla u|})}
\langle\nabla\varphi,\nabla|\nabla u|^{1-\alpha}\rangle \, ,
\notag\\
& = \lim_{r\to 0} \int_{B_2(p)\setminus T_{r}(Z_{|\nabla u|})} \varphi\cdot\Delta|\nabla u|^{1-\alpha} -\frac{(1-\alpha)}{2} \lim_{r\to 0} \int_{B_2(p)\cap S_{r}(Z_{|\nabla u|})}\varphi\cdot
\frac{N(|\nabla u|^2)}{|\nabla u|^{1+\alpha}}\, , \notag\\
& = (1-\alpha)\int_{B_2(p)\setminus Z_{|\nabla u|}} \varphi\cdot \left(\frac{\Big(|\nabla^2 u|^2-(1+\alpha)|\nabla|\nabla u||^2 
+ \Ric(\nabla u,\nabla u)\Big)}{|\nabla u|^{1+\alpha}}\right)\, ,
\end{align}
where in dropping the last boundary term, we have used
\eqref{e:nodal_volume_est} and finite vanishing order (b)  to estimate
\begin{align}
\lim_{r\to 0} \int_{B_2(p)\cap S_{r}(Z_{|\nabla u|})}\varphi\cdot
\frac{N(|\nabla u|^2)}{|\nabla u|^{1+\alpha}} &= 2\lim_{r\to 0} \int_{B_2(p)\cap S_{r}(Z_{|\nabla u|})}\varphi\cdot
\frac{|\nabla|\nabla u||}{|\nabla u|^{\alpha}}&\leq C\lim_{r\to 0} r^{1} r_i^{-1+(1-\alpha)}=C\lim_{r\to 0} r^{1-\alpha}\to 0\, .
\end{align}
For a related argument, see \cite{Dong_Nodal}.

To finish the proof, observe that since ${\rm trace}(\nabla^2 u) = \Delta u = 0$, 
it follows if $\lambda_1,\ldots,\lambda_n$ are the eigenvalues of $\nabla^2u$ then $\sum \lambda_i = 0$.  
In particular, if $\lambda_n$ is the largest eigenvalue then by the Schwarz inequality,
\begin{align}
\lambda_1^2+\cdots+\lambda^2_n \geq \frac{1}{n-1}(\lambda_1+\cdots+\lambda_{n-1})^2 + \lambda_n^2 
\geq \frac{n}{n-1}\lambda_n^2\, .
\end{align}
This leads to the improved Kato inequality
\begin{align}
|\nabla^2 u(\rv)|^2\leq \big(1-\frac{1}{n}\big)|\nabla^2 u|^2\, ,
\end{align}
where $\rv$ is any vector with $|\rv|=1$.

Thus, if rewrite 
\begin{align}
|\nabla |\nabla u|| = \Big|\nabla^2 u\left(\frac{\nabla u}{|\nabla u|}\right)\Big|\, ,
\end{align}
and apply the improved Kato inequality 
and $\Ric_{M^n}\geq -\delta$, we get 
\begin{align}
\label{e:katoineq}
\Delta |\nabla u|^{1-\alpha}
\geq \frac{1-(n-1)\alpha}{n}\frac{|\nabla^2 u|^2}{|\nabla u|^{1+\alpha}}-(1-\alpha)\delta |\nabla u|^{1-\alpha}\, ,
\end{align}
which gives nontrivial information  for any  $\alpha<\frac{1}{n-1}$, which we now assume.  
Namely, we get the distributional inequality
\begin{align}\label{e:bochner1}
\frac{|\nabla^2 u|^2}{|\nabla u|^{1+\alpha}}\leq C(n,\alpha)\Big(\Delta|\nabla u|^{1-\alpha}+\delta|\nabla u|^{1-\alpha}\Big)\, .
\end{align}

Finally let $\varphi\geq 0$ be a smooth function as in (\ref{cutoff1}), (\ref{cutoff2}), with ${\rm supp}\,\varphi\subset B_2(p)$,
$|\nabla\varphi|,|\Delta\varphi|\leq C(n)$.  

By multiplying both sides of \eqref{e:bochner1} by $\varphi$ and integrating we obtain
\begin{align}
\int_{B_2(p)} \varphi \frac{|\nabla^2 u|^2}{|\nabla u|^{1+\alpha}}
&\leq C(n,\alpha)\int_{B_2(p)}\Big(\varphi \Delta|\nabla u|^{1-\alpha}+\varphi\delta|\nabla u|^{1-\alpha}\Big)\, , \notag\\
&\leq C(n,\alpha)\int _{B_2(p)}\Delta\varphi\Big(|\nabla u|^{1-\alpha}-\fint_{B_2(p)}|\nabla u|^{1-\alpha}\Big) 
+ C(n,\alpha)\delta\int_{B_2(p)}|\nabla u|^{1-\alpha}\, ,\notag\\
&\leq C(n,\alpha)\int_{B_2(p)} \Big||\nabla u|^{1-\alpha}-\fint_{B_2(p)}|\nabla u|^{1-\alpha}\Big| 
+ C(n,\alpha)\delta\int_{B_2(p)}|\nabla u|^{1-\alpha}\, ,
\end{align}

Now we use that if $u$ is a harmonic $\delta$-splitting map then $|\nabla u|^{1-\alpha}$ is
 bounded and $\fint_{B_2(p)} \Big||\nabla u|^{1-\alpha}-\fint_{B_2(p)}|\nabla u|^{1-\alpha}\Big|$ is small.
  In particular, for $\delta$ sufficiently small, we have 
\begin{align}
\fint_{B_{3/2}(p)}\frac{|\nabla^2 u|^2}{|\nabla u|^{1+\alpha}}
&\leq C(n)\fint_{B_2(p)} \varphi \frac{|\nabla^2 u|^2}{|\nabla u|^{1+\alpha}}\notag\\
&\leq C(n,\alpha)\fint_{B_2(p)} \Big||\nabla u|^{1-\alpha}-\fint_{B_2(p)}|\nabla u|^{1-\alpha}\Big|
 + C(n,\alpha)\delta\fint_{B_2(p)}|\nabla u|^{1-\alpha}\leq \epsilon\, ,
\end{align}
which proves (\ref{e:higher_1}).

\begin{remark}
\label{r:cutoff}
 For $0\leq \alpha < \frac{1}{n-1}$, there is another way of seeing that \eqref{e:dist_bochner_k=1} holds in the distributional sense, 
which uses only the fact that $Z_{|\nabla u|}\cap{\rm supp}\, \varphi$ has  Hausdorff dimension $\leq n-\frac{n}{n-1}$ and  the improved Kato inequality.
From the Hausdorff dimension bound,
 it follows that there is a nondecreasing sequence of cutoff functions $\psi_i$ converging
pointwise to $1$ on $B_2(p)\setminus (Z_{|\nabla u|}\cap {\rm supp}\, \varphi)$ each of which vanishes in a neighborhood of
$Z_{|\nabla u|}$ and such that  $|\nabla \psi_i|_{L^q}\to 0$, for all $q<\frac{n}{n-1}$. For 
the case $\alpha=0$, the claim follows by applying the divergence
theorem to the vector fields $\psi_i\nabla |\nabla u|$, noting that 
$|\nabla |\nabla u|\, |\in L^\infty\subset L^{q'}$
and using H\"older's inequality.  For $0<\alpha<\frac{1}{n-1}$, one uses an iterative version of the above argument. 
 For additional details on this instance of the divergence
theorem, see  e.g. Section 2 of
\cite{Cheeger}.
\end{remark}

Next we prove  (2).  The vanishing of the singular part for $\ell=1$ is contained in
part (1).

By invoking Lemma \ref{l:bochner_general},  all that remains
is to bound from below, the term,
 $\langle\Delta \omega^\ell,\frac{\omega^\ell}{|\omega^\ell|}\rangle_{-}$, in (\ref{kato2}).
On  $B_2(p)\setminus Z_{|\omega^\ell|}$, by Bochner's formula, we have
\begin{align}
\label{e:lower}
\Delta \omega^\ell = \sum_a du^{1}\wedge\cdots \Ric(du^{a})\wedge\cdots\wedge du^\ell 
+ 2\sum_{a\ne b,j} du^{1}\wedge \nabla_j(du^{a})\wedge\cdots\wedge 
\nabla_j(du^{b})\wedge\cdots\wedge du^\ell\, ,
\end{align}
from which it follows that
\begin{align}
\label{e:lower1}
\langle\Delta \omega^\ell,\frac{\omega^\ell}{|\omega^\ell|}\rangle_{-}
\geq -C(n)\left(\delta
+\sum_a|\nabla u^j|^2\right)|\omega^\ell|\, .
\end{align}
Since $u$ is a $\delta$-splitting map, by using (\ref{e:mass})
of Lemma \ref{l:bochner_general}, this suffices to complete the proof.
\end{proof}

\begin{remark} A simple example of a harmonic map $u:\R^n\to \R^k$, for which the distributional Laplacian
 $\Delta |\omega|$ has a singular part with positive mass 
 is furnished by the $2$-form $dx\wedge d(x^2-y^2)=-2y\cdot dx\wedge dy$
 (which can be thought of as depending on $(n-2)$ additional variables).
 We do not know whether an $\epsilon$-splitting map with small 
 $\epsilon$ can furnish such an example, though this seems within reason.
\end{remark}

\subsection{Proof of the Transformation theorem}\label{t:transformation}

In this subsection, we prove the Transformation theorem (Theorem \ref{t:transformation_theorem_intro})
which
constitutes the
technical heart of the Slicing theorem (Theorem \ref{t:slicing_intro}). 
 We will assume for notational simplicity that $M^n$ is complete, but it is an
 easy exercise to show that this may be weakened to the local assumption that $B_4(p)$ has compact closure in $M^n$.
First we recall the definition of the  singular scale:

{\it Let $u:B_2(p)\to\dR^k$ be a harmonic function.  For $\delta>0$ let us define for $x\in B_1(p)$ the  singular 
scale $s^\delta_x\geq 0$ as the infimum of all radii $s$, such that for all $s< r<\frac{1}{2}$ and all $1\leq \ell\leq k$ 
we have the estimate
$$
r^2\fint_{B_r(x)} |\Delta|\omega^\ell|| \leq \delta \fint_{B_r(x)}|\omega^\ell|\, ,
$$
where $\omega^\ell = du^{1}\wedge\cdots\wedge du^{\ell}$.  \\
}


Next recall that Theorem \ref{t:transformation_theorem_intro} states the following.

{\it For every $\epsilon>0$ there exists 
$\delta=\delta(n,\epsilon)>0$ such that if 
$\Ric_{M^n}\geq -\delta$
 and $u:B_2(p)\to \dR^k$  is a harmonic $\delta$-splitting map,
 then for each $x\in B_1(p)$ and $r\geq s^\delta_x$ there exists a lower triangular matrix $A=A(x,r)$ with positive diagonal entries such
 that $A\circ u:B_r(x)\to \dR^k$ is a harmonic $\epsilon$-splitting map.
}

\begin{proof}[Proof of Theorem \ref{t:transformation_theorem_intro}.]
 The strategy will be a proof by induction. Thus, 
 we will begin with the simplest case of $k=1$.  The following is a slightly more general form of the statement 
we wish to prove.
 
\begin{lemma}\label{l:transformation_theorem_k1}
Let $u:B_{2r}(x)\to \dR$ be a harmonic function with $r\leq 1$.  
Then for every $\epsilon>0$ there exists $\delta(n,\epsilon)>0$ such that if $\Ric_{M^n}\geq -(n-1)\delta$ and
\begin{align}\label{e:k1:1}
r^2\fint_{B_{2r}(x)} |\Delta|\nabla u|| \leq \delta \fint_{B_{2r}(x)}|\nabla u|\, ,
\end{align}
then for $A=\Big(\fint_{B_{r}(x)}|\nabla u|\Big)^{-1}>0$ we have that $A\circ u:B_{r}(x)\to \dR$ is an $\epsilon$-splitting map.
\end{lemma}

As in 
the proof of Theorem \ref{t:hessian_estimate_intro}, 
 the fact
that $u$ is harmonic leads to the improved Kato inequality,
$|\nabla|\nabla u^a|\,|^2\leq \frac{n-1}{n}|\nabla^2u^a|^2$,
from which we can compute
\begin{align}
\Delta |\nabla u| \geq \frac{1}{n}\frac{|\nabla^2 u|^2}{|\nabla u|}-(n-1)
\delta
|\nabla u|\, .
\end{align}
In particular, the estimate \eqref{e:k1:1} gives rise to the estimate
\begin{align}
r^2\fint_{B_{2r}(x)}\frac{|\nabla^2 u|^2}{|\nabla u|}\leq C(n)\delta\fint_{B_{2r}(x)}|\nabla u|\, ,
\end{align}
from which, as previously noted (see (\ref{e:svf2}), (\ref{e:svf3}) ) we get
\begin{align}
\label{e:fundamental}
r\fint_{B_{2r}(x)}|\nabla^2 u| \leq \Big(r^2\fint_{B_{2r}(x)}\frac{|\nabla^2 u|^2}{|\nabla u|}\Big)^{1/2}\cdot
\Big(\fint_{B_{2r}(x)}|\nabla u|\Big)^{1/2}\leq C\delta^{1/2}\fint_{B_{2r}(x)}|\nabla u|\, .
\end{align}

Let us put $v= \Big(\fint_{B_{2r}(x)}|\nabla u|\Big)^{-1}u$, so that $\fint_{B_{2r}(x)}|\nabla v|=1$.
  The lower Ricci bound implies that a Poincar\'e inequality holds. When combined with the last inequality 
this gives
\begin{align}
\fint_{B_{2r}(x)}\big||\nabla v|-1\big|\leq C(n)\delta^{1/2}\, .
\end{align}
By using the doubling property, we have after possible increasing $C(n)$, that for every $y\in B_{3r/2}(x)$,
\begin{align}\label{e:k1:2}
\fint_{B_{r/2}(y)}\big||\nabla v|-1\big|\leq C\delta^{1/2}\, .
\end{align}
In particular,    
\begin{align}
1-C\delta^{1/2}\leq\frac{\fint_{B_{2r}(x)}|\nabla v|}{\fint_{B_r(x)}|\nabla v|}\leq 1+C\delta^{1/2}\, .
\end{align}
Let us observe that we may also use the Ricci lower bound,
the Poincar\'e inequality, and the Harnack inequality to conclude the weak gradient estimate $\sup_{ B_{\frac{7}{4}r}(x)} |\nabla v|\leq C(n)$.  Now, if we can show that, for $\delta$ sufficiently
 small, the map $v:B_r(x)\to \dR$ is an $\epsilon/2$-splitting, then
 for $k=1$, the proof will be complete

Now as in \cite{ChC1}, let $\varphi\geq 0$ be a cutoff function satisfying $\varphi(y)=1$ if $y\in B_{5r/3}(x)$ with
 $\varphi(y)\equiv 0$ if $y\not\in B_{2r}(x)$, and such that $r|\nabla\varphi|, r^2|\Delta\varphi|\leq C(n)$.  
Let $\rho_t(y,dz)$ be the heat kernel on $M^n$.  Consider for $y\in B_{3r/2}(x)$ the one parameter family
\begin{align}
\label{e:sharp1}
\int \big(|\nabla v|-1\big)\varphi\,\rho_t(y,dz)\, .
\end{align}
Note that for all $y\in A(0,r)$, $z\in A(3r/2,2r)$ and $t\in [0,r^2]$,
we have $|\rho_t(y,dz)|<C(n)\Vol(B_{\sqrt{t}}(x))^{-1}$; see \eqref{e:hk1}.
It follows that for $t\in[0,r^2]$, we have
\begin{align}
\frac{d}{dt}\int_{B_{2r}(x)}
 \Big(|\nabla v|-1\Big)\varphi\,\rho_t(y,dz) 
&\geq  \int_{B_{2r}(x)}\left(\left( \frac{|\nabla^2 v|^2-|\nabla|\nabla v||^2}{|\nabla v|}-(n-1)\delta^2|\nabla v|  \right)\varphi
+2\langle \nabla |\nabla v|,\nabla\varphi\rangle + (|\nabla v|-1)\Delta\varphi\right)\rho_t(y,dz)\, ,\notag\\
&\geq -C(n)\delta^{2}
 -C(n)\int_{A(3r/2,2r)} \Big(r^{-1}|\nabla^2 v|+r^{-2}\big||\nabla v|-1\big|\Big)\rho_t(y,dz)\, ,\notag\\
&\geq -C\delta^{1/2}r^{-2}\, .
\end{align}
 Integrating this yields
\begin{align}
(|\nabla v|(y) -1)\leq C\delta^{1/2}+\int \Big(|\nabla v|-1\big)\varphi\,\rho_{r^2}(y,dz)
\leq C\delta^{1/2}+C\fint_{B_{2r}(x)} \big||\nabla v|-1\big|\leq C\delta^{1/2}\, .
\end{align}
In particular we have
\begin{align}
\label{e:sharp2}
\sup_{B_{3r/2}(x)}|\nabla v|\leq 1+C\delta^{1/2}\, .
\end{align}
Combining this with the integral estimate \eqref{e:k1:2} we get
\begin{align}\label{e:sharp3}
\fint_{B_{3r/2}(x)}\big||\nabla v|^2-1\big|\leq C\delta^{1/2}\, .
\end{align}

Now using the Bochner formula
\begin{align}
\Delta|\nabla v|^2 = 2|\nabla^2 v|^2 +2\Ric(\nabla v,\nabla v)\geq 2|\nabla^2 v|^2-C\delta^2|\nabla v|^2\, ,
\end{align}
 we can estimate
\begin{align}
\fint_{B_r(x)}|\nabla^2 v|^2&\leq C(n)\fint_{B_{3r/2}(x)}\varphi|\nabla^2 v|^2\\
&\leq C\fint_{B_{3r/2}(x)}\varphi\Big(\Delta(|\nabla v|^2-1) +
\delta
|\nabla v|^2\Big)\notag\\
&\leq C\fint_{B_{3r/2}(x)}|\Delta\varphi|\big||\nabla v|^2-1\big| + C
\delta
\fint_{B_{3r/2}(x)}|\nabla v|^2\, ,\notag\\
&\leq Cr^{-2}\delta^{1/2}\, .
\end{align}
Hence, for $\delta(n,\epsilon)$ sufficiently small, we have that $v$ is an $\epsilon/2$-splitting, 
which as previously remarked, proves the theorem for the case $k=1$.


We now turn to the proof of Theorem \ref{t:transformation_theorem_intro}, 
which will proceed by induction.


Assume the Theorem has been proved for some $k-1\geq 1$. 
We will prove the result for $k$ by arguing by contradiction.  

Thus, we can suppose that for some $\epsilon>0$ the result is false.  There is no harm in assuming 
$0<\epsilon\leq \epsilon(n)$ is sufficiently small, which we will do from time to time.  Then, for some $\delta_j\to 0$, we can find a sequence of spaces 
$(M^n_j,g_j,p_j)$ with $\Ric_{M^n_j}\geq -\delta_j$, 
and mappings $u_j:B_2(p_j)\to \dR^k$,
which are $\delta_j$-splitting mappings, 
for which there exists $x_j\in B_1(p_j)$ and
 radii $r_j\geq s^{\delta_j}(x_j)$, such that there is no lower triangular matrix 
 $A$ with positive diagonal entries, such that $A\circ u:B_{r_j}(x_j)\to \dR^k$ 
is an $\epsilon$-splitting map.  Without loss of generality, we can assume $r_j$ is the supremum of those radii for which 
there is no such matrix.  In particular, there exists such a matrix $A_j$ corresponding to the radius $2r_j$. 
Observe that $r_j\to 0$. Indeed, we can see this just by using the identity map $A= I$, since 
$\delta_j\to 0$ and $u:B_2(p)\to \dR^2$ is a 
$\delta_j$-splitting map.

Now, set $v_j=A_j \circ\big(u_j - u_j(x_j)\big)$
and 
consider the rescaled spaces $(M^n_j,g'_j,x_j)$ with $g'_j\equiv r_j^{-2}g$.
 Thus, $v_j
 :B_{2r^{-1}_j}(x_j)\to \dR^k$ 
 is a harmonic 
function on this space.  We have normalized so that $v(x_j)=0$. As before,
for all $2\leq r\leq 2r_j^{-1}$, there is a lower triangular matrix $A_r$
with positive entries on the diagonal, such that  $A_r\circ v_j:B_r(x_j)\to R^k$ is an $\epsilon$-splitting map
and with our current normalization,  $A_2=I$, the identity map.



\noindent
{\bf Note:} Throughout the remainder of the argument, when there is no danger of
confusion, for ease of notation, we will sometimes omit the subscript $j$ from various quantities 
including $v$ and $A$, which in
actuality depend on $j$. For example, we omit the subcript $j$ from the matrices $A_r, \, A_{2r}$
in Claim 1 below.

We will now break the proof  into a series of claims.\\

{\bf Claim 1:} For each $2\leq r\leq 2r_j^{-1}$ we have  
\begin{equation}
\label{e:c1}
(1-C(n)\epsilon)A_{2r}\leq A_{r}\leq (1+C(n)\epsilon)A_{2r}\,  .
\end{equation}

Since $A_{2r}\circ v:B_{2r}(x_j)\to \dR^k$ is an $\epsilon$-splitting map, 
we have 
\begin{align}
\fint_{B_{2r}(x_j)}\big|\langle\nabla(A_{2r}\circ v)^a,\nabla(A_{2r}\circ v)^b\rangle - \delta_{ab}\big|<\epsilon\, ,
\end{align}
and thus, by doubling of the volume measure, we have
\begin{align}
\fint_{B_{r}(x_j)}\big|\langle\nabla(A_{2r}\circ v)^a,\nabla(A_{2r}\circ v)^b\rangle - \delta^{ab}\big|<C(n)\epsilon\, .
\end{align}
However in addition we also have 
\begin{align}
\fint_{B_{r}(x_j)}\big|\langle\nabla(A_{r}\circ v)^a,\nabla(A_{r}\circ v)^b\rangle - \delta^{ab}\big|<\epsilon\, .
\end{align}
By using the Gram-Schmidt process, it follows that there exist lower triangular matrices, $T_1,T_2$
with $|T_1-I|<C(n)\epsilon$, $|T_2-I|<C(n)\epsilon$, such that 
\begin{align}
\fint_{B_{r}(x_j)}\langle(T_2A_{2r}\circ \nabla v)^a,(T_2A_{2r}\circ \nabla v)^b\rangle =\delta^{ab}\, ,
\notag
\end{align}
\begin{align}
\fint_{B_{r}(x_j)}\langle(T_1A_{r}\circ \nabla v)^a,(T_1A_{r}\circ \nabla v)^b\rangle =\delta^{ab}\, .
\notag
\end{align}
We can assume that $\epsilon$ has been chosen 
small enough that $T_1$, $T_2$ have positive diagonal entries which implies that the lower triangular
matrices
$T_1A_r$ and $T_2A_{2r}$ do as well. Define $H$ by
$$
(H)_{s,t}=\fint_{B_{r}(x_j)}\langle\nabla v_s,\nabla v_t\rangle \, .
$$
It follows from the above that we have two so-called Cholesky
decompositions of the positive definite symmetric matrix $H$; \cite{Golub_Van_Loan96}. Namely,
$((T_1A_r)^{-1}))^*(T_1A_r)^{-1}=((T_2A_{2r})^{-1}))^*(T_2A_{2r})^{-1}=H$.
Since for lower triangular matrices with positive diagonal entries and $H$ positive definite,
the Cholesky decomposition is unique, it follows that
$(T_1A_r)^{-1}=(T_2A_{2r})^{-1}$. Therefore, $T_1A_r=T_2A_{2r}$ which suffices to prove the claim.
 $\square$\\

Now let us record some very important consequences of Claim 1.  First, since by our normalization, $A_2\equiv I$, 
we have for $r\geq 2$ the sublinear growth estimate
\begin{align}
|A_r|,\, |A^{-1}_r|\leq r^{C(n)\epsilon}\, .
\end{align}
In particular, since $A_r\circ v:B_{r}(x_j)\to \dR^k$ is an $\epsilon$-splitting,
 and hence $\sup_{B_r(x_j)}|\nabla(A_r\circ  v)|\leq 1+\epsilon$, we have for any $2\leq r\leq r_j^{-1}$ the {\it sublinear growth} conditions
\begin{align}\label{e:k2:2}
&\sup_{B_r(x_j)}|\nabla v^a_j|\leq (1+C\epsilon)r^{C\epsilon}\, ,\notag\\
&\sup_{B_r(x_j)}|\omega_j|\leq (1+C\epsilon)r^{C\epsilon}\, ,\notag\\
&r^2\fint_{B_r(x_j)} |\nabla^2 v^a_j|^2\leq C\epsilon r^{C\epsilon}\, ,
\end{align}
where $\omega_j\equiv dv_{j}^1\wedge\cdots\wedge dv_{j}^k$ is the pullback $k$-form.  

\begin{remark}
\label{r:slg}
The sublinearity of the growth estimates in (\ref{e:k2:2}) will play a fundamental role in the proof; see in particular, Claims 3--5.
\end{remark}

Our first application of these estimates is the following, which uses the induction statement to conclude
that $v_{j}^1,\ldots,v_{j}^{k-1}$ are improving in their splitting behavior as $j\to\infty$.  \\

{\bf Claim 2:}  There exists a lower triangular matrix $A$ such that $A\circ v: B_2(x_j)\to \dR^k$ is a 
$C(n)\epsilon$-splitting while for each $R>0$ the restricted map $A\circ v:B_R(x_j)\to \dR^{k-1}$, obtained by 
dropping the last function, is an $\epsilon_j(R)$-splitting map, where $\epsilon_j(R)\to 0$ if $j\to\infty$
and $R$ is fixed.\\

To prove the claim let us first denote by $\tilde v:B_{2r_j^{-1}}(x_j)\to \dR^{k-1}$ the map obtained by 
dropping the last function $v^k$.  By our induction hypothesis, there exists for every $r\geq 2$, a
lower 
triangular matrix $\tilde A_r\in GL(k-1)$ with positive diagonal entries, such that $\tilde A_r\circ \tilde v:B_{r}(x_j)\to \dR^{k-1}$ is an 
$\epsilon_j$-splitting map with $\epsilon_j\to 0$.  Since both $\tilde v$ and $\tilde A_2\circ \tilde v$ are in 
particular $\epsilon$-splittings on $B_2(x_j)$ with $\tilde A_2$ lower triangular, then arguments similar to
those in Claim 1 give $|\tilde A_2-I|<C(n)\epsilon$, and  the growth estimates
\begin{align}\label{e:k:1}
\sup_{B_r(x_j)}|\nabla (\tilde A_2\circ \tilde v)|&\leq (1+C\epsilon_j)r^{C\epsilon_j}\, ,\notag\\
r^2\fint_{B_r(x_j)} |\nabla^2 (\tilde A_2\circ \tilde v)|^2&\leq C\epsilon_j r^{C\epsilon_j}\, .
\end{align}

In particular, we can use the Hessian estimate and a Poincar\'e inequality to conclude
\begin{align}
\Big|\fint_{B_2(x)}\big|\langle \nabla(\tilde A_2\circ \tilde v)^a,&\nabla(\tilde A_2\circ \tilde v)^b\rangle-\delta^{ab}\big|
-\fint_{B_R(x)}\big|\langle \nabla(\tilde A_2\circ \tilde v)^a,\nabla(\tilde A_2\circ \tilde v)^b\rangle-\delta^{ab}\big|\Big|\notag\\
&\leq \fint_{B_2(x)}\Big|\big|\langle \nabla(\tilde A_2\circ \tilde v)^a,\nabla(\tilde A_2\circ\tilde v)^b\rangle-\delta^{ab}\big|
-\fint_{B_R(x)}\big|\langle \nabla(\tilde A_2\circ \tilde v)^a,\nabla(\tilde A_2\circ \tilde v)^b\rangle-\delta^{ab}\big|\Big|\notag\\
&\leq C(n,R)\fint_{B_R(x)}\Big|\big|\langle \nabla(\tilde A_2\circ \tilde v)^a,\nabla(\tilde A_2\circ \tilde v)^b\rangle-\delta^{ab}\big|
-\fint_{B_R(x)}\big|\langle \nabla(\tilde A_2\circ \tilde v)^a,\nabla(\tilde A_2\circ \tilde v)^b\rangle-\delta^{ab}\big|\Big|\notag\\
&\leq C(n,R)\fint_{B_R(x)}\big|\nabla \langle \nabla(\tilde A_2\circ \tilde v)^a,\nabla(\tilde A_2\circ \tilde v)^b\rangle\big|
\leq \epsilon_j(R)\to 0\, .
\end{align}
Thus, for each $R>0$ fixed we have  $\tilde A_2\circ \tilde v:B_R(x_j)\to \dR^{k-1}$ is an $\epsilon_j(R)$-splitting,
where $ \epsilon_j(R)\to 0$
 when $j\to\infty$ with $R$.
Finally, if we let $A=\tilde A_2\oplus 1$ act on $\dR^k$ by fixing the last component, then we have proved the claim. $\square$\\

\noindent
{\bf Note:} As a point of notation, we mention that as above, from now on, the symbol, $\epsilon_j(R)$, 
 will always denote a quantity, regardless of origin, satisfying $\epsilon_j(R)\to 0$,
 when $j\to\infty$ with $R$ 
  fixed.

\noindent
{\bf Note:} In the course of the proof, on more than one occasion, we will replace $v$ by $A\circ v$, 
where $A$ is a lower triangular matrix with positive diagonal entries and with 
$|A-I|<C(n)\epsilon$.  In particular, from 
this point on in the proof, we will assume $v^a$ has been normalized as in Claim 2.  
Thus $v^a:B_2(x_j)\to \dR^k$ will be taken to be an $C\epsilon$-splitting map, while $v^a:B_R(x_j)\to \dR^{k-1}$ is an $\epsilon_j(R)$-splitting map.  \\

A useful consequence is that for each $R>0$ and $1\leq \ell\leq k-1$,
we have 
\begin{align}\label{e:k:3}
\fint_{B_R(x_j)}|\nabla^2 v^a|^2\leq \epsilon_j(R)\to 0\, .
\end{align}

\begin{remark}
\label{r:ltg}
By way of orientation, we mention at this point that our long term goal is to show 
$$\fint_{B_R(x_j)}|\nabla^2 v^k_j|^2\leq \epsilon_j(R)\to 0\, $$
which is the content of Claim 6. Once this has been achieved, the proof will be virtually complete.
\end{remark}

 Our next goal is to study in more detail the properties of $\omega_j = \omega^k_j=dv^1_{j}\wedge\cdots\wedge dv^k_{j}$. 
 First, since
\begin{align}
\nabla\omega_j = \nabla(dv^1_j)\wedge\cdots\wedge dv^k_j+\cdots+dv^1_j\wedge\cdots\wedge\nabla(dv^k_j)\, ,
\end{align}
we can use \eqref{e:k2:2} to obtain for $2\leq r\leq r_j^{-1}$,
\begin{align}\label{e:k2:2.2}
r^2\fint_{B_r(x_j)}|\nabla \omega_j|^2\leq C\epsilon\,r^{C\epsilon}\, .
\end{align}

Recall that our underlying assumptions are that for every $r\geq 1$,
we have
\begin{align}
r^2\fint_{B_r(x_j)}|\Delta|\omega_j|\, |\leq \delta_j \fint_{B_r(x_j)}|\omega_j|\, .
\end{align}
By combining this with \eqref{e:k2:2}, we get that for every $2\leq r\leq r_j^{-1}$,
\begin{align}\label{e:k2:1}
r^2\fint_{B_r(x_j)}|\Delta|\omega_j|\, |\leq C\delta_j r^{C\epsilon}\, . 
\end{align}

Now we are ready to make our third claim:\\

{\bf Claim 3:} For each fixed $R\geq 1$, we have  $\fint_{B_R(x_j)}\big||\omega_j|^2-\fint_{B_R(x_j)}|\omega_j|^2\big|\to 0$.\\

The proof of Claim 3 will rely on
the sublinear growth estimates 
(\ref{e:k2:2}), (\ref{e:k2:2.2}),  (\ref{e:k2:1}), standard heat kernel 
estimates for almost nonnegative
 Ricci curvature, (\ref{e:hk1})--(\ref{e:hk2}) and the Bakry-Emery 
gradient estimate for the heat kernel (\ref{e:be}).  In particular, the sublinear growth
condition in  (\ref{e:k2:2.2}) enters crucially in (\ref{e:k2:4}) and its consequence  (\ref{e:S}).

Fix $R\geq 1$ and consider the maximal function
\begin{align}
M^R(x)\equiv \sup_{r\leq R} \fint_{B_r(x)}|\Delta|\omega_j||\, ,
\end{align}
for $x\in B_R(x_j)$.  
Since by the Bishop-Gromov inequality, the Riemannian measure is doubling, we can combine the usual 
maximal function arguments with \eqref{e:k2:1} and conclude that there exists a subset $U_j\subseteq B_{R}(x_j)$ such that
\begin{align}
\label{e:mf}
&\frac{\Vol(B_R(x_j)\setminus U_j)}{\Vol(B_R(x_j))}\leq \epsilon_j(R)\to 0\, ,\notag\\
&M^R(x)\leq \epsilon_j(R)\to 0\, ,
\end{align}
for all $x\in U_j$.  Relation (\ref{e:mf}) will be used in (\ref{e:fundamental1}).

As with the symbol, $\epsilon_j(R)$, the symbol $\epsilon_j(S)$
 will always denote a quantity, regardless of origin, satisfying $\epsilon_j(S)\to 0$
 when $j\to\infty$ with $S$ fixed.

 Now let $\varphi\geq 0$ be a smooth cutoff function as in \cite{ChC1}, such that
$\varphi\equiv 1$ on $B_{r_j^{-1}/2}(p)$, ${\rm supp}\, \varphi\subset B_{r_j^{-1}}(p)$,
and such that $r^{}_j|\nabla\varphi|,r_j^{2}| $$ $$ \Delta\varphi|\leq C(n)$.  For $x\in B_R(x_j)$ let us consider the function
\begin{align}
\int_{M^n_j} |\omega_j|\varphi \rho_t(x,dy)\, ,
\end{align}
where $\rho_t$ is the heat kernel centered at $x$. 
  Then we have the equality
\begin{align}\label{e:k:2}
\frac{d}{dt}\int_{M^n_j} |\omega_j|\varphi \rho_t(x,dy) = \int \Big(\Delta|\omega_j|\varphi
+2\langle\nabla |\omega_j|,\nabla \varphi\rangle+|\omega_j|\Delta\varphi\Big)\rho_t(x,dy)\, .
\end{align}
As a consequence of our assumption that $\Ric_{M^n_j}\geq -\delta_j r_j^{2}$, 
we have 
the usual heat kernel estimates (\cite{SY_Redbook})
\begin{equation}
\label{e:hk1}
\rho_t(x,y)\leq C(n)\Vol(B_{\sqrt{t}}(x))^{-1/2}\Vol(B_{\sqrt{t}}(y))^{-1/2}
e^{-\frac{d^2(x,y)}{4t}+C(n)\delta_j r_j^{2}t}\, ,
\end{equation}
which implies that for $y\in B_{r_j^{-1}}(x)$ and $t\leq r_j^{-2}$, we have
\begin{equation}
\rho_t(x,y)\leq C(n)\Vol(B_{\sqrt{t}}(x))^{-1/2}\Vol(B_{\sqrt{t}}(y))^{-1/2}e^{-\frac{d^2(x,y)}{4t}}\, .
\end{equation}
 We can use the volume doubling and 
monotonicity properties to observe the following useful inequality. If $y\in B_r(x)$, then 
\begin{align}
\label{e:hk2}
\rho_t(x,y)&\leq C(n)\Big(\frac{\Vol(B_{r}(x))}{\Vol(B_{\sqrt{t}}(x))^{1/2}\Vol(B_{\sqrt{t}}(y))^{1/2}}\Big)\Vol(B_r(x))^{-1}
e^{-\frac{d^2(x,y)}{4t}}\notag\\
&\leq C(n)\Big(\frac{r}{t^{1/2}}\Big)^n\Vol(B_r(x))^{-1}e^{-\frac{d^2(x,y)}{4t}}\, .
\end{align}

Let us fix $S>>R\geq 2$ and consider times $0<t\leq S^2$.  
By combining the heat kernel estimate, (\ref{e:hk2}), with the growth estimates
 (\ref{e:k2:2}), (\ref{e:k2:2.2}),
for all $x\in B_R(x_j)$ and $0<t\leq S^2$,
we can bound the second two terms of \eqref{e:k:2} by
\begin{align}
\label{e:last2}
\int_{M^n_j} \big|\langle\nabla |\omega_j|,\nabla \varphi\rangle+|\omega_j|\,|\Delta\varphi|\big|\,\rho_t(x,dy) 
&= \int_{A_{r_j^{-1}/2,r_j^{-1}}(x_j)} 
\big|\langle\nabla |\omega_j|,\nabla \varphi\rangle+|\omega_j|\,|\Delta\varphi|\big|\rho_t(x,dy)\, \notag\\
&\leq Cr_jr_j^{1-C\epsilon}\Vol(B_{r_j^{-1}}(x_j)) \Vol(B_{\sqrt{t}}(x))^{-1/2}\Vol(B_{\sqrt{t}}(y))^{-1/2}\,e^{-\frac{1}{4t}r^{-2}_j}\notag\\
&\leq Cr_j^{2-C\epsilon}\Big(\frac{1}{r_j t^{1/2}}\Big)^{n}e^{-\frac{1}{4t}r^{-2}_j}\leq\epsilon_j(S)\to 0\, .
\end{align}

To estimate the first term of \eqref{e:k:2} is more involved.  To this end,
we begin with an estimate in which we must
restrict attention to points $x\in U_j\subseteq B_R(x_j)$; see (\ref{e:mf}).  
Below, we write $t=r^2$ and so, we consider $0<r<S$.
We also put $r_\alpha=2^\alpha r$.  Suppose first that $\sqrt t=r\leq R$.
Then we have
\begin{align}
\label{e:fundamental1}
\int_{M^n_j} \big|\Delta|\omega_j|\big|\varphi\rho_{r^2}(x,dy) 
&= \int_{B_{r}(x)}\big|\Delta|\omega_j|\big|\varphi\rho_{r^2}(x,dy)
 + \sum_\alpha \int_{A_{r_\alpha,r_{\alpha+1}}(x)}\big|\Delta|\omega_j|\big|\varphi\rho_{r^2}(x,dy)\, 
 \\
&\leq C(n)\fint_{B_{r}(x)}\big|\Delta|\omega_j|\, \big| 
+ C(n)\sum_\alpha \Big(\frac{r_\alpha}{r}\Big)^ne^{-\big(r^{-1}r_\alpha\big)^2}\fint_{B_{2^\alpha r}(x)}
\big|\Delta|\omega_j|\, \big|\notag\\
&\leq C(n)\fint_{B_{r}(x)}\big|\Delta|\omega_j|\, \big| 
+ C(n)\sum_\alpha 2^{n\alpha}e^{-2^{2\alpha}}\fint_{B_{2^\alpha r}(x)}\big|\Delta|\omega_j|\, \big|\notag 
\\
 &= C(n)\fint_{B_{r}(x)}\big|\Delta|\omega_j|\big| 
+ C(n)\sum_{r^\alpha\leq R} 2^{n\alpha}e^{-2^{2\alpha}}\fint_{B_{2^\alpha r}(x)}\big|\Delta|\omega_j|\,\big |
+ C(n)\sum_{r_\alpha>R} 2^{n\alpha}e^{-2^{2\alpha}}\fint_{B_{2^\alpha r}(x)}\big|\Delta|\omega_j|\, \big|\, \notag
\\
 &\leq C\epsilon_j(R)+ C\sum_{r_\alpha\leq R} 2^{n\alpha}e^{-2^{2\alpha}}\epsilon_j(R)+ CR^{-2}\sum_{r^\alpha>R} 2^{n\alpha}e^{-2^{2\alpha}}\delta_j\to 0\, .\notag
\end{align}
Note that in estimating the first two terms in the last line of 
(\ref{e:fundamental1}) 
we use the maximal function estimate (\ref{e:mf}),  which is the reason
for restricting attention to $x\in U_j$.  For the third term in the last line we use (\ref{e:k2:1}).

Similarly, $\sqrt t=r>R$, the first two terms on the last line of (\ref{e:fundamental1}) are absent 
and we just get
\begin{align}
\label{e:fundamental2}
\int_{M^n_j} \big|\Delta|\omega_j|\big|\varphi\rho_{r^2}(x,dy) \leq 
CR^{-2}\sum_\alpha 2^{n\alpha}e^{-2^{2\alpha}}\delta_j\to 0\, .\
\end{align}

By combining (\ref{e:k:2}), (\ref{e:last2}), (\ref{e:fundamental1}), (\ref{e:fundamental2}), we get
 for $x\in U_j$ and $0<t\leq S^2$,
\begin{align}\label{e:k2:3}
\Big|\frac{d}{dt}\int_{M^n_j} |\omega_j|\varphi \rho_t(x,dy)\Big| \leq \epsilon_j(S)\to 0\, ,
\end{align}
uniformly in $U_j$.  

At this point,  by using \eqref{e:k2:3} and integrating with respect to $t$
from $0$ to $S^2$,  we have for any $x\in U_j\subseteq B_R(x_j)$, 
\begin{align}\label{e:k2:5}
\Big||\omega_j|(x)-\int_{M^n_j} |\omega_j|\varphi \rho_{S^2}(x,dy)\Big|\leq \epsilon_j(S)
\cdot S^2\to 0\, ,
\end{align}
uniformly in $U_j$.

By arguing in a manner similar to the above
(but without the need for a maximal function estimate)
 we can use \eqref{e:k2:2.2}, to see that for all  $x\in B_{2R}(x_j)$ 
\begin{align}\label{e:k2:4}
\int_{M^n_j} \big|\nabla(|\omega_j|\varphi)\big|^2\rho_{S^2}(x,dy)& \leq 2\int_{M^n_j} |\nabla\omega_j|^2 
+ |\omega_j|^2|\nabla\varphi|^2\rho_{S^2}(x,dy)\notag\\
&\leq C\sum 2^{n\alpha}e^{-2^{2\alpha}}\fint_{B_{2^\alpha S}(x)} |\nabla \omega_j|^2
+Cr_{j}^{2-C\epsilon}\Big(\frac{1}{Sr_j}\Big)^{n}e^{-\frac{1}{S^2}r_j^{-2}}\, \notag\\
&\leq C\,S^{-2+C\epsilon}+\epsilon_j(S)\, ,
\end{align}
where without loss of generality, we can assume that our original $\epsilon$ has been
chosen so that 
$-2+C\epsilon<0$.  As previously mentioned, it is at just this point that the sublinearity in (\ref{e:k2:2.2}) has entered crucially,
 giving rise 
to the {\it negative} power of $S$ in (\ref{e:k2:4}), which comes to fruition in
(\ref{e:S}).


We  have that
$H_t\big(|\omega_j|\varphi\big)=\int_{M^n_j} |\omega_j|\varphi \rho_{t}(x,dy)$ 
solves the heat equation. So using the Bakry-Emery gradient estimate, 
 \cite{BakryEmery_diffusions}, we have for any $x\in B_{2R}(x_j)$ 
\begin{align}
\label{e:be}
|\nabla H_t\big(|\omega_j|\varphi\big)|^2(x)\leq 
e^{\delta_{j}r_{j}^2 t}H_t|\nabla(|\omega_j|\varphi)|^2(x)\, .
\end{align}
In particular, using \eqref{e:k2:4} we have 
\begin{align}
{}\sup_{B_{2R}(x_j)}\Big|\nabla_x \int_{M^n_j} |\omega_j|\varphi \rho_{S^2}(x,dy)\Big| \leq \frac{C}{S^{1-C\epsilon/2}}+\epsilon_j(S)
\, .
\end{align}
Combining this with \eqref{e:k2:5} we get for any pair of points, $x,y\in U_j$, 
\begin{align}
\label{e:S}
\big|\, |\omega_j|(x) -|\omega_j|(y)\big|&\leq \big|\omega_j(x)
 -\int_{M^n_j} |\omega_j|\varphi \rho_{S^2}(x,dz)\big|+\big|\omega_j(y) 
-\int_{M^n_j} |\omega_j|\varphi \rho_{S^2}(y,dz)\big|\notag\\
&+\big|\int_{M^n_j} |\omega_j|\varphi \rho_{S^2}(x,dy) -\int_{M^n_j} |\omega_j|\varphi \rho_{S^2}(y,dz)\big|\notag\\
&\leq \epsilon_j(S) + \frac{CR}{S^{1-C\epsilon/2}}\, .
\end{align}

By letting $S$ tend to infinity sufficiently slowly, we get for 
 $x,y\in U_j$, that 
\begin{align}
\big|\, |\omega_j|(x) -|\omega_j|(y)\big|&\leq \epsilon_j(R)\to 0\, .
\end{align}

Finally, to finish the proof, we use the supremum bound \eqref{e:k2:2.2} on $|\omega|$ to note that for $x\in U_j$, we have
\begin{align}
\Big|
\fint_{B_R(x_j)}|\omega_j|^2-|\omega_j|^2(x)\Big|&\leq \fint_{B_R(x_j)}\Big||\omega_j|^2-|\omega_j|^2(x)\Big|\notag\\
&\leq \fint_{B_R(x_j)}\big||\omega_j|-|\omega_j|(x)\big|\cdot \big||\omega_j|+|\omega_j|(x)\big|\notag\\
&\leq C(n,R)\fint_{B_R(x_j)}\big||\omega_j|-|\omega_j|(x)\big|\notag\\
&\leq C(n,R)\fint_{U_j}\big||\omega_j|-|\omega_j|(x)\big|+C(n,R)\fint_{B_R(x_j)\setminus U_j}\big||\omega_j|-|\omega_j|(x)\big|\notag\\
&\leq C(n,R)\epsilon_j (R) +C(n,R)\cdot \frac{\Vol(B_R(x_j)\setminus U_j)}{\Vol(B_R(x_j))}\to 0\, . 
\end{align}

Hence, we have 
\begin{align}
\fint_{B_R(x_j)}\Big|\,  |\omega_j|^2-\fint_{B_R(x_j)}|\omega_j|^2\Big|\leq
 \Big|\fint_{B_R(x_j)}|\omega_j|^2-|\omega_j|^2(x)\Big|+\fint_{B_R(x_j)}\Big||\omega_j|^2-|\omega_j|^2(x)\Big|\to 0\, ,
\end{align}
which proves the claim. $\square$\\

We know from \eqref{e:k2:2}, (\ref{e:k2:2.2}) that $|\nabla\omega_j|$ has $L^2$ bounds.
  It is crucial to improve these to bounds that are small compared to $\epsilon$.  This is the content of the next claim:\\

{\bf Claim 4:}
For fixed $R$ we have 
\begin{equation}\label{e:omegatozero}
\fint_{B_R(x_j)}|\nabla \omega_j|^2
\leq \epsilon_j(R)\to 0\,  .
\end{equation}

To see this fix $R$ and as in \cite{ChC1}, let $\varphi:B_{2R}(x_j)\to \dR^+$ be a cutoff function 
with $\varphi\equiv 1$ on $B_R(x_j)$ and $R|\nabla\varphi|,\, R^2|\Delta\varphi|\leq C(n)$.  
We use the Bochner formula
\begin{align}
\Delta |\omega_j|^2 &= 2|\nabla \omega_j|^2+2\langle\sum_b dv^1_j\wedge\cdots \Ric(dv^{b}_j)\wedge
\cdots\wedge dv^{k}_j,\omega\rangle\notag\\
&\,\,\,\,\,\,\,\,\,\,\,\,\,\,\,\,\,\,\,\,\,\,\,+\langle\sum_{a\neq b} dv^{1}_j\wedge \nabla^c(dv^{a}_j)\wedge
\cdots\wedge \nabla_c(dv^{b}_j)\wedge\cdots\wedge dv^{k}_j,\omega_j\rangle \notag\\
&\geq 2|\nabla\omega_j|^2 - C(n)\delta_j^2r_j^2|\omega_j|^2-C(n)|\nabla(dv)|^2|\omega_j|^2\notag\\
&\,\,\,\,\,\,\,\,\,\,\,\,\,\,\,\,\,\,\,\,\,\,\,+\langle\sum_{a\neq b} dv^{1}_j\wedge \nabla^c(dv^{a}_j)
\wedge\cdots\wedge \nabla_c(dv^{b}_j)\wedge\cdots\wedge dv^{k}_j,\omega_j\rangle\, ,
\end{align}
which together with the growth estimates \eqref{e:k2:2} allows us to compute
\begin{align}\label{e:k2:6}
\fint_{B_R(x_j)}|\nabla \omega_j|^2&\leq C(n)\fint_{B_{2R}(x_j)} \varphi\Delta|\omega_j|^2 
+C(n,R)\sum_{a\neq b}\fint_{B_{2R}(x_j)} |\nabla^2 v^a_j|\,|\nabla^2 v^b|
+ C(n,R)\delta_jr_j^2\notag\\
&\leq  C \fint_{B_{2R}(x_j)} \Delta\varphi\,\big(|\omega_j|^2-\fint_{B_{2R}(x_j)}|\omega_j|^2\big)\notag\\
&\,\,\,\,\,\,
+C(n,R)\sum_{a\neq b}\Big(\fint_{B_{2R}(x_j)
}|\nabla^2 v^a|^2\Big)^{1/2}\Big(\fint_{B_{2R}(x_j)}|\nabla^2 v^b|^2\Big)^{1/2}+ \epsilon_j(R)\notag\\
&\leq C\fint_{B_{2R}(x_j)} \big||\omega_j|^2-\fint_{B_{2R}(x)}|\omega_j|^2\big|+\epsilon_j(R)\leq \epsilon_j(R)\to 0\, ,
\end{align}
where we have used Claim 3 and \eqref{e:k:3}.  
Note that it is important that we have $a\neq b$ in the summation, so that at least one of the 
Hessian terms in each factor is going to zero as $j\to\infty$.  This proves the claim.\qed

As mentioned in Remark \ref{r:ltg}, 
to complete the proof we must show that $\fint_{B_R(x_j)}|\nabla^2 v^k_j|^2\to 0$ 
as $j\to\infty$.  
To prove this we will first pass to limits and obtain information on the limiting space.  
That is, we have been considering a sequence 
$(M^n_j,d_j,x_j)$ 
with $\Ric_{M^n_j}\geq -\delta_jr_j^2\to 0$.
  After passing to a subsequence if necessary, we can take a measured pointed Gromov-Hausdorff limit
\begin{align}
(M^n_j,d'_j,x_j)\stackrel{d_{GH}}{\longrightarrow} (X,d,x)\, ,
\end{align}
to obtain an $RCD(n,0)$ space $X$, see \cite{Ambrosio_Calculus_Ricci}, \cite{Ambrosio_Ricci}.
The fact that $X$ is
an $RCD(n,0)$ space is used below in applying the mean value estimate (\ref{e:mvt}),
which is known to hold for such spaces.

In addition, we can assume that the
functions $v_j^\ell$ converge to harmonic functions.
\begin{align}
v_j^\ell \to v^\ell:X\to \dR\, .
\end{align}
Indeed, for any ball $B_R(x_j)$ we
 can characterize $v^\ell_j$ as minimizers of the Dirichlet energy with fixed Dirichlet boundary values.
  Our assertion then follows from the lower semicontinuity of the Dirichlet energy \cite{Ambrosio_Ricci} 
combined with the Mosco convergence of the Dirichlet form \cite{GMS_Stability}, to see that the limit also
 minimizes the Dirichlet energy on any ball.

Observe first, that by using Claim 2 and Lemma \ref{l:harmonic_splitting_intro},  we have
\begin{align}
X=\dR^{k-1}\times Y\, ,
\end{align}
where $v^1,\ldots,v^{k-1}:X\to \dR$ are linear functions which induce the $\dR^{k-1}$
 factor and we can identify $Y=(v^1,\ldots,v^{k-1})^{-1}(0^{k-1})$.  We are left with understanding the behavior of $v^k$.
  We will see in Claim 6 that it too is linear, and in the process prove our Hessian estimate.  We first show the following:\\

{\bf Claim 5:}  There exists $a_1,\ldots,a_{k-1}\in \dR$ with $|a_\ell|<C(n)\epsilon$ such that 
$v^k-a_1v^1-\cdots-a_{k-1}v^{k-1}:X\to \dR$ is a function of only the $Y$ variable.\\

\vskip1mm

To prove the claim let us fix any vector $V\in \dR^{k-1}$ and consider the map $Dv^k:X\to \dR$ defined by
\begin{align}
D v^k(y)= v^k(y+V)-v^k(y)\, ,
\end{align}
where of course, the translation $x\to x+V$ is  well defined, since $X\equiv \dR^{k-1}\times Y$.  
The function $v^k(y)$ is harmonic, and the translation map $x\to x+V$ is a
measure preserving isometry.
 Thus,  $v^k(x+V)$ is a harmonic function as well.  Since $X$ is an $RCD$ space, 
and hence the Laplacian $\Delta$ on $X$ is linear, it follows that $Dv^k$ is harmonic. Using the estimates \eqref{e:k2:2} 
we have 
the growth condition

\begin{align}
\sup_{B_r(x)}|Dv^k|\leq C|V|^{1+C\epsilon}\cdot r^{C\epsilon}\, .
\end{align}

This is to say that $Dv^k$ is a harmonic function with sublinear growth.  
It follows that $Dv^k$ must be a constant.  Indeed, let $\varphi$ be a cutoff on 
$B_{2S}(x)$ with $\varphi\equiv 1$ on $B_S(x)$ and $|\nabla\varphi|\leq 10S^{-1}$. 
Then on the one hand, we have since $Dv^k$ is harmonic and the Dirichlet form is bilinear that

\begin{align}
 0&=\fint_{B_{2S}(x)} \langle\nabla Dv^k,\nabla (\varphi^2 Dv^k)\rangle\notag\\
  &=\fint_{B_{2S}(x)} \varphi^2 |\nabla Dv^k|^2+2\fint_{B_{2S}(x)} \varphi\,Dv^k\,
 \langle\nabla Dv^k,\nabla\varphi\rangle\, .
\end{align}

By rearranging terms, we obtain
\begin{align}
\fint_{B_{S}(x)} |\nabla Dv^k|^2&\leq \fint_{B_{2S}(x)} \varphi^2 |\nabla Dv^k|^2\\
&\leq \frac{1}{2}
\fint_{B_{2S}(x)} \varphi^2 |\nabla Dv^k|^2+8\fint_{B_{2S}(x)} |Dv^k|^2|\nabla \varphi|^2\\
&\leq CS^{-2+C\epsilon}\, ,
\end{align}
where without loss of generality, we can assume that $\epsilon$ is so small that $-2+C\epsilon<0$.

On the other hand,
 $\Ric_{M^n_j}\geq -(n-1)\delta_j r_j^2\to 0$ 
and so $X$ is an $RCD(n,0)$ space. On such spaces, 
there is a mean value inequalilty for the norm squared of the gradient of a harmonic function;
 see for instance \cite{MondinoNaber_Rectifiability}.  When applied to the harmonic function
$Dv^k$ it gives for  $r>0$ fixed and $S\to \infty$ 

\begin{align}
\label{e:mvt}
\sup_{B_r(x)}|\nabla Dv^k|^2 \leq C\fint_{B_S(x)}|\nabla Dv^k|^2\leq  CS^{-2+C\epsilon} \to 0\, .
\end{align}
 Note that once again, we have exploited the sublinearity of the growth estimates.  In particular, it now follows that $Dv^k$ is a constant.  
Since this holds for any $V\in \dR^{k-1}$, we have that $v^k$ is linear in the $\dR^{k-1}$ variable.  
More precisely, since the $\dR^{k-1}$ factor is spanned by $v^1,\ldots, v^{k-1}$ we have 
\begin{align}
v^k=v^k_Y+a_1v^1+\cdots+a_{k-1}v^{k-1}\, ,
\end{align}
where $v^k_Y:Y\to \dR$. 
 Since $v_j\to v:X\to \dR^{k}$ are $C\epsilon$-splittings on $B_2(x_j)$, we automatically have the bounds $|a_\ell|\leq C(n)\epsilon$. 
This finishes the claim.  $\square$\\

To complete the proof, we want to see that the Hessians of $v^k_j$ are tending to zero as $j\to\infty$. 
This is the content of Claim 6 below. However, prior to stating this claim, we will make some additional
normalizations.

To begin with, we can use Claim 5 to further normalize the mappings $v_j$ by
 composing  with another lower triangular matrix with positive diagonal
 entries.  Indeed, as a corollary 
of Claim 5, we can choose a lower triangular matrix $A$ with $|A-I|<C(n)\epsilon$, and whose restriction to 
the first $(k-1)\times (k-1)$ terms is the identity, such that $Av_j:B_2(x_j)\to \dR^k$ 
is still an $C(n)\epsilon$-splitting, while $A\circ v^k_j\to A\circ v^k:\dR^{k-1}\times Y\to \dR$ 
is independent of the $\dR^{k-1}$ factor.  Further, let us consider the induced form
 $A\circ \omega_j=d(A\circ v_j^1)\wedge\cdots\wedge d(A\circ v^k_j) = dv_j^1\wedge\cdots\wedge d(A\circ v^k_j)$. 
 Then after multiplying the $k^{th}$ row of $A$ by a constant $c$ with $|c-1|\leq C(n)\epsilon$ we may further assume that
\begin{align}
\fint_{B_2(x_j)}|A\circ \omega_j|^2 = 1\, .\\ \notag
\end{align}

From this point forward in the proof, for ease of notation, we will write $v_j$, for what
was denoted above by $A  \circ v_j$  In particular, this $v_j$ differs from 
the original mapping $u_j$ only by composition with a lower triangular matrix with positive diagonal entries. We will eventually see that $v_j:B_1(x_j)\to \dR^k$ is an $\epsilon_j$-splitting, 
which will give the desired contradiction and finish the proof.

{\bf Claim 6.} For each $R>0$, we have $\fint_{B_R(x_j)}|\nabla^2 v^k_j|^2\leq \epsilon_j(R)\to 0$.\\

The fact that  $v_j:B_{R}(x_j)\to \dR^{k-1}$ is an $\epsilon_j(R)$-splitting,
 $$
\fint_{B_2(x_j)}|\omega_j|^2 = 1\, ,
$$
together with
$$
\fint_{B_R(x_j)}|\nabla\omega_j|^2\leq \epsilon_j(R)\to 0\, ,
$$
implies
\begin{align}\label{e:k:2:1}
\fint_{B_R(x_j)}\big| |\omega^\ell_j|-1\big|\leq \epsilon_j(R)\qquad ({\rm for\,\, all}\,\, 1\leq\ell\leq k)\, .
\end{align}
 
Now we will show that
\begin{align}
\fint_{B_R(x_j)}\big||\nabla v^k_j|^2-1\big|\leq \epsilon_j(R)\to 0\, .
\end{align}
Once this is accomplished, as we have done repeatedly, we can argue with Bochner's formula
to obtain the Hessian estimate in the claim.

Define the $1$-form
\begin{align}
V_j\equiv \langle \omega^{k-1}_j,\omega_j\rangle\, .
\end{align}
Note $\omega^{k-1}_j\wedge V_j$ is proportional to 
$\omega_j=\omega^{k-1}_j\wedge d v^k_j= \omega^{k-1}_j\wedge \big(d v^k_j-\pi_{k-1}d v^k_j\big)$.  
More generally, we have that $V_j\in \text{span}\{\nabla v^1_j,\ldots,\nabla v^{k}_j\}$ is perpendicular to 
$\text{span}\{\nabla v^1_j,\ldots,\nabla v^{k-1}_j\}$. From the above, we get
\begin{align}\label{e:k:2:0}
\fint_{B_R(x_j)}|V_j-\big(dv^k_j-\pi_{k-1}dv^k_j\big)|\leq \epsilon_j(R)\, .
\end{align}

On the other hand, by \eqref{e:k2:6} we have 
\begin{align}
\fint_{B_R(x_j)}|\nabla V_j|^2\leq \epsilon_j(R)\to 0\, ,
\end{align}
and thus using \eqref{e:k:2:1} we have 
\begin{align}
\fint_{B_R(x_j)}\big| |V_j|-1\big|\leq \epsilon_j(R)\, .
\end{align}
Therefore, from \eqref{e:k:2:0} we get
\begin{align}\label{e:k:2:-1}
\fint_{B_R(x_j)}\Big||dv^k_j-\pi_{k-1}dv^k_j|-1\Big|\leq \epsilon_j(R)\to 0\, .
\end{align}
It follows that our main concern is to show $|\pi_{k-1}(d v^k_j)|\to 0$ in $L^1$ as $j\to\infty$.  Because we have the estimate 
\begin{align}\label{e:k_new:2}
\int_{B_R(x_j)}|\langle \nabla v^a_j,\nabla v^b_j\rangle-\delta^{ab}|\leq \epsilon_j(R)\to 0\, ,
\end{align}
for $a,b<k$ this is equivalent to showing that
\begin{align}\label{e:k:2:-2}
\int_{B_R(x_j)}|\langle \nabla v^\ell_j,\nabla v^k_j\rangle|\leq \epsilon_j(R)\to 0\, ,
\end{align}
for all $\ell<k$, which will be our primary goal now.\\

To accomplish this let us fix some $\ell<k$ and recall that $M^n_j\to X\equiv \dR^{k-1}\times Y$, where the $v^\ell_j\to v^\ell$ converge to the linear splitting factors and $v^k_j\to v^k$ converges to a function on the $Y$ variable.  In particular, notice in the limit that $|\langle \nabla v^\ell,\nabla v^k\rangle| = 0$.  One could therefore prove the result by showing that the energies of a sequence of harmonic functions actually converge in $L^1_{loc}$ to the energies of the limiting harmonic functions.  We will proceed by essentially proving a more effective version of this statement.\\

Thus, for each $(s,y)\in \dR^{k-1}\times Y\cap B_R(x_j)$ and $0<\epsilon_j<<r_2<<r_1<<1$ let us consider an open set $U(s,y,r_1,r_2)$ such that
\begin{align}
&\big(B_{r_2}(s_1,\ldots,s_{\ell-1})\times (s_\ell-r_1,s_\ell+r_1)\times B_{r_2}(s_{\ell+1},\ldots,s_{k-1})\times B_{r_2}(y)\big)\cap B_{R+r_1}(x_j)\subseteq U(s,y,r_1,r_2)\notag\\
&U(s,y,r_1,r_2)\subseteq \big(B_{2r_2}(s_1,\ldots,s_{\ell-1})\times (s_\ell-2r_1,s_\ell+2r_1)\times B_{2r_2}(s_{\ell+1},\ldots,s_{k-1}\big)\big)\cap B_{R+2r_1}(x_j)\, .
\end{align}
with respect to the Gromov-Hausdorff map from $M^n_j$ to $\dR^{k-1}\times Y$.  Clearly, we have the volume estimate
\begin{align}
C^{-1}(n,\rv,R)\leq r_1^{-1} r^{-(n-1)}_2\Vol(U(s,y,r_1,r_2))\leq C(n,\rv,R)\, ,
\end{align}
where $\rv>0$ is the noncollapsing constant.  Let us notice that as $r_2<<r_1\to 0$ we have that $U(s,y,r_1,r_2)$ is an approximately a product of balls with diameter tending to zero, and such that for each $z_1\in U(s,y,r_1,r_2)$ we have the important estimates
\begin{align}\label{e:k_new:1}
&\fint_{U(s,y,r_1,r_2)}\frac{|v^a_j(z_1)-v^a_j(z_2)|}{d(z_1,z_2)}dv_g(z_2)<O(\frac{r_2}{r_1})+\epsilon_j(R)\, ,\text{ for }a\neq \ell\notag\\
&\fint_{U(s,y,r_1,r_2)}\sqrt{\bigg|\frac{|v^\ell_j(z_1)-v^\ell_j(z_2)|}{d(z_1,z_2)}-1\bigg|}dv_g(z_2)<O(\frac{r_2}{r_1})+\epsilon_j(R)\, .
\end{align}
Note that the integrands above are bounded and converging to zero pointwise away from a set whose measure is going to zero relative to $U$ as $\frac{\epsilon_j}{r_2},\frac{r_2}{r_1}\to 0$.  In words, the $v^a_j$ for $a\neq \ell$ are becoming approximately constant functions and $v^\ell_j$ is becoming a norm one linear function in the domains as $\frac{\epsilon_j}{r_2},\frac{r_2}{r_1}\to 0$.

Now let $z_1,z_2\in U(s,y,r_1,r_2)$ with $\gamma_{z_1,z_2}:[0,d(z_1,z_2)]\to M$ a minimizing geodesic connecting them and let $d\equiv d(z_1,z_2)$.  Without loss of generality, let us assume that $v^\ell_j(z_2)\geq v^\ell_j(z_1)$.  Otherwise, the argument below works with the reverse geodesic $\gamma_{z_2,z_1}$.  We can estimate
\begin{align}
|\langle\nabla v^\ell_j,\nabla v^k_j\rangle|(z_1) &= \big|\fint_0^d \langle\nabla v^\ell_j,\nabla v^k_j\rangle - \fint_0^d\int_0^t \nabla_{\dot\gamma}\langle\nabla v^\ell_j,\nabla v^k_j\rangle\big| \, ,\notag\\
&= \big|\fint_0^d \langle \dot\gamma,\nabla v^k_j\rangle +\fint_0^d \langle\nabla v^\ell_j-\dot\gamma,\nabla v^k_j\rangle - \fint_0^d\int_0^t \nabla_{\dot\gamma}\langle\nabla v^\ell_j,\nabla v^k_j\rangle\big| \, ,\notag\\
&\leq C\Big(\,\Big|\frac{v^k_j(z_2)-v^k_j(z_1)}{d}\Big|+\fint_{\gamma_{z_1,z_2}} |\nabla v^\ell_j-\dot\gamma|+\int_{\gamma_{z_1,z_2}} |\nabla^2 v^\ell_j|+\int_{\gamma_{z_1,z_2}} |\nabla^2 v^k_j|\, \Big)\, .
\end{align}
To deal with the second term on the last line let us observe that since $|\nabla v^\ell_j|\leq 1+\epsilon_j$, we have
\begin{align}
\Big(\fint_{\gamma_{z_1,z_2}} |\nabla v^\ell_j-\dot\gamma|\Big)^2&\leq \fint_{\gamma_{z_1,z_2}} |\nabla v^\ell_j-\dot\gamma|^2\, \notag\\
&\leq \fint_{\gamma_{z_1,z_2}} 2\big(1-\langle\nabla v^\ell_j,\dot\gamma \rangle\big)+\epsilon_j\, \notag\\
&\leq 2\Big(1-\frac{v^\ell_j(z_2)-v^\ell(z_1)}{d}\Big)+\epsilon_j\, \notag\\
&\leq 2\Big|1-\frac{|v^\ell_j(z_2)-v^\ell(z_1)|}{d}\Big|+\epsilon_j\, ,
\end{align}
where we have used our normalizing condition that $v^\ell_j(z_2)\geq v^\ell(z_1)$ in the last line.  Plugging this into our estimate for $|\langle\nabla v^\ell_j,\nabla v^k_j\rangle|(z_1)$ we obtain
\begin{align}
|\langle\nabla v^\ell_j,\nabla v^k_j\rangle|(z_1)\leq C\Big(\,&\Big|\frac{v^k_j(z_2)-v^k_j(z_1)}{d}\Big|+\sqrt{\bigg|1-\frac{|v^\ell_j(z_2)-v^\ell(z_1)|}{d}\bigg|}\;+\notag\\
&+\int_{\gamma_{z_1,z_2}} |\nabla^2 v^\ell_j|+\int_{\gamma_{z_1,z_2}} |\nabla^2 v^k_j|\, \Big)+\epsilon_j\, .
\end{align}
Since this holds for each $z_2\in U(y,r,r_1,r_2)$ we can average both sides and use \eqref{e:k_new:1} to estimate
\begin{align}
|\langle\nabla v^\ell_j,\nabla v^k_j\rangle|(z_1) \leq O(\frac{r_2}{r_1})+C\fint_{U(s,y,r_1,r_2)}\Big(\int_{\gamma_{z_1,z_2}} |\nabla^2 v^\ell_j|+\int_{\gamma_{z_1,z_2}} |\nabla^2 v^k_j|\, \Big)dv_g(z_2)+\epsilon_j(R)\, .
\end{align}
Integrating over $z_1$ then gives us the estimate
\begin{align}
\fint_{U(s,y,r_1,r_2)}|\langle\nabla v^\ell_j,\nabla v^k_j\rangle|&\leq O(\frac{r_2}{r_1})+C\fint_{U\times U}\Big(\int_{\gamma_{z_1,z_2}} |\nabla^2 v^\ell_j|+\int_{\gamma_{z_1,z_2}} |\nabla^2 v^k_j|\, \Big)+\epsilon_j\, ,\notag\\
&\leq O(\frac{r_2}{r_1})+Cr_1\fint_{U(s,y,10r_1,10r_2)}\Big(|\nabla^2 v^\ell_j|+|\nabla^2 v^k_j|\, \Big)+\epsilon_j\, .
\end{align}
In the last line we have used a sharpening of the standard segment inequality, which takes into account that all the minimizing geodesics beginning and ending in $U(s,y,r_1,r_2)$ are contained in $U(s,y,10r_1,10r_2)$.  Given this the conclusion follows from the proof of the standard segment inequality.  Rewriting the above gives us
\begin{align}\label{e:k:2:2}
\int_{U(s,y,r_1,r_2)}|\langle\nabla v^\ell_j,\nabla v^k_j\rangle|&\leq \Big(O(\frac{r_2}{r_1})+\epsilon_j\Big)\Vol(U(s,y,r_1,r_2))+Cr_1\int_{U(s,y,10r_1,10r_2)}\Big(|\nabla^2 v^\ell_j|+|\nabla^2 v^k_j|\, \Big)\, ,
\end{align}

Now to complete the proof, let us choose for each $r_1,r_2$ fixed a covering 
\begin{align}
B_R(x_j)\subseteq \bigcup U(s_i,y_i,r_1,r_2)\, ,
\end{align}
such that the sets $U(s_i,y_i,10r_1,10r_2)$ overlap at most $C(n)$ times.  This is possible using the GH condition with $\epsilon_j<<r_2$.  By applying \eqref{e:k:2:2} to each of these and summing we obtain the estimate
\begin{align}
\int_{B_R(x_j)}|\langle\nabla v^\ell_j,\nabla v^k_j\rangle|&\leq \Big(O(\frac{r_2}{r_1})+\epsilon_j\Big)\Vol(B_{2R}(x_j))+Cr_1\int_{B_{2R}(x_j)}\Big(|\nabla^2 v^\ell_j|+|\nabla^2 v^k_j|\, \Big)\, ,
\end{align}
or that
\begin{align}
\fint_{B_R(x_j)}|\langle\nabla v^\ell_j,\nabla v^k_j\rangle|&\leq O(\frac{r_2}{r_1})+Cr_1\fint_{B_{2R}(x_j)}\Big(|\nabla^2 v^\ell_j|+|\nabla^2 v^k_j|\, \Big)+\epsilon_j(R)\notag\\
&\leq O(\frac{r_2}{r_1})+Cr_1+\epsilon_j(R)\, ,
\end{align}
where in the last line we have used that we have uniform $L^2$ estimates on the Hessians of $v^a_j$.  The estimates above hold for all $0<\epsilon_j<<r_2<<r_1<<1$.  To finish the proof let us now choose $r_{2,j}, r_{1,j}\to 0$ such that $\frac{r_{2,j}}{r_{1,j}},\frac{\epsilon_j}{r_{2,j}}\to 0$.  This proves the estimate \eqref{e:k:2:-2}, and therefore by \eqref{e:k:2:-1} we have that
\begin{align}
\fint_{B_R(x_j)}\big|\,|\nabla v_j^k|-1\big|\leq \epsilon_j(R)\to 0\, .
\end{align}
When we combine this with the $L^\infty$ estimate on $|\nabla v^k_j|$, this gives the $L^2$ estimate
\begin{align}\label{e:k_new:3}
\fint_{B_R(x_j)}\big|\,|\nabla v_j^k|^2-1\big|\leq \epsilon_j(R)\to 0\, .
\end{align}
  
Finally, since $v^k_j$ is harmonic we can now argue with Bochner's formula as in the proof of (\ref{kato_1}),
to obtain the  Hessian estimate, $\fint_{B_R(x_j)}|\nabla^2 v^k_j|^2\leq \epsilon_j(R)\to 0$.
This completes the proof of the claim.$\square$\\

 Now we can finish the proof of the Transformation theorem.
  Indeed, we will see that $v_j=A\circ u:B_1(x_j)\to \dR^k$
 is the desired $\epsilon_j(R)$-splitting.  Claim 6 gives
\begin{align}\label{e:k:8}
\fint_{B_R(x_j)}|\nabla^2 v^\ell_j|^2 \to 0\, ,
\end{align}
for all $1\leq \ell\leq k$, while \eqref{e:k:2:-2}, \eqref{e:k_new:2} and \eqref{e:k_new:3} imply
\begin{align}\label{e:k:9}
\fint_{B_R(x_j)}|\langle\nabla v^a_j,\nabla v^b_j\rangle -\delta^{ab}|\to 0\, .
\end{align}
To see that $v_j:B_1(x_j)\to \dR^k$ is an $\epsilon_j(R)$-splitting on $B_1(x_j)$, the last step is to show that 
$|\nabla v^k_j|\leq 1+\epsilon_j\to 1$.  However this follows immediately from  \eqref{e:k:8} 
and \eqref{e:k:9} by using precisely the same argument as in (\ref{e:sharp1})--(\ref{e:sharp3}).

Thus, for $j$ sufficiently large we see that $v_j:B_1(x_j)\to \dR^k$ is an $\epsilon$-splitting.  This is a contradiction, so the proof is complete.

\end{proof}

\section{Proof of Theorem \ref{t:slicing_intro}, the Slicing Theorem}
\label{s:slicing}

The goal of this section is to prove the Slicing Theorem (Theorem \ref{t:slicing_intro}).  Recall the statement:

{\it
For each $\epsilon>0$ there exists $\delta(n,\epsilon)>0$ such that if $M^n$ 
satisfies $\Ric_{M^n}\geq -(n-1)\delta$ and if $u:B_2(p)\to \dR^{n-2}$ is a harmonic 
$\delta$-splitting map, then there exists a subset $G_\epsilon\subseteq B_1(0^{n-2})$ 
which satisfies the following:
\begin{enumerate}
\item $\Vol(G_\epsilon)>\Vol(B_1(0^{n-2}))-\epsilon$.
\vskip1mm

\item If $s\in G_\epsilon$ then $u^{-1}(s)$ is nonempty.
\vskip1mm

\item For each $x\in u^{-1}(G_\epsilon)$ and $r\leq 1$ there exists a lower
triangular matrix
$A\in GL(n-2)$ with positive diagonal entries, 
such that $A\circ u:B_r(x)\to \dR^{n-2}$ is an $\epsilon$-splitting map.
\end{enumerate}
}

\begin{proof}[Proof of Theorem \ref{t:slicing_intro}.]
Recall from subsection \ref{ss:pst} the measure $\mu$ defined in (\ref{mudef}) and $\delta_3=\delta_3(n,\epsilon)$
 in the Transformation theorem; see the sentence prior to (\ref{cBdef}).
It was shown subsection \ref{ss:pst}
that in view of  Theorem \ref{t:hessian_estimate_intro} and the transformation theorem, 
 Theorem \ref{t:transformation_theorem_intro},  to complete the proof of the
Slicing theorem, it suffices to verify 
that for $1/4\geq r\geq s^{\delta_3}_x$,
$\mu$ satisfies the doubling condition $|u(B_{r}(x)|\leq  C(n)\cdot r^{-2}\mu(B_{r}(x))$
and the volume estimate $|u(B_{r}(x))|\leq  C(n)\cdot r^{-2}\mu(B_{r}(x)) $ on the image of a ball; see
(\ref{e:doubling_intro}), (\ref{e:intro4}).

\begin{lemma}
\label{l:doubling}
For each $x$ and
$1/4\geq r\geq s^{\delta_3}_x$ we have the doubling condition
\begin{equation}
\label{e:doubling}
\mu(B_{2r}(x))\leq C(n)\mu(B_r(x))\, .
\end{equation}
\end{lemma}
\begin{proof}
By Theorem \ref{t:transformation_theorem_intro}, there exists a 
lower triangular
matrix $A\in GL(n-2)$ with positive diagonal entries, such that 
\begin{align}
\label{e:mumuprime}
u'=A\circ u:B_{2r}(x)\to \dR^{n-2}
\end{align}
is an $\epsilon$-splitting.  Let $dv_g$ denote the Riemannian measure
and set $\omega'\equiv du'^1\wedge\cdots\wedge du'^{n-2}$. Define
the measure $\mu'$ by $\mu'= \Big(\int_{B_{3/2}(p)}|\omega|\Big)^{-1}|\omega'|dv_g$.  Then
\begin{align}
\mu' = \det(A)\mu\, .
\end{align}

In particular this gives us
\begin{align}\label{e:slice:4}
\frac{\mu'(B_{2r}(x))}{\mu'(B_r(x))} = \frac{\mu(B_{2r}(x))}{\mu(B_r(x))}\, ,
\end{align}
and it is equivalent to show the ratio bound for $\mu'$.  Now since $u'$ is an $\epsilon$-splitting 
we have the estimate
\begin{align}
&\fint_{B_{2r}(x)}|\, |\omega'|-1|\leq C(n)\epsilon\, .
\end{align}

Hence, we also have the estimate
\begin{align}
&\fint_{B_{r}(x)}|\, |\omega'|-1|\leq \frac{\Vol(B_{2r}(x))}{\Vol(B_r(x))}\fint_{B_{2r}(x)}|\, |\omega'|-1|\leq C(n)\epsilon\, ,
\end{align}
which of course uses the doubling property for the Riemannian measure.  By combining the previous
 two estimates we get
\begin{align}
\big(1-C\epsilon\big)\Vol(B_r(x))&\leq \mu'(B_r(x))\leq \big(1+C\epsilon\big)\Vol(B_r(x))\,  ,\notag\\
\big(1-C\epsilon\big)\Vol(B_{2r}(x))&\leq \mu'(B_{2r}(x))\leq \big(1+C\epsilon\big)\Vol(B_{2r}(x))\, .
\end{align}

Finally, by using the definition of $\mu'$, we arrive at:
\begin{align}
\mu'(B_{2r}(x)) &= \Big(\int_{B_{3/2}(p)}|\omega|\,dv_g\Big)^{-1}\int_{B_{2r}(x)}|\omega'|  \notag\\
&\leq (1+C(n)\epsilon)\Big(\int_{B_{3/2})}|\omega|\,dv_g\Big)^{-1}\Vol(B_{2r}(x)) \notag\\
&\leq C(n)\Big(\int_{B_{3/2})}|\omega|\,dv_g\Big)^{-1}\Vol(B_r(x))\\
&\leq C(n)\Big(\int_{B_{3/2}p)}|\omega|\,dv_g\Big)^{-1}\int_{B_r(x)}|\omega '| \notag\\
&= C(n)\mu'(B_r(x))\, ,
\end{align}
which by \eqref{e:slice:4} completes the proof.
\end{proof}



Recall the collection, $\cB_{\delta_3}$, of bad balls, defined in   (\ref{cBdef}).    
  The proof of the Slicing theorem (Theorem \ref{t:slicing_intro}) requires that the image under $u$ of  $\cB_{\delta_3}$,
 has measure $<\epsilon/2$; see (\ref{e:intro4}), (\ref{e:sum_intro}).
  If in (\ref{e:sum_intro}) the measure  $\mu$ were instead the 
  usual riemannian measure,
 then 
  since $u$ is Lipschitz, standard estimates could be used to show just that.  
  On the face of it, however, the $\mu$-content estimate is much weaker, since for balls where the determinant 
  $|\omega|$ of $u$ is small, then
$\mu(B_r(x))/\Vol(B_r(x))$ is small as well.

On the other hand, in the spirit of Sard's theorem, we will see in the next lemma 
that at least for balls $B_r(x)$ with $1/4\geq r\geq s^{\delta_3}_x$, we recover
 this loss because the volume of the image $u(B_r(x))$ is correspondingly small.  
\begin{lemma} 
\label{l:mu_vol_comparison}
If $1/4\geq r\geq s^\eta_x$, then
\begin{equation}
\label{e:intro41}
|u(B_{r}(x)|\leq  C(n)\cdot r^{-2}\mu(B_{r}(x)) \, .
\end{equation}
\end{lemma}
\begin{proof}
As in Lemma \ref{l:doubling}, choose a lower triangular matrix $A\in GL(n-2)$
with positive diagonal entries, such that 
\begin{align}
u'=A\circ u:B_{2r}(x)\to \dR^{n-2}
\end{align}
is an $\epsilon$-splitting and define the measure $\mu'$ as in Lemma \ref{l:doubling}.
Then as in (\ref{e:mumuprime}), $\mu' = \det(A)\mu$.

Since $u'$ is an $\epsilon$-splitting, we have the estimates
\begin{align}\label{e:slice:3}
&\fint_{B_{2r}(x)}||\omega'|-1|\leq C(n)\epsilon\, ,\notag\\
&u'(B_r(x))\subseteq B_{2r}(u'(x))\, .
\end{align}
By the first estimate above,
\begin{align}
\mu'(B_r(x)) &= \Big(\int_{B_{3/2}(p)} |\omega|\Big)^{-1}\int_{B_r(x)} |\omega'|\, ,\notag\\
&\geq (1-C(n)\epsilon)\frac{\Vol(B_r(x))}{\Vol(B_{3/2}(p))}\fint_{B_r(x)} |\omega'|\notag\\
&\geq (1-C\epsilon)\frac{\Vol(B_r(x))}{\Vol(B_{3/2}(x))}\geq C(n) r^n\, ,
\end{align}
where in the last step we have used volume monotonicity for the Riemannian measure.  On the other hand, by the second 
estimate of \eqref{e:slice:3},
\begin{align}
|u'(B_r(x))|\leq C(n)r^{n-2}\, .
\end{align}
Combining these gives the estimate
\begin{align}
|u'(B_r(x))|\leq C(n) r^{-2}\mu'(B_r(x))\, . 
\end{align}

To relate these back to the original function $u$, we observe that
\begin{align}
|u'(B_r(x))| &= \det(A)|u(B_r(x))|\, ,\notag\\
\mu'(B_r(x)) &= \det(A)|\mu(B_r(x))|\, ,
\end{align}
which immediately gives (\ref{e:intro41}).
This completes the proof.
\end{proof}


As previously noted, Lemmas \ref{l:doubling} and \ref{l:mu_vol_comparison} suffice to complete the proof of the
Slicing theorem.
\end{proof}

\section{Codimension $4$ Regularity of Singular Limits}\label{s:codim4}

In this section we  prove Theorem \ref{t:main_codim4}.  Thus, we consider a
Gromov-Hausdorff limit space, 
\begin{align}
(M^n_j,d_j,p_j)\stackrel{d_{GH}}{\longrightarrow}(X,d,p)\, ,
\end{align}
of a sequence of Riemannian manifolds $(M^n_j,g_j,p_j)$, satisfying $|\Ric_{M^n_j}|\leq n-1$ and $\Vol(B_1(p_j))>\rv>0$.  
We will show that there exists a subset $\cS\subseteq X$ of codimension $4$ such that $X\setminus \cS$ is a 
$C^{1,\alpha}$-Riemannian manifold.  In this section, we will show that $\cS$ has Hausdorff codimension 4.
 We will postpone the improvement to Minkowski codimension $4$ until Section \ref{s:qs_estimates}.

As mentioned in  Section \ref{s:intro}, it has been understood since \cite{ChC2} 
that the main technical challenge lies in showing that spaces of the 
form $\dR^{n-2}\times C(S^1_\beta)$, where $S^1_\beta$ is the circle of circumference $\beta\leq 2\pi$, cannot arise as limit spaces 
unless $\beta= 2\pi$ and hence $\dR^{n-2}\times C(S^1_\beta)= \dR^n$.  The  
 Slicing Theorem (Theorem \ref{t:slicing_intro}) was expressly designed to enable us to handle this 
point via a blow up argument.  We will do this  in subsection \ref{ss:codim2}.

 In subsection \ref{ss:codim3} we 
 prove that more general spaces of the form $\dR^{n-3}\times C(Y)$ cannot arise as limit spaces. 
 The proof of this statement, has a very different feel than the proof ruling out the codimension two limits, 
and essentially comes down to a bordism and curvature pinching argument for $3$-manifolds.  

Finally,
 in subsection \ref{ss:proof_codim4} we combine the tools developed in the 
previous subsections to prove the Hausdorff estimates
of Theorem \ref{t:main_codim4}.

\subsection{Nonexistence of Codimension $2$ Singularities}\label{ss:codim2}

In this subsection, we use the tools of Section \ref{s:slicing} in order to prove that spaces that are
$(n-2)$-symmetric
cannot arise as noncollapsed limits of manifolds with bounded Ricci curvature.

\begin{theorem}[$(n-2)$-Symmetric Limits]\label{t:codim2}
Let $(M^n_j,g_j,p_j)$ be a sequence of Riemannian manifolds satisfying 
$|\Ric_{M^n_j}|\to 0$, $\Vol(B_1(p_j))>\rv>0$ and such that
\begin{align}
(M_j^n,d_j,p_j) \stackrel{d_{GH}}{\longrightarrow} \dR^{n-2}\times C(S^1_\beta)\, .
\end{align}
Then $\beta=2\pi$ 
and $\dR^{n-2}\times C(S^1_\beta)=\dR^n$.
\end{theorem}

\begin{proof}[Proof of Theorem \ref{t:codim2}]
We will prove the result by contradiction. So let us assume it is false. Then there exists a sequence $(M^n_j,g_j,p_j)$ of 
Riemannian manifolds satisfying $|\Ric_{M^n_j}|\to 0 $, $\Vol(B_1(p_j))>\rv>0$ and such that
\begin{align}
(M_j^n,d_j,p_j) \to \big(\dR^{n-2}\times C(S^1_\beta),d,p\big)\, ,
\end{align}
with $\beta<2\pi$ and $p$ a vertex.  

Note first that by the noncollapsing assumption we have $\beta\geq \beta_0(n,v)$.  

Now by Lemma \ref{l:harmonic_splitting_intro}, there exists $\delta_j$-splitting maps 
$u_j:B_2(p_j)\to\dR^{n-2}$ with $\delta_j\to 0$.  Fix some sequence $\epsilon_j\to 0$ which is 
tending to zero so slowly compared to $\delta_j$,  that Theorem \ref{t:slicing_intro} holds
for $u_j:B_2(0)\to \dR^{n-2}$ with $\epsilon_j$.  Let $G_{\epsilon_j}\subseteq B_1(0^{n-2})$ be
 the corresponding good values of $u_j$, and let $s_j\in G_{\epsilon_j}\cap B_{10^{-1}}(0^{n-2})$
 be fixed regular values.  \\  

Observe that $\dR^{n-2}\times C(S^1_\beta)$ is smooth outside of the singular set 
$\cS= \dR^{n-2}\times\{0\}\subseteq \dR^{n-2}\times C(S^1_\beta)$.  In particular on
$\dR^{n-2}\times C(S^1_\beta)$ we have $r_h(x)\approx 1/d(x,\cS)$, where $r_h$ is the harmonic 
radius as in Section \ref{s:intro} and $d$ denotes distance.  By the standard $\epsilon$-regularity theorem, it follows
that the convergence of $M^n_j$ is in $C^{1,\alpha}\cap W^{2,q}$ away from $\cS$, for every 
$\alpha<1$ and $q<\infty$.  Let $f_j:B_{\epsilon^{-1}_j}(p)\to B_{\epsilon^{-1}_j}(p_j)$ be the 
$\epsilon_j$-Gromov Hausdorff maps, and let us denote $\cS_j\equiv f_j(\cS)\subseteq M^n_j$.  
Then by the previous statements, for every $\tau>0$, all $j$ sufficiently large, and 
$x\in B_1(p_j)\setminus T_\tau(\cS_j)$, we have $r_h(x)\geq \frac{\tau}{2}$.

Consider again the submanifold $u^{-1}_{j}(s_j)\cap B_1(p_j)$.  Define the scale
\begin{align}
r_j= \min\{r_h(x):x\in u^{-1}_{j}(s_j)\cap B_1(p_j)\}\, .
\end{align}
By the discussion of the previous paragraph, this minimum is actually obtained at some 
$x_j\in u^{-1}_{j}(s_j)\cap B_1(p_j)$, with $x_j\to \cS_j\cap B_{10^{-1}}(p_j)$. Moreover, since $S^1_\beta$,  the cross-section of
the cone factor,  satisfies $0<\beta<2\pi$, it follows  that $r_j\to 0$.  According to Theorem \ref{t:slicing_intro},
there exists a lower triangular matrix $A_j\in GL(n-2)$ with positive diagonal entries,
such that 
$v_j\equiv A_j\circ \big(u_j-s_j\big):B_{r_j}(x_j)\to \dR^{n-2}$ is an $\epsilon_j$-splitting map.
Note that we have renormalized so that each of our regular values is the zero level set.

Now let us consider the sequence $(M^n_j,r_j^{-1}d_j,x_j)$.  After passing to a subsequence if necessary, 
which we will continue to denote by $  (M^n_j,r_j^{-1}d_j,x_j)$,  have 
\begin{align}
(M^n_j,r_j^{-1}d_j,x_j)\stackrel{d_{GH}}{\longrightarrow} (X,d_X,x)\, ,
\end{align}
in the pointed Gromov-Hausdorff sense, where $X$ splits off $\dR^{n-2}$ isometrically.

By our noncollapsing 
assumption we have $\Vol(B_1(x_j))>c(n)\rv>0$, and hence, in the rescaled spaces, we have 
$\Vol(B_r(x_j))>c\rv r^n$ for all $r\leq R_j\to \infty$.  In particular, $X$ has Euclidean volume
growth at $\infty$ i.e. 
$\Vol(B_r(x'))>c\rv\,r^n$ for all $r>0$.  

After possibly passing to another subsequence, we can limit the functions $v_j$ to a function $v:X\to \dR^{n-2}$. 
 Note that by our normalization, we have
$v_j:B_{2}(x_j)\to \dR^{n-2}$ are $\epsilon_j$-splittings, and that by 
Theorem \ref{t:transformation_theorem_intro}, we have for each $R>2$ that the maps $v_j:B_{R}(x_j)\to \dR^{n-2}$ 
are $C(n,R)\epsilon_j$-splittings.  In particular, we can conclude that 
\begin{align}
X= \dR^{n-2}\times S\, ,
\end{align}
where $v:X\to\dR^{n-2}$ is the projection map and  $S= u^{-1}(0)$.  \\

Now by construction, in the rescaled spaces we have for any $y\in u^{-1}_j(0)$ that $r_h(y)\geq 1$.  
Therefore, the limit $X$ is $C^{1,\alpha}\cap W^{2,q}$ in a neighborhood of $u^{-1}(0)$, and hence 
$S= u^{-1}(0)$ is a nonsingular surface.  Thus, since $X=\dR^{n-2}\times S$ it follows  that $X$ is
 at least a $C^{1,\alpha}\cap W^{2,q}$ manifold with $r_h\geq 1$.  Since the Ricci curvature is uniformly 
bounded, in fact tending to zero, we have by the standard $\epsilon$-regularity theorem that the 
convergence $(M^n_j,r_j^{-1}d_j,x_j)\to (X, d_X,x)$ is in 
$C^{1,\alpha}\cap W^{2,q}$.  Because the convergence is in $C^{1,\alpha}\cap W^{2,q}$ we have that 
$r_h$ behaves  continuously in the limit; \cite{Anderson_Einstein}.  In particular, we have  $r_h(x'_j)\to r_h(x')$ and so, $r_h(x')=1$. 

On the other hand, since $|\Ric_{M^n_j}|\to 0$ and $X$ is $C^{1,\alpha}\cap W^{2,q}$ it follows that $X$ is a 
smooth Ricci flat manifold.  
This is easiest to see by writing things out in harmonic coordinates on $X$; see \cite{Anderson_Einstein} 
for the argument.  Now since $X=\dR^{n-2}\times S$, we can conclude that $S$ is smooth and Ricci flat, hence flat. 
 In particular, we have that $X$ is flat. Since we have already shown that
$X$ has Euclidean volume growth, this implies that $X=\dR^n$.
However, we have also already concluded that $r_h(x')=1$, which gives us our desired contradiction.
\end{proof}

We end this subsection with the following corollary, which states that a noncollapsed limit space is smooth away from 
a set of codimension $3$. We will use this in the next subsection to show $(n-3)$-symmetric splittings cannot arise as limits.

\begin{corollary}\label{c:codim3}
Let $(M^n_j,g_j,p_j)$ denote a sequence of Riemannian manifolds satisfying $|\Ric_{M^n_j}|\leq n-1$, $\Vol(B_1(p_j))>\rv>0$ and such that
\begin{align}
(M_j^n,d_j,p_j) \to (X,d,p)\, .
\end{align}
Then there exists a subset, $\cS\subseteq X$, with $\dim\cS\leq n-3$, such that for each $x\in X\setminus \cS$,
 we have $r_h(x)>0$.  In particular, $x\in X\setminus \cS$ is a $C^{1,\alpha}$ Riemannian manifold.
\end{corollary}
\begin{proof}
Recall the standard stratification of $X$.  In particular, if we consider the subset $\cS^{n-3}\subset X$ 
we have that $\dim \cS^{n-3}\leq n-3$, and that for every point $x\not\in \cS^{n-3}$ there exists {\it some} 
tangent cone at $x$ which is isometric to $\dR^{n-2}\times C(S^1_\beta)$.  That is, there exists $r_a\to 0$ such that
\begin{align}
(X,r_a^{-1}d,x)\to \dR^{n-2}\times C(S^1_\beta)\, .
\end{align}
However by Theorem \ref{t:codim2} we then have  $\beta= 2\pi$, which is to say that 
\begin{align}
(X,r_a^{-1}d,x)\to \dR^n\, .
\end{align}
Thus, for $a\in \dN$ sufficiently large, we can apply the standard $\epsilon$-regularity theorem, Theorem \ref{t:standard_epsilon_regularity},
to see that a neighborhood of $x$ is a $C^{1,\alpha}$ Riemannian manifold, which proves the corollary.
\end{proof}

\subsection{Nonexistence of Codimension $3$ singularities}\label{ss:codim3}

In this subsection we use the tools of Section \ref{s:slicing} and Section \ref{ss:codim2} in order to prove that 
$(n-3)$-symmetric metric spaces cannot arise as limits of manifolds with bounded Ricci curvature.  
Specifically, we prove the following:

\begin{theorem}[$(n-3)$-Symmetric Limits]\label{t:codim3}
Let $(M^n_j,g_j,p_j)$ denote a sequence of Riemannian manifolds satisfying $|\Ric_{M^n_j}|\to 0$, $\Vol(B_1(p_j))>\rv>0$ and such that
\begin{align}
(M_j^n,d_j,p_j) \to \dR^{n-3}\times C(Y)\, ,
\end{align}
in the pointed Gromov-Hausdorff sense, where $Y$ is some compact metric space.  
Then $Y$ is isometric to the unit $2$-sphere and hence $\dR^{n-3}\times C(Y)=\dR^n$.
\end{theorem}
\begin{proof}
Let us assume that this is not the case and study such a limit space $\dR^{n-3}\times C(Y)$.  
The first observation is that by Corollary \ref{c:codim3}, it follows that $Y$ is a smooth surface. 
 Indeed, if there were a point $y\in Y$ such that $r_h(y)=0$.  Then since $X= \dR^{n-3}\times C(Y)$ 
it would follow that there is a set of codimension at least $2$ such that $r_h\equiv 0$, 
which cannot happen by Corollary \ref{c:codim3}.\\

Since $Y$   a $C^{1,\alpha}\cap W^{2,q}$ manifold and  $|\Ric_{M^n_j}|\to 0$, 
it follows that $Y$ is a smooth Einstein manifold satisfying $\Ric_Y= g$.  
Because $Y$ 
is a surface, this means in particular that $Y$ has constant sectional curvature $\equiv 1$.  
Thus,  either $Y=\dR\dP^2$ or $Y= S^2$, the unit $2$-sphere, and  in the latter case we are done. 

So let us study the case $Y= \dR\dP^2$.  For $\epsilon>0$ small, choose 
$u_j:B_2(p_j)\to \dR^{n-3}$ to be an $\epsilon$-splitting as in Lemma \ref{l:harmonic_splitting_intro}.  
Note that away from the singular set, $\cS\equiv \dR^{n-3}\times \{0\}$, we have that the $M_j^n$ 
converge to $\dR^{n-3}\times C(Y)$ in $C^{1,\alpha}$.  If $f_j:B_2(p)\to B_2(p_j)$ denote the Gromov-Hausdorff 
maps, we put $\cS_j=f_j(\cS)$.  Then for $\tau>0$ small but fixed, we  have for $j$ 
sufficiently large, that on $B_1(p)\setminus T_\tau(\cS_j)$, the estimates $|\nabla u_j|>\frac{1}{2}$ and 
$|\nabla^2 u_j|\leq 1$ hold.  

Consider Poisson approximation $h_j$ 
to the square of distance
 function $d^2(x,p_j)$ on $B_2(p_j)$.  That is, $\Delta h_j = 2n$ and $h_j=1$ on $\partial B_2(p_j)$.  
We have (see for instance \cite{ChC1}) that $|h_j-d(\cdot,p_j)|\to 0$ uniformly in $B_2(p_j)$, and 
again, because the convergence is in $C^{1,\alpha}$, we have for $j$ sufficiently large that 
$|\nabla h|>\delta$ and $|\nabla^2 h|\leq 4n$ on $B_1(p_j)\setminus B_\tau(\cS_j)$.  Once again,
appealing to the $C^{1,\alpha}$ convergence,   for all $j$ sufficiently large and all 
$s\in B_1(0^{n-3})$ we have that $u^{-1}(s)\cap h^{-1}(1)$ is diffeomorphic to $\dR\dP^2$. By
Sard's theorem, there exists a regular value $s_j\in B_1(0^{n-3})$.  Then for $j$ sufficiently large, 
$u^{-1}_j(s_j)\cap\{h\leq 1\}$
is a smooth $3$-manifold, whose boundary is diffeomorphic to $\dR\dP^2$.  However, the second Stiefel-Whitney
number of $\dR\dP^2$ is nonzero, and in particular, $\dR\dP^2$ 
does not bound a smooth $3$-manifold.  This contradicts $Y= \dR\dP^2$.

\end{proof}

\subsection{Proof of Hausdorff Estimates of Theorem \ref{t:main_codim4}}
\label{ss:proof_codim4}

With Theorem \ref{t:codim3} in hand, the proof of Theorem \ref{t:main_codim4} becomes standard,
 and follows the same lines as the proof of Corollary \ref{c:codim3}.  Thus,  consider a sequence
\begin{align}
(M^n_j,d_j,p_j)\to (X,d,p)
\end{align}
of Riemannian manifolds satisfying $|\Ric_{M^n_j}|\leq n-1$ and $\Vol(B_1(p_j))>\rv>0$, which 
Gromov-Hausdorff converge to some $X$.  Recall again the standard stratification of $X$, which is 
reviewed in subsection \ref{ss:stratification}.  More specifically let us consider the closed stratum $\cS^{n-4}(X)\subseteq X$.  
On the one hand, we have from \cite{ChC1}
\begin{align}
\dim \cS^{n-4}\leq n-4\, .
\end{align}
On the other hand, we have that for every point $x\not\in \cS^{n-4}$, there exists {\it some} tangent cone at 
$x$ which is isometric to $\dR^{n-3}\times C(Y)$.  That is, for some sequence $r_a\to 0$ we have 
\begin{align}
(X,r_a^{-1}d,x)\to \dR^{n-3}\times C(Y)\, .
\end{align}
However, by Theorem \ref{t:codim3}, we have that $Y$ is isometric to the unit $2$-sphere, and hence,
\begin{align}
(X,r_a^{-1}d,x)\to \dR^n\, .
\end{align}
Then for $a\in \dN$ sufficiently large, we can apply the standard $\epsilon$-regularity theorem, Theorem \ref{t:standard_epsilon_regularity},
to see that $r_h(x)>0$, and thus, that a neighborhood of $x$ is a $C^{1,\alpha}$ Riemannian manifold.
 This proves the theorem.

\section{The $\epsilon$-regularity Theorem}\label{s:eps_reg}

In Section \ref{s:codim4}, we showed  that limit spaces satisfying our assumptions
 must be smooth away from a closed subset of codimension $4$.  
However, the strongest applications come from a more effective version of this statement. 
 In particular, the curvature estimates of Theorem \ref{t:main_estimate} and the Minkowski estimates of
 Theorem \ref{t:main_codim4} will require a more rigid statement.  Namely, in this section,
 we will prove the following:

\begin{theorem}\label{t:eps_reg}
There exists $\epsilon(n,\rv)>0$ such that if $M^n$ satisfies $|\Ric_{M^n}|\leq \epsilon$,  $\Vol(B_1(p))>\rv>0$.  
and
\begin{align}
d_{GH}\big(B_2(p),B_2(0)\big)<\epsilon\, ,
\end{align}
where $0$ is a vertex of the cone $\dR^{n-3}\times C(Y)$,
for some metric space $Y$, then we have 
\begin{align}
r_h(p)\geq 1\, .
\end{align}
Consequently, if $M^n$ is Einstein,  we have the bound
\begin{align}
\sup_{B_1(p)}|\Rm|\leq 1\, .
\end{align}
\end{theorem}

\begin{proof}
Given $n$ and $\rv>0$, assume no such $\epsilon$ exists.  Then there exists a sequence of spaces 
$(M^n_j,g_j,p_j)$ such that $|\Ric_{M^n_i}|\leq \epsilon_j\to 0$, $\Vol(B_1(p_j))>\rv>0$ and 

\begin{align}
d_{GH}\big(B_2(p_j),B_2(0_j)\big)<\epsilon_j\to 0\, ,
\end{align}
where $0_j\in \dR^{n-3}\times C(Y_j)$ is a vertex 
but $r_h(p_j)<1$.  After possibly passing to a 
subsequence,we  have
\begin{align}
B_2(p_j)\to B_2(0)\, ,
\end{align}
where $0\in \dR^{n-3}\times C(Y)\equiv X$ is a vertex. 
 But if $C(Y)$ has any point with $r_h(x)=0$, 
then there is a set of Hausdorff codimension $3$ in $X$ which is not smooth.  By the Hausdorff estimate 
of Theorem \ref{t:main_codim4} this is not possible, so it must be that $C(Y)$ is smooth.  Thus, $Y$ is a 
smooth manifold, and in fact, $C(Y)$ is itself be smooth if and only if $Y$ is the unit $2$-sphere. Thus,
\begin{align}
B_2(p_j)\to B_2(0^n)\subseteq \dR^n\, .
\end{align}
Now we can apply the standard $\epsilon$-regularity theorem, 
to conclude $r_h(p_j)\geq 1$, which is a contradiction.
\end{proof}

\section{Quantitative Stratification and Effective Estimates}\label{s:qs_estimates}
Having shown in Sections \ref{s:codim4} and \ref{s:eps_reg} that noncollapsed limits of
 Einstein manifolds are smooth away from a closed codimension $4$ subset, we will now give some applications. 
 In particular, we will use the ideas of quantiative stratification first introduced in \cite{CheegerNaber_Ricci} 
in order to improve the codimension estimates on singular sets of limit spaces
to curvature estimates on Einstein manifolds.  More precisely, in this section, 
we will prove Theorem \ref{t:main_estimate}.  We will also improve the Hausdorff dimension 
estimate of Theorem \ref{t:main_codim4} to a Minkowski dimension estimate.  One can view this as an easy corollary of Theorem \ref{t:main_estimate}. 

We begin here by reviewing the quantitative stratification and the main results on it from \cite{CheegerNaber_Ricci}. 
These will play a crucial role in our estimates.  In subsection \ref{ss:curv_einstein} we combine the main results concerning the 
quantitative stratification, stated in Theorem \ref{t:quant_strat}, with the $\epsilon$-regularity of Theorem \ref{t:eps_reg} 
in order to prove the main estimates on Einstein manifolds given in Theorem \ref{t:main_estimate}. 
 In subsection \ref{ss:einstein_harmonic_estimates} we apply the regularity results of Theorem \ref{t:main_estimate} 
in order to conclude stronger results about the behavior of harmonic functions on Einstein manifolds.\\

The idea of \cite{CheegerNaber_Ricci} was to make the notion of stratification more effective.  
The standard stratification, recalled in subsection \ref{ss:stratification}, is used to show that that most points 
have a lot of symmetry infinitesimally.  The quantitative stratification is used to show that most balls of a 
definite size have a lot of approximate symmetry.  In particular, the quantitative stratification
 introduced in \cite{CheegerNaber_Ricci} exists and gives nontrivial information even on a smooth manifold,
 whereas, on a smooth space, the standard stratification is always trivial.  This point is crucial to the 
proof of Theorem \ref{t:main_estimate}.  To make this precise we begin by defining a more local version of
approximate symmetry.

\begin{definition}
Given a metric space $Y$ with $y\in Y$, $r>0$ and $\epsilon>0$, we say that $y$ is $(k,\epsilon,r)$-symmetric
 if there exists a $k$-symmetric space $Y'$ such that $d_{GH}(B_{r}(y),B_{r}(y'))<\epsilon r$, where $y'\in Y'$ is a vertex.
\end{definition}

Recall from Section that $Y'$ is $k$-symmetric if $Y'=\dR^k\times C(Z')$.  To state the definition in words, we say 
that $Y$ is $(k,\epsilon,r)$-symmetric if the ball $B_r(x)$ looks very close to having $k$-symmetries.  
The quantitative stratification is then defined as follows:

\begin{definition}
For each $\epsilon>0$, $0<r<1$ and $k\in\dN$, define the closed quantitative $k$-stratum, $\cS^k_{\epsilon,r}(X)$, by
\begin{align}
\cS^k_{\epsilon,r}(X)\equiv \{x\in X:\text{ for no $r\leq s\leq 1$ is $x$ a $(k+1,\epsilon,s)$-symmetric point}\}\, .
\end{align}
\end{definition}

Thus, the closed  stratum $\cS^k_{\epsilon,r}(X)$ is the collection of points such that no ball of size at least $r$ is 
almost $(k+1)$-symmetric.  The first main result of \cite{CheegerNaber_Ricci} is to show that for manifolds
 which are noncollapsed and have lower Ricci curvature bounds, the set $\cS^k_{\epsilon,r}(X)$ is small in a
 very strong sense.  To say this a little more carefully, if one pretends that the $k$-stratum
is a well behaved $k$-dimensional submanifold, then one would expect the volume of the $r$-tube
 around the set to behave like $Cr^{n-k}$. Although we don't know this to be the case,
the following slightly weaker statement does hold.

\begin{theorem}[Quantitative Stratification, \cite{CheegerNaber_Ricci}]\label{t:quant_strat}
Let $M^n$ satisfy $\Ric\geq -(n-1)$ with $\Vol(B_1(p))>{\rm v}>0$.  Then for every $\epsilon,\eta>0$ 
there exists $C=C(n,\rv,\epsilon,\eta)$ such that
\begin{align}
\Vol\left(T_r\left(\cS^k_{\epsilon,r}(M)\cap B_1(p)\right)\right)\leq C r^{n-k-\eta}\, .
\end{align}
\end{theorem}
\begin{remark}
In \cite{CheegerNaber_Ricci}, the theorem is stated with $\epsilon\equiv \eta$, however it is easily seen to be equivalent to the statement above.
\end{remark}

\subsection{Proof of Theorem \ref{t:main_estimate}}\label{ss:curv_einstein}

In this subsection we combine Theorem \ref{t:eps_reg} and Theorem \ref{t:quant_strat}
 in order to prove Theorem \ref{t:main_estimate}.  
\begin{proof} (of Theorem \ref{t:main_estimate})
Let $(M^n,g,p)$ satisfy $|\Ric_{M^n}|\leq n-1$
 and $\Vol(B_1(p))>\rv>0$.  We will first show that for every $q<2$ there exists $C=C(n,\rv,q)>0$ such that
\begin{align}
\fint_{B_1(p)} r_h^{-2q} \leq C\, .
\end{align}
Simultaneously, we will show that if $M^n$ is Einstein, then this can be improved to
\begin{align}
\fint_{B_1(p)} r_x^{-2q} \leq C\, ,
\end{align}
where $r_x$ denotes the regularity scale at $x$.

Let $q<2$ be fixed and set $\eta=4-2q$, and let us consider Theorem \ref{t:quant_strat} 
with $\epsilon=\epsilon(n)>0$ chosen from Theorem \ref{t:eps_reg} and $\eta$ as above.  
Thus, there exists $C(n,\rv,q)$ such that
\begin{align}\label{e:curv_ein1}
\Vol(T_r(\{x\in \mathcal S^{n-4}_{\epsilon,2r}\cap B_1(p)\})) < Cr^{4-\eta}\, .
\end{align}

Note that by rescaling, we may regard the $\epsilon$-regularity theorem (Theorem \ref{t:eps_reg})
as stating that if $x$ is $(n-3,\epsilon,2r)$-symmetric, then $r_h> r$, and if $M^n$ is 
Einstein then $r_x>r$.  In fact, we have that if $x$ is $(n-3,\epsilon,s)$-symmetric for any 
$s\geq 2r$, then $r_h>r$.  Thus, if $x\not\in \mathcal S^{n-4}_{\epsilon,2r}$, then $r_h>\frac{s}{2}> r$.  
The contrapositive gives the inclusion
\begin{align}
\{x\in B_1(p):r_h\leq r\} \subseteq \mathcal S^{n-4}_{\epsilon,2r}\cap B_1(p)\, ,
\end{align}
which by (\ref{e:curv_ein1}) implies the desired estimate
\begin{align}
\Vol(T_r(\{x\in B_1(p):r_h\leq r\})) < Cr^{4-\eta}\leq Cr^{2q}\, .
\end{align}
If $M^n$ is Einstein, then Theorem \ref{t:eps_reg} allows us to replace $r_h$ with $r_x$, as claimed.\\

Now, for $q<2$, let us prove the $L^q$ bound on the curvature from Theorem \ref{t:main_estimate}.  
For this note that if $r_h(x)>r$ then by definition there exists a harmonic coordinate system,
$\Phi:B_r(0^n)\to M$ 
with $\phi(0)=x$ and such that 
\begin{align}
||g_{ij}-\delta_{ij}||_{C^0(B_r(0))}+r||\partial_k g_{i,j}||_{C^0(B_r(0))}<10^{-3}\, ,
\end{align}
where $g_{ij}=\Phi^*g$ is the pullback metric.  Since the Ricci curvature satisfies the 
bound $|\Ric_{M^n}|\leq n-1$, this implies that
\begin{align}
|\Delta_x g_{ij}| < C(n)r^{-2}\, ,
\end{align}
where $\Delta_x$ denotes the Laplacian written in coordinates.
Then for every $\alpha<1$ and 
$s<\infty$, we have the scale invariant estimates
\begin{align}
&r^{1+\alpha}||\partial_k g_{ij}||_{C^\alpha(B_{\frac{3r}{4}}(0))}\leq C(n,\alpha)\, ,\notag\\
&r^{2}||g_{i,j}||_{W^{2,s}(B_{\frac{3r}{4}}(0))}\leq C(n,s)\, .
\end{align}

In particular, applying this to $s=q$ we get 
\begin{align}\label{e:main_est:1}
r^{2q}\fint_{B_{r/2}(x)}|\Rm|^q\leq C(n)r^{2q}\fint_{B_{3r/4}(0)}|\Phi^*\Rm|^q < C(n,q)\, .\\\notag
\end{align}

Put $\eta =2-q$.  Then $q+\frac{\eta}{2}<2$.  Then we have already shown that
\begin{align}\label{e:main_est:2}
\Vol(T_r(\{x\in B_1(p):r_h\leq r\})) < Cr^{2q+\eta}\, ,
\end{align}
for $C(n,\rv,q)>0$. Consider the covering $\{B_{r_h(x)}(x)\}$ of $B_1(p)$, and 
a subcovering $\{B_{r_j}(x_j)\}$ by mutually disjoint balls, such that
\begin{enumerate}
\item $B_1(p)\subseteq \bigcup B_{r_j}(x_j)$ with $r_j=\frac{1}{2}r_h(x)$.
\vskip1mm

\item $\{B_{r_j/4}(x_j)\}$ are disjoint.
\end{enumerate}

By using \eqref{e:main_est:2}, we see that for each $\alpha\in\dN$, we have
\begin{align}
\sum_{2^{-\alpha-1}<r_j\leq 2^{-\alpha}} \Vol(B_{r_j}(x_j)) \leq C r_j^{2q+\eta} = C\,r_j^{2q}\, 2^{-\eta\alpha}\, .
\end{align}
By summing over $\alpha$, this gives 
\begin{align}
\sum r_j^{-2q}\Vol(B_{r_j}(x_j))\leq C\sum 2^{-\eta\alpha}\leq C(n,\rv,q)\, .
\end{align}
Finally, combining this with \eqref{e:main_est:1} we get 
\begin{align}
\fint_{B_1(p)}|\Rm|^q &\leq \rv^{-1}\sum \int_{B_{r_j}(x_j)}|\Rm|^q \notag\\
&\leq C(n,\rv,q)\sum r_j^{-2q}\Vol(B_{r_j}(x_j))\leq C(n,\rv,q)\, , 
\end{align}
which finishes the proof of Theorem \ref{t:main_estimate}. 
\end{proof}

\subsection{$L^q$-Estimates for Harmonic Functions on Einstein Manifolds}
\label{ss:einstein_harmonic_estimates}

In this subsection we give some applications of Theorem \ref{t:main_estimate}.  In particular, 
we obtain Sobolev bounds for harmonic functions and solutions of more general equations on 
manifolds with bounded Ricci curvature.  As we have used repeatedly, given a lower bound on 
Ricci curvature, there is a definite $L^2$ bound on  the Hessian of a harmonic function; see  (\ref{kato_1}).
 However, the example of a rounded off $2$-dimensional cone shows that one does not
have definite $L^q$ bounds for any $q>2$; see Example \ref{sss:Example3_ConeSpace}.
 In this subsection, we will see that the situation is better for noncollapsed spaces with bounded Ricci curvature.  
Namely, one can obtain $L^q$ bounds on the Hessians of such harmonic functions for all $q<4$.  
More generally, we show the following: 

\begin{theorem}\label{t:Lp_bounds_hessian}
For every $q<4$ there exists $C=C(n,\rv,q)$ such that if $M^n$ satisfies $|\Ric_{M^n}|\leq n-1$ and $\Vol(B_1(p))>\rv>0$ and
 $u:B_2(p)\to \dR$ satisfies 
$$|u|\leq 1\, ,$$
$$|\Delta u|\leq 1\, ,$$ 
then for every $q<4$ 
\begin{align}
\fint_{B_1(p)} |\nabla^2 u|^q \leq C\, .
\end{align}
\end{theorem}
\begin{remark}
It follows from the $L^2$ curvature estimate (\ref{e:main_dimfour:1}) of
 Theorem \ref{t:main_dimfour} and \cite{Cheeger}, that in dimension $4$, we actually have a full $L^4$ bound on $|\nabla^2 u|$.  If Conjecture \ref{conj:L2} is correct, then this holds in all dimensions.
\end{remark}

\begin{proof}
Let us first note that by a Green's function estimate, we have
\begin{align}\label{e:Lp_bound_hessian:1}
\sup_{B_{3/2}(p)} |\nabla u| \leq C(n,v)\, .
\end{align}
Indeed, for $x\in B_{2}(p)$ we can write
\begin{align}
u(x) = h(x)+\int_{B_2(p)} G(x,y)\Delta(u-h)\,dv_g(y) = h(x)+\int_{B_2(p)} G(x,y)\Delta u\,dv_g(y)\, ,
\end{align}
where $h$ is a harmonic function with $h\equiv u$ on $\partial B_2(p)$.  Standard estimates, see \cite{SY_Redbook}, on the Green's function on spaces with lower Ricci bounds gives us $|\nabla_x G(x,y)|\leq C(n,v)d(x,y)^{1-n}$ in our domain, and since $h$ is a bounded harmonic function we have $|\nabla h|\leq C(n)$ on $B_{3/2}(p)$.  Combining these gives us \eqref{e:Lp_bound_hessian:1}.  In fact, with a little more work one can drop the volume dependence in the estimate, though it makes no difference for our purposes since this is not true for the hessian estimate.

Now using Theorem \ref{t:main_estimate} we know for each $\epsilon>0$ that
\begin{align}\label{e:Lp_hessian:1}
\Vol(T_r(\{x\in B_1(p): r_h(x)\leq r\}))\leq C_\epsilon(n,\rv,\epsilon)r^{4-\epsilon}\, .
\end{align}
In particular, let us consider the sets
\begin{align}
\cC_\alpha \equiv \{x\in B_1(p):r_{\alpha}\leq r_h(x)\leq r_{\alpha-1}\}\, ,
\end{align}
where $r_\alpha\equiv 2^{-\alpha}$.  For the set $\cC_\alpha$,
we have the cover $\{B_{r_\alpha}(x)\}_{x\in \cC_\alpha}$. We can choose a finite subcovering $\{B_{r_\alpha/2}(x_i)\}_1^{N_\alpha}$ such that the balls $B_{r_\alpha/8}(x_i)$ are mutually disjoint. 
Using \eqref{e:Lp_hessian:1} we have 
\begin{align}
N_\alpha \leq C_\epsilon r_\alpha^{4-n-\epsilon}\, .
\end{align}

On each ball $B_{r_\alpha/2}(x_j)$ we can use standard elliptic estimates along with the
 gradient bound $|\nabla u|\leq C(n)$ to get the scale-invariant estimate
\begin{align}
r_\alpha^q\fint_{B_{r_\alpha/2}(x_j)}|\nabla^2 u|^q \leq C(n,\rv,q)\, ,
\end{align}
for any $q<\infty$.  In particular, if we choose $q<4$ and pick $\epsilon = \frac{4-q}{2}$, then we have 
\begin{align}
\int_{B_{r_\alpha/2}(x_j)}|\nabla^2 u|^q \leq C(n,\rv)\,r_\alpha^{n-4+2\epsilon}\, .
\end{align}
Combining this with \eqref{e:Lp_hessian:1} gives
\begin{align}
\int_{\cC_\alpha}|\nabla^2 u|^q \leq C(n,\rv)\,r_\alpha^{n-4+2\epsilon}\cdot N_\alpha\leq C(n,\rv,q)\,r_\alpha^\epsilon\, .
\end{align}
Finally, by summing  over $\cC_\alpha$ we get the estimate
\begin{align}
\int_{B_1(p)}|\nabla^2 u|^q \leq C(n,\rv,p)\lambda^q\,\sum_\alpha r_\alpha^\epsilon = C\,\sum 2^{-\epsilon\alpha} = C(n,\rv,q)\, ,
\end{align}
as claimed.
\end{proof}
\vspace{1 cm}

\section{Improved Estimates in Dimension 4}\label{s:dimension_four}

In this section we apply the codimension $4$ estimates of Theorem \ref{t:main_codim4} 
to prove the finite diffeomorphism and $L^2$ curvature bounds of Theorem \ref{t:main_dimfour} and 
Theorem \ref{t:main_finite_diffeo}.

In subsection \ref{ss:prelim}, we  recall some necessary preliminaries.

In subsection  \ref{ss:annulus}, we use the codimension $4$ estimate of Theorem \ref{t:main_codim4} to prove
the existence of {\it good} annuli which have curvature and harmonic radius control.  

In subsection  \ref{ss:regularity_scale_dimfour} we first use this to show that in the noncollapsed situation, 
at any point, away from a definite number of scales, {\it every} annulus is good.  
We combine this with a counting argument, which plays the role of an effective version of the fact any
 infinite collection of points has a limit point, in order to prove the harmonic radius estimates of 
Theorem \ref{t:main_estimate}. 

In subsection \ref{ss:finite_diffeo_dimfour} we prove the finite diffeomorphism statement of
 Theorem \ref{t:main_finite_diffeo}.  
Morally, the argument is quite similar to the one in \cite{Anderson-Cheeger}, 
though it is designed to be more effective in nature.  
In fact, the argument in Section \ref{ss:finite_diffeo_dimfour} is quite general and works for any collection 
of uniformly noncollapsed smooth manifolds with bounded Ricci curvature, such that all
 Gromov-Hausdorff limits and blow ups have only isolated singularities.

In subsection \ref{ss:L2_curvature_dimfour}, we give a local version of the finite diffeomorphism theorem.  
Our main application of this is to prove {\it a priori} $L^2$ estimates on the curvature on a noncollapsed  $4$-manifold 
with bounded Ricci curvature.

\subsection{Diffeomorphisms and Harmonic Radius}
\label{ss:prelim}

To control the diffeomorphism type of a manifold, or of part of a manifold, 
the basic tool one needs is to control the total number of coordinate charts, the
number of domains these charts which can intersect a given chart and the 
change of coordinate maps between these charts in a suitably strong topology.  
This type of result 
has a long history, going back to \cite{Cheeger_finiteness} in the context of bounded sectional curvature.  
In particular, control on the harmonic radius
 enables one to implement such an argument.

In this subsection we recall two theorems that will be used later.
 We refer the reader to the book \cite{Petersen_RiemannianGeometry} for proofs of these statements. 
 The first theorem states that when two manifolds with harmonic radius bounded from below 
are sufficiently Gromov-Hausdorff close, then they must be diffeomorphic.

\begin{theorem}\label{t:GH_implies_diffeo1}
For every $\epsilon>0$, there exists $\delta=\delta(n,\epsilon)$, such that the following holds.
 If $M^n_1,M^n_2$ are Riemannian manifolds 
and  $U_j\subset M_j$ are subsets such that $r_h(x)>r>0$ for each $x\in U_j$,  
 and $$d_{GH}(B_r(U_1),B_r(U_2))<\epsilon r\, ,$$ then there exist open sets 
$B_{r/2}(U_j)\subseteq U'_j\subseteq B_{r}(U_j)$ and a $C^2$ diffeomorphism 
$\Phi:U'_1\to U'_2$, such that 
\begin{align}
||g_1-\Phi^*g_2||_{C^0}<\epsilon\, .
\end{align}
If we further assume $|\Ric_{M^n_j}|\leq n-1$, $j=1,2$, then  $\Phi$ is
 in $C^{2,\alpha}\cap W^{3,q}$ for all $\alpha<1$ and $q<\infty$, and in
 harmonic coordinates on $U'_1$ we have 
\begin{align}
||g_1-\Phi^*g_2||_{C^0}+r^{1+\alpha}||\partial_i\Phi^*g_2||_{C^\alpha}+r^{2}||\partial_i\partial_j\Phi^*g_2||_{L^q}
\leq C(n,\alpha,q)\epsilon\, .
\end{align}
\end{theorem}

The idea of the proof of Theorem \ref{t:GH_implies_diffeo1}
 is to cover the set $U_1$ by harmonic charts $B_{r/2}(x_j)$ of definite size, the intersection
of whose domains also have a definite size or are empty and such that each chart domain intersects
at most a definite number of distinct chart domains.
By restricting the Gromov-Hausdorff map $f:U_1\to U_2$ to $U_1$, and using that the image of 
each ball $f(B_r(x_j))$ lies in a harmonic coordinate chart of $U_2$, we can construct a suitable smooth 
approximation of $f$.  Then using the estimates of the local charts one can see this smoothing 
of $f$ is  the 
required diffeomorphism.

In a related direction, instead of trying to use the harmonic radius to directly to construct 
diffeomorphisms between nearby manifolds, we can use it to simply bound the number 
of diffeomorphism types of a space.  Precisely, we have the following:

\begin{theorem}\label{t:harmonic_rad_finite_diffeo}
There exists $C=C(n,D)$  with the following property. Let $(M^n,g)$ denote
 a Riemannian manifold and $U\subseteq M$  a subset such 
that $r_h(x)>r>0$ for all $x\in U$ and such that $\diam (U)\leq D\cdot r$.  Then there exists 
an open set $U'$ with $T_{r/2}(U)\subseteq U'\subseteq T_r(U)$, such that $U'$ has at
 most one of $C$ diffeomorphism types. 
\end{theorem}

The idea of the proof of the above is that $U$ may be covered by a controlled number 
of harmonic charts $B_{r/2}(x_j)$ with suitable control as above on the intersections of their domains.
  The geometry estimates on the charts automatically
 imply control over the transition functions between these charts. Hence there are a 
finite number of ways this finite collection of balls can be pasted together.  

\subsection{Annulus Estimates}\label{ss:annulus}

In this section, we use Theorem \ref{t:main_codim4} to prove our basic annulus 
estimates on $4$-manifolds with bounded Ricci curvature.  These estimates are the key first steps towards
proving  the 
finite diffeomorphism statements and the corresponding curvature estimates of Theorem \ref{t:main_dimfour}.  
To state our main result for this subsection let us recall the volume ratio
\begin{align}
\cV^\delta_r(x):=-\ln\left(\frac{\Vol(B_r(x))}{\Vol(B_r(0^4_{-\delta}))}\right)\, ,
\end{align}
where $0^4_{-\delta}$ is a base point in the $4$-dimensional hyperbolic space of constant curvature $-\delta$;
by the Bishop-Gromov theorem, this ratio is monotone {\it increasing} for a manifold with Ricci curvature bounded from below
 $\Ric_{M^n}\geq -3\delta$.   It has been understood since \cite{ChC1} that 
almost constancy of  $\cV^\delta_r(x)$
 over a range of 
scales leads to cone behavior of the underlying metric space.  Our main result of this subsection states that in the context of bounded Ricci curvature and dimension $4$, 
almost constancy 
 of this volume ratio leads to much stronger control up to diffeomorphism and pointwise geometric control.

\begin{theorem}\label{t:annulus}
 For every $\epsilon>0$ there exists $\delta(\rv,\epsilon)>0$ 
such that if $M^4$ satisfies $|\Ric_{M^4}|\leq 3\delta$, $\Vol(B_1(p))>\rv>0$ and $|\cV^\delta_{4}(p)-\cV^\delta_{1/4}(p)|<\delta$, then there exists a 
discrete subgroup $\Gamma\subseteq {\rm O}(4)$, unique up to conjugacy,
 with $|\Gamma|\leq N(\rv)$ such that the following hold:
\begin{enumerate}
\item For each $x\in A_{\epsilon,2}(p)$ we have the harmonic radius lower bound $r_h(x)>r_0(\rv)\epsilon$.
\vskip1mm

\item There exists a subset $A_{\epsilon,2}(p)\subseteq U\subseteq A_{\epsilon/2,2+\epsilon}(p)$ and a 
diffeomorphism $\Phi:A_{\epsilon,2}(0)\to U$, with $0\in \dR^4/\Gamma$, such that
if $g_{ij}=\Phi^*g$ is the pullback metric then
\begin{align}
||g_{ij}-\delta_{ij}||_{C^{0}} + ||\partial_k g_{ij}||_{C^{0}}<\epsilon\, .
\end{align}
\end{enumerate}
\end{theorem}

\begin{proof}
 The proof is by contradiction.  So let us assume for some $\epsilon>0$ there is no such $\delta(\rv,\epsilon)>0$.  
Thus, we have a sequence of spaces $(M^4_j,g_j,p_j)$ with $\Vol(B_1(p_j))>\rv>0$, $|\Ric_{M^4_j}|\leq \delta_j\to 0$ and 
$|\cV_{4}(p_j)-\cV_{1/4}(p_j)|<\delta_j\to 0$, but the conclusions of the theorem fail.  After passing to a subsequence 
we can take a limit
\begin{align}
(M^4_j,d_j,p_j)\stackrel{d_{GH}}{\longrightarrow}(X,d,p)\, .
\end{align}
Using the almost volume cone implies almost metric cone theorem of \cite{ChC1}, we then have 
\begin{align}
B_4(p) = B_4\big(y_0)\, ,
\end{align}
where $y_0\in C(Y)$ is the cone vertex and $Y$ some metric space of diameter $\leq\pi$.  \\

Now using Theorem \ref{t:main_codim4}, we know that away from a set of codimension $4$ in $C(Y)$,  the 
harmonic radius $r_h>0$ is bounded uniformly from below.  Assume there is some point $y\in Y$ such that $r_h(y)=0$ 
and consider the ray $\gamma_y$ in $C(Y)$ through the point $y$. In that case, it would follow that for
 every point of $\gamma_y$, the harmonic radius $r_h=0$ vanishes.  The ray $\gamma$ 
has Hausdorff dimension $1$, and therefore its
existence would contradict Theorem \ref{t:main_codim4}.  
Thus, we conclude that $r_h>0$ and that $Y=(Y,g_Y)$ 
is a $C^{1,\alpha}\cap W^{2,q}$ manifold for every $\alpha<1$ and $q<\infty$.  

Now by writing the formula for the Ricci tensor in harmonic coordinates and using
$|\Ric_{M^4_j}|\to 0$, it follows that $C(Y)$ is smooth and Ricci flat away from the vertex.  In particular, since $C(Y)$
 is a metric cone over $Y$, we must  $\Ric_{Y^3} = 3g^Y$.  Since in dimension $3$, constant Ricci curvature
 implies constant sectional curvature, it follows  $Y=S^3/\Gamma$ has constant sectional curvature $\equiv 1$.
  Additionally, we know from the volume bound, $\Vol(B_1(p))>\rv>0$, 
that the order $|\Gamma|<N(\rv)$ is uniformly bounded.  
In particular, we have that $C(Y)=\dR^4/\Gamma$ is an orbifold with an isolated singularity. 

 It now follows that there exists 
$r_0(\rv)>0$ such that for $y\in \dR^4/\Gamma$ with $|y|=1$, we have
\begin{align}
B_{2r_0}(y) = B_{2r_0}(0^4)\, ,
\end{align}
where $0^4\in \dR^4$.  In particular, for all $j$ sufficiently large, we have from the standard $\epsilon$-regularity theorem, 
Theorem \ref{t:standard_epsilon_regularity}, that for all $x\in A_{\epsilon,2}(p_j)$, the harmonic radius,
$r_h(x)>r_0(\rv,\epsilon)=r_0(\rv)\epsilon$ is bounded uniformly from below independent of $j$.  
Thus, if there exists $\epsilon$ as above, for which there is no $\delta(\rv,\epsilon)$, it must be (2) that fails to hold.

However, by  using again the diffeomorphism statement of Theorem \ref{t:GH_implies_diffeo1},
we have that for $j$ sufficiently large, there exists diffeomorphisms
\begin{align}
\Phi_j:A_{\epsilon,2}(0)\to M^4_j\, ,
\end{align}
such that 
\begin{align}
\Phi_j^*g_j\stackrel{C^{1,\alpha}\cap W^{2,q}}{\longrightarrow} dr^2+r^2g_Y\, .
\end{align}
For $j$ sufficiently large, this implies that (2) holds; a contradiction.
\end{proof}

\subsection{Regularity Scale Estimates}\label{ss:regularity_scale_dimfour}

In this subsection we prove the harmonic and regularity scale estimates \eqref{e:main_dimfour:2} of 
Theorem \ref{t:main_dimfour}.  
We know already from Theorem \ref{t:main_codim4} that if $M^4\to X$ is a limit space, then the 
singular set of $X$ has dimension zero.  
The estimate \eqref{e:main_dimfour:2} may be viewed as an effective version of this statement. 
 Indeed, \eqref{e:main_dimfour:2} 
not only gives a bound on the number of singularities which can appear, but it gives a bound on
 the number of balls with large 
curvature concentration.  Motivated by Theorem \ref{t:annulus} and the constructions of
 \cite{CheegerNaber_Ricci}, we begin with 
the following definition which will be useful in subsequent sections as well.

\begin{definition}\label{d:scale_tuple}
Consider the scales $r_\alpha= 2^{-\alpha}$.  For each $x\in M$ we associate the 
infinite tuple $T(x)\in \dZ_2^\dN$ defined by
\begin{align}
T_\alpha(x)\equiv 
\begin{cases}
1 \text{  if }|\cV^\delta_{4r_\alpha}(x)-\cV^\delta_{r_\alpha/4}(x)|\geq \delta\,\notag\\
0 \text{  if }|\cV^\delta_{4r_\alpha}(x)-\cV^\delta_{r_\alpha/4}(x)|< \delta\, .
\end{cases}
\end{align}
We denote by $|T|(x)= \sum T_\alpha(x)$ the number of {\it bad} scales at $x\in M^4$.\\
\end{definition}
\begin{remark}
The definition of $T(x)$ relies on a choice of $\delta>0$.  When we want to stress this, we will write
 $T^\delta(x)$, but otherwise will supress this dependence. 
\end{remark}

We begin with the following; see also \cite{CheegerNaber_Ricci} for the same statement in a more general context:

\begin{lemma}\label{l:bad_scale_estimate}
Let $\Ric_{M^4}\geq -3\delta$ and $\Vol(B_1(p))>\rv>0$ with $\delta\leq 1$.  
Then for each $\delta'>0$ and $x\in B_2(p)$ there exists at most $N(\rv,\delta')$ 
scales $\alpha\in \dN$ such that
\begin{align}
\big|\cV^\delta_{r_{\alpha+1}}(x)-\cV^\delta_{r_{\alpha}}(x)\big|>\delta'\, .
\end{align}
\end{lemma}

\begin{proof}
For $x\in B_2(p)$ fixed, we have
\begin{align}
\Vol(B_1(x))\geq C(n)^{-1}\Vol(B_3(x))\geq C^{-1}\Vol(B_1(p))\geq C^{-1}\rv>0\, ,
\end{align}
and so,
\begin{align}
\cV^\delta_1(x)\leq -\ln\Big(C^{-1}\rv\Big)=C(n,\rv)\, .
\end{align}
The monotonicity of $\cV^\delta_r(x)$ gives
\begin{align}
C(n,\rv)-1\geq \cV^\delta_1(x)-\cV^\delta_0(x) &= \sum \Big(\cV^\delta_{r_\alpha}(x)
-\cV^\delta_{r_{\alpha+1}}(x)\Big)\notag\\
&=\sum \Big|\cV^\delta_{r_\alpha}(x)-\cV^\delta_{r_{\alpha+1}}(x)\Big|\, .
\end{align}
In particular, there are at most $N= C(n,\rv)(\delta')^{-1}$ elements $\alpha\in \dN$ such that
\begin{align}
\Big|\cV_{r_\alpha}(x)-\cV_{r_{\alpha+1}}(x)\Big|>\delta'\, ,
\end{align}
as claimed.
\end{proof}

Let us point out the following useful corollary:

\begin{corollary}\label{c:bad_scale_estimate}
Let $M^4$ satisfy $\Ric_{M^4}\geq -3\delta$ and $\Vol(B_1(p))>\rv>0$. 
 Then for each $x\in B_2(p)$ we have 
\begin{align}
|T^\delta|(x)\leq 
N(\rv,\delta)\, .
\end{align}
\end{corollary}
\begin{remark}
\end{remark}
\begin{proof}
Put $\delta'=\delta/3$.  Then for $x\in B_2(p)$, there are at most $N(\rv,\delta)=C(\rv)\delta^{-1}$ 
scales $\alpha$ for which
\begin{align}
\Big|\cV_{r_\alpha}(x)-\cV_{r_{\alpha+1}}(x)\Big|>\frac{\delta}{3}\, .
\end{align}
Hence, there are at most $3N$ elements $\alpha\in \dN$ such that
\begin{align}
\Big|\cV_{r_\beta}(x)-\cV_{r_{\beta+1}}(x)\Big|>\frac{\delta}{3}\, ,
\end{align}
for some $\beta\in \{\alpha-1,\alpha,\alpha+1\}$.  Therefore, for all other $\alpha$, we must  have 
\begin{align}
\Big|\cV_{4r_\alpha}(x)-\cV_{r_{\alpha}/4}(x)\Big|<\delta\, ,
\end{align}
which proves the corollary.
\end{proof}

\begin{remark} 
In fact, both the lemma and the corollary work in all dimensions.
\end{remark}

We end this subsection with a proof of the regularity scale estimate \eqref{e:main_dimfour:2} from 
Theorem \ref{t:main_dimfour}.  One can view the proof as an effective version of the fact that an infinite 
collection of points must have a limit point.\\

\begin{proof}[Proof of Estimate \eqref{e:main_dimfour:2} of Theorem \ref{t:main_dimfour}]

Let $M^n$ satisfy $|\Ric_{M^4}|\leq 3$ and $\Vol(B_1(p))>\rv>0$.  We will prove the estimate for the
 harmonic radius $r_h$. The same argument works  in the Einstein case to control the regularity scale.

So let $0<\epsilon<<1$ be fixed
with $\delta(n,\epsilon)$ chosen to satisfy Theorem \ref{t:annulus}.  Consider the set
\begin{align}
\{x\in B_1(p): r_h(x)<r\}\, .
\end{align}
In view of the doubling condition implied by the Bishop-Gromov inequality, we have by a standard construction that
there exists a covering $\{B_r(x_j)\}_1^N$ with
\begin{align}
\{x\in B_1(p): r_h(x)<r\}\subseteq \bigcup_j B_{r}(x_j)\, ,
\end{align}
but such that $\{B_{r/4}(x_j)\}$ are disjoint.  Such coverings, which we  will term ``efficient'',
will be constructed several 
times below.
Note that
\begin{align}
T_r\big(\{x\in B_1(p): r_h(x)<r\}\big)\subseteq \bigcup_j B_{2r}(x_j)\, ,
\end{align}
and thus
\begin{align}
\Vol\Big(T_r\big(\{x\in B_1(p): r_h(x)<r\}\big)\Big)\leq \sum_1^N \Vol(B_{2r}(x_j))\leq C(n)N\cdot r^4\, .
\end{align}
Hence, our goal is to control the number of balls $N$ in the covering.  Denote by $\cC\equiv\{x_j\}_1^N$ 
the corresponding  collection of centers.\\

Now note the following:  if $x_j$ is one of our ball centers and $T_\alpha(x_j)=0$, then by Theorem \ref{t:annulus},
we have for every $x\in A_{r_\alpha/2,2r_\alpha}(x_j)$ that $r_h(x)>\bar r(\rv)\cdot r_\alpha$.  In particular, if 
$r_\alpha>\bar r^{-1}r$,  this implies that
\begin{align}\label{e:main_four_regscale:1}
x_k\not\in  A_{r_\alpha/2,2r_\alpha}(x_j) \, ,
\end{align}
for any other ball center $x_k$.

Now let us inductively build a sequence of decreasing subsets 
$\cC^{k+1}\subseteq \cC^k\subseteq \cdots \subseteq \cC$ and associated radii $s_k=r_{\alpha_k}>0$ with 
$\diam (\cC^k)<4 s_k$.  There are three key inductive properties that will be proved about these sets:
\begin{enumerate}
\item There exists $C(n)>0$ such that the cardinality  of $\cC^k$ satisfies
\begin{align}
\big|\#\cC^k\big|\geq C^{-k}\big|\#\cC\big| = C^{-k} N\, .
\end{align}
\item For every $x^k_j\in \cC^k$ we have 
\begin{align}
\sum_{0\leq \alpha\leq \alpha_k} T^\delta_\alpha (x^k_j)\geq k\, .
\end{align}
\item If $\big|\#\cC^k\big|>1$ and $s_k> \bar r^{-1}r$ then  $\cC^{k+1}\neq \emptyset$.\\
\end{enumerate}

Before constructing the sequence of sets, let us see that once the construction is complete,  we will have 
proved our desired estimate on $N$.  Indeed, let $k$ be the largest index such that $\cC^{k}\neq \emptyset$. 
 By the third property we must have either $|\#\cC^k|=1$ or $s_k\leq \bar r^{-1}r$, at which point we get by a 
covering argument that $|\# \cC^k|< C(n)$.  By Lemma \ref{l:bad_scale_estimate} and the second property 
we have that $k\leq k(n,\rv,\delta)=k(n,\rv)$, and thus by the first property we have 
\begin{align}
N\leq C(n)^{k(n,\rv)}\cdot |\# \cC^k|\leq C(n,\rv)\, ,
\end{align}
which proves the result.\\

Now let $\cC^0\equiv \cC$ with $s_0=1$. Clearly, the inductive properties hold for $\cC^0$.  
Assume we have built $\cC^k\subseteq \cC$ with $s_k>0$ satisfying the inductive properties, 
and let us build $\cC^{k+1}$.  First note that if $|\#\cC^k|=1$ or $s_k\leq \bar r^{-1}r$, then we let 
$\cC^{k+1}=\emptyset$.  Our construction will otherwise give us a nonempty $\cC^{k+1}$, so that 
the third inductive property will automatically be satisfied.  So let us denote $s'_k = \diam(\cC^k)\cdot 2^{-10}$.
Choose an efficient covering  $\{B_{s'_k}(x^k_j)\}$, where $x^k_j\in \cC^k$, so that 
the balls in $\{B_{s'_k/4}(x^k_j)\}$
 are disjoint.  Note that because $\diam(\cC^k)<4s_k$, the usual doubling estimates imply that there are 
at most $C(n)$ balls in this covering.  We choose the ball $B_{s'_k}(y)$ such that $\cC^k\cap B_{s'_k}(y)$ 
has the largest cardinality of any ball from the covering.  Then we define $\cC^{k+1}= \cC^k\cap B_{s'_k}(y)$.\\

By that by our choice of ball, $B_{s'_k}(y)$, we have 
\begin{align}
\big|\#\cC^{k+1}\big| &= \big|\#\cC^k\cap B_{s'_k}(y)\big|\geq C(n)^{-1} \big|\# \cC^k\big|\geq C^{-(k+1)} N\, ,
\end{align}
so that $\cC^{k+1}$ satisfies the first inductive property.  To find $s_{k+1}$ and prove the second inductive property,
 let us define the following.  For each $x^{k+1}_j\in \cC^{k+1}$ if 
\begin{align}
T^\delta_{\alpha_k+7}(x^{k+1}_j)=1\, ,
\end{align}
then let us set $\beta_j=\alpha_k+7$, and otherwise let $\beta_j\geq \alpha_k+8$ be the largest integer such that
 $T^\delta_{\beta_j-1}(x^{k+1}_j)=0$ but $T^\delta_{\beta_j}(x^{k+1}_j)=1$.  Note that 
$B_{s'_k}(y)\subseteq B_{r_{\alpha_k+7}}(x^k_j)$.  Let $\alpha_{k+1}\equiv \max\{\beta_j,\lceil-\log_2\big(\bar r r^{-1}\big)\rceil\}$
 with  $x^{k+1}\in \cC^{k+1}$ the associated element which attains the maximum, and note by \eqref{e:main_four_regscale:1} that
\begin{align}
\cC^{k+1} =\cC^k\cap B_{s'_k}(y) = \cC^k\cap B_{2^{-\alpha_{k+1}+1}}(x^{k+1})\, .
\end{align}

In particular, with $s_{k+1}=r_{\alpha_{k+1}}$ then $\diam(\cC^{k+1})<4s_{k+1}$, and the second inductive property holds, 
which completes the induction step of the construction, and hence, the proof.

\end{proof}

\subsection{Finite Diffeomorphism Type}\label{ss:finite_diffeo_dimfour}

In this subsection we will prove Theorem \ref{t:main_finite_diffeo} and give some refinements 
which will be useful for the $L^2$-curvature estimate of Theorem \ref{t:main_dimfour}.  \\

We begin by associating to a good scale the subgroup of ${\rm O}(4)$ occuring in Theorem
\ref{t:annulus}.  

\begin{definition}
Let $\epsilon,\delta>0$ be such that Theorem \ref{t:annulus} holds.  For $x\in B_1(p)$ and $\alpha\in\dN$ such 
that $T^\delta_\alpha(x)=0$, we denote by $[\Gamma_\alpha(x)]\subseteq O(4)$, the conjugacy
 class of the
 discrete 
subgroups arising from Theorem \ref{t:annulus}.
\end{definition}
In the sequel, $\Gamma_\alpha$ will denote some arbitrary
element of $[\Gamma_\alpha]$; only the isometry class of
of $S^3/\Gamma_\alpha$, which is independent of the particular
choice, is significant.

The following is the key {\it Neck} lemma for our finite diffeomorphism of Theorem \ref{t:main_finite_diffeo}. 
 In essence, the proof of Theorem \ref{t:main_finite_diffeo} will come from decomposing $M$ into a finite number of distinct pieces. 
 What we are refering to informally as the {\it neck} regions will be diffeomorphic to 
cylinders $\dR\times S^3/\Gamma$.
They will connect the  pieces which will be referred to as {\it body} regions.

\begin{lemma}\label{l:neck}
For every $0<\epsilon\leq \epsilon(\rv)$, there exists $\delta=\delta(\rv,\epsilon)$ 
with the following properties. Let $M^4$ satisfy $|\Ric_{M^4}|\leq 3\delta$ and $\Vol(B_1(p))>\rv>0$.
Let  $x\in B_1(p)$ and assume $\alpha_1\in \dN$ satisfies 
$T^\delta_{\alpha_1}(x)=0$ with $\Gamma_{\alpha_1}$ the corresponding group.  Then if $\alpha_2\in \dN$ is such that 
$\cV^\delta_{r_{\alpha_2}/4}(x)\geq \ln\big|\Gamma_{\alpha_1}\big|-\delta$,  there exists a subset 
$A_{r_{\alpha_2}/2,2r_{\alpha_1}}(x)\subseteq U\subseteq A_{(1-\epsilon)r_{\alpha_2}/2,2(1+\epsilon)r_{\alpha_1}}(x)$ 
and a diffeomorphism $\Phi:A_{r_{\alpha_2}/2,2r_{\alpha_1}}(0)\to U$, where $0\in \dR^4/\Gamma_{\alpha_1}$, such that
 if $g_{ij}=\Phi^*g$ is the pullback metric, we have
\begin{align}
||g_{ij}-\delta_{ij}||_{C^0(A_{r_{\alpha}/2},r_{2\alpha})}+r_{\alpha}||\partial_k g_{ij}||_{C^0(A_{r_{\alpha}/2},2r_{\alpha})}<\epsilon
\end{align}
\end{lemma}
\begin{proof}
We will fix $\epsilon(\rv)>0$ later.  For the moment let any $\epsilon_1>0$ be arbitrary with $\delta_1(\rv,\epsilon_1)>0$ the 
corresponding number from Theorem \ref{t:annulus}.  If $T^{\delta_1}_{\alpha_1}(x)=0$ then there exists a diffeomorphism 
\begin{align}
\Phi_{\alpha_1}:A_{r_{\alpha_1}/2,2r_{\alpha_1}}(0)\to U_{\alpha_1}\, ,
\end{align}
where $0\in \dR^4/\Gamma_{\alpha_1}$ and $A_{r_{\alpha_1}/2,2r_{\alpha_1}}(x)\subseteq U_\alpha
\subseteq A_{(1-\epsilon)r_{\alpha_1}/2,(1+\epsilon)r_{\alpha_1}}(x)$, such that 
\begin{align}\label{e:fd:annulus:1}
||\Phi_{\alpha_1}^*g_{ij}-\delta_{ij}||_{C^0(A_{r_{\alpha_1}/2},2r_{\alpha_1})}
+r_{\alpha_1}||\partial_k \Phi_{\alpha_1}^*g_{ij}||_{C^0(A_{r_{\alpha_1}/2},2r_{\alpha_1})}<\epsilon_1\, .
\end{align}
In particular, if $\epsilon>0$ is fixed and $2\delta(n,\epsilon)$ is the corresponding number from 
Theorem \ref{t:annulus}, then we can choose $\epsilon_1=\epsilon_1(\epsilon,\rv)$ sufficiently small so that
\begin{align}
\cV^\delta_{r_{\alpha_1}}(x)<\ln|\Gamma_{\alpha_1}|+\delta\, .
\end{align}
Thus, if $\alpha_2$ is such that
\begin{align}
\cV^\delta_{r_{\alpha_2}/2}(x)\geq \ln|\Gamma_{\alpha_1}|-\delta\, ,
\end{align}
then for all $\alpha_1\leq\alpha\leq\alpha_2$ we have  $T^{2\delta}_{\alpha}(x)=0$.  \\

By Theorem \ref{t:annulus}, there exists for each $\alpha_1\leq \alpha\leq \alpha_2$,
 a diffeomorphism 
\begin{align}
\Phi_\alpha:A_{r_{\alpha}/2,2r_{\alpha}}(0)\to U_\alpha\, ,
\end{align}
where $0\in \dR^4/\Gamma_\alpha$ and $A_{r_{\alpha}/2,2r_{\alpha}}(x)
\subseteq U_\alpha\subseteq A_{(1-\epsilon)r_{\alpha}/2,2(1+\epsilon)r_{\alpha}}(x)$, such that 
\begin{align}\label{e:fd:annulus:11}
||\Phi_\alpha^*g_{ij}-\delta_{ij}||_{C^0(A_{r_{\alpha}/2},2r_{\alpha})}
+r_{\alpha}||\partial_k \Phi_\alpha^*g_{ij}||_{C^0(A_{r_{\alpha}/2},2r_{\alpha})}<\epsilon\, .
\end{align}
In particular this implies that $\Gamma_\alpha=\Gamma$
can be chosen independent of 
$\alpha$. 

 Next we focus on the inverse maps
\begin{align}
\Phi_\alpha^{-1}:U_\alpha \to A_{r_{\alpha}/2,2r_{\alpha}}(0)\, .
\end{align}
Observe that by \eqref{e:fd:annulus:11}, after possibly composing $\Phi_\alpha$ with a rotation of $\dR^4/\Gamma$ 
we can assume for $x\in U_\alpha\cap U_\beta$ that
\begin{align}\label{e:fd:annulus:2}
|\Phi_\alpha^{-1}(x)-\Phi_\beta^{-1}(x)|<\epsilon r_\alpha\, .
\end{align}
Now let $\epsilon<\epsilon(\rv)$ be sufficiently small, so that if $x\in \dR^4/\Gamma$, then 
$B_{\epsilon |x|}(x)\subseteq \dR^4/\Gamma$ is isometric to the standard Euclidean ball $B_{\epsilon|x|}(0^4)\subseteq \dR^4$.  
Note in particular that if $\{x_i\}\in B_{\epsilon|x|}(x)$ is a collection of points, then any convex combination is well defined.

For each $\alpha$ let $\varphi'_\alpha:U_\alpha\to \dR$ be a smooth cutoff function such that
\begin{align}
\varphi'_\alpha(x) =
\begin{cases}
&1 \text{ if }x\in A_{3r_\alpha/8,15r_\alpha/8}(x)\, ,\notag\\
&0 \text{ if }x\not\in A_{r_\alpha/2,2r_\alpha}(x)\, ,
\end{cases}
\end{align}
and such that $|\nabla\varphi'_\alpha|\leq 10r_\alpha^{-1}$.  If we set
$\varphi'(x)=\sum_\alpha \varphi'_\alpha(x)$ then $1\leq \varphi'(x)\leq 4$.  In particular, 
\begin{align}
\varphi_\alpha = \frac{\varphi'_\alpha(x)}{\varphi'(x)}:U_\alpha\to \dR\, ,
\end{align}
sarisfies $\sum\varphi_\alpha(x)=1$, and so, is a partition of unity, with $|\nabla \varphi_\alpha|\leq 40r_\alpha^{-1}$.  

Define the map
\begin{align}
\Phi^{-1}: U=\bigcup_\alpha U_\alpha\to A_{r_{\alpha_2}/2,2r_{\alpha_1}}(0)\, ,
\end{align}
given by
\begin{align}
\Phi^{-1}(x) = \sum_\alpha \varphi_\alpha(x)\Phi^{-1}_\alpha(x)\, .
\end{align}
 (As previously noted, the convex combination is well defined since the $\Phi^{-1}_\alpha(x)$ all live in a
 ball which is isometric to a Euclidean ball.)  On each domain, $U_\alpha$, we have by \eqref{e:fd:annulus:1} and 
\eqref{e:fd:annulus:2} that $\Phi^{-1}$ and $\Phi^{-1}_\alpha$ are $C^1$-close. Hence, $\Phi^{-1}$ is a diffeomorphism,
 and a quick computation using \eqref{e:fd:annulus:1} and \eqref{e:fd:annulus:2} verifies the desired estimates:
\begin{align}\label{e:fd:annulus:3}
||\Phi^*g_{ij}-\delta_{ij}||_{C^0(A_{r_{\alpha}/2},2r_{\alpha})}
+r_{\alpha}||\partial_k \Phi_\alpha^*g_{ij}||_{C^0(A_{r_{\alpha}/2},2r_{\alpha})}<C\epsilon\, .
\end{align}
By choosing $\epsilon$ appropriately small,  we complete the proof.
\end{proof}

The following lemma could be termed a ``gap lemma''. It will be used to tell us that if we consider two distinct neck regions,
 then the complexity of the smaller neck region must be strictly less than that of the larger neck region. 

\begin{lemma}\label{l:order_drop_diffeo}
For each $\delta<\delta(\rv)$, there exists
 $\bar \alpha(\delta, \rv)$ with the following property. If $|\Ric_{M^4}|\leq 3\delta$,
$\Vol(B_1(p))>\rv>0$, 
 and $\cV^\delta_{r_{1}}(x)<\ln N-\delta$ for some $N\in \dN$ and $x\in B_1(p)$, we have 
$$\cV^\delta_{r_{\bar\alpha}}(x)<\ln\Big(N-1\Big)+\delta\, .$$
\end{lemma}
\begin{proof}
First note by Theorem \ref{t:annulus} that if $\delta$ is fixed, then there exists $\delta'(\rv,\delta)$ such that if 
$|\Ric|\leq 3\delta'$ and if $T^{\delta'}_0=0$, then 
\begin{align}
\big|\cV^{\delta'}_{1}(x)-\ln|\Gamma_0|\big|<\delta\, .
\end{align}

By rescaling this inequality, we see that in the context of this lemma, the following holds. If
 $x\in B_1(p)$, $\alpha>\bar\alpha(v,\delta)$ and 
\begin{align}
\big|\cV^\delta_{4r_\alpha}(x)-\cV^\delta_{r_\alpha/4}(x)\big|<\delta'\, ,
\end{align}
then we have 
\begin{align}
\big|\cV^{\delta}_{r_\alpha}(x)-\ln|\Gamma_\alpha|\big|<\delta\, .
\end{align}
In particular, for $x\in B_1(p)$, we can apply Lemma \ref{l:bad_scale_estimate} to see that there exists a scale 
$\alpha\leq \bar\alpha(v,\delta)$ such that
\begin{align}
\big|V^\delta_{r_{\alpha}}(x)-\cV^\delta_{r_{\alpha+1}}(x)\big|<\delta'\, ,
\end{align}
and hence
\begin{align}
\big|\cV^{\delta}_{r_\alpha}(x)-\ln|\Gamma_\alpha|\big|<\delta\, .
\end{align}
However,  if 
\begin{align}
\cV^{\delta}_{r_\alpha}(x)<\ln N-\delta\, ,
\end{align}
 this implies $|\Gamma_\alpha|<N$, which completes the proof.   
\end{proof}

In Lemma \ref{l:neck} we have built the required structure for constructing the {\it neck} regions of our decomposition.  What is left is to build the {\it body} regions of the decomposition.  The following lemma will be applied in the proof of Theorem \ref{t:main_finite_diffeo} in order to construct the various body regions.

\begin{lemma}\label{l:body_diffeo}
For every $\delta>0$, there exists $r_0(\rv,\delta),N(\rv,\delta)>0$ with 
the following properties.  Let $M^4 $ satisfy  $|\Ric_{M^4_j}|\leq 3\delta$, $\Vol(B_1(p))>\rv>0$.
 Then there exists points $\{x_j\}_1^N$ with $N\leq N(\rv,\delta)$,
and scales $\alpha_j\in \dN$ with $r_j\equiv r_{\alpha_j}>r_0$, such that 
\begin{enumerate}
\item $T^\delta_{\alpha_j}(x_j)=0$,
\vskip1mm

\item If $x\in B_1(p)\setminus \bigcup_j B_{r_{j}}(x_j)$ then $r_h(x)>r_0$, 
\vskip1mm

\item If $\beta_j\in \dN$ denotes the largest integer such that $\cV^\delta_{r_{\beta_j}/4}(x_j)\geq \ln|\Gamma_j|-\delta$,
then for every $x\in B_{2r_{\beta_j}}(x_j)$ we have 
\begin{align}
\cV^\delta_{r_{\beta_j}/8}(x)< \ln|\Gamma_j|-\delta\, .
\end{align}
\end{enumerate}
\end{lemma}
\begin{proof}
Let $\delta>0$ be chosen with $\delta'(\rv,\delta)$ to be chosen later.  Note that by Lemma \ref{l:bad_scale_estimate},
 for each $x\in B_1(p)$  there exists $\alpha_x\leq \bar\alpha(v,\delta')$ such that $T^{\delta'}_{\alpha_x}(x)=0$. 
 Consider the covering $\{B_{r_{\alpha_x}}(x)\}$ of $B_1(p)$, and choose an efficient 
subcovering $\{B_{r_j}(x'_j)\}_1^N$,
where $r_j=r_{\alpha_{x'_j}}$ and  the balls in $\{B_{r_j/4}(x'_j)\}$ are disjoint.  The usual doubling arguments imply
 that $N\leq N(v,\delta')$.  

By Theorem \ref{t:annulus}, if we are given $\epsilon>0$, then we can choose $\delta'(\rv,\epsilon,\delta)$ such that 
for each $x\in B_{\epsilon r_j}(x'_j)$ we have $T^\delta_{\alpha_j}(x)=0$, while for each $x\in A_{\epsilon r_j,2r_j}(x_j)$ 
we have $r_h(x)>\bar r(v,\epsilon) r_j\geq r_0(v,\epsilon,\delta')$. 
Let $\Gamma_j$ be the group associated to $B_{r_j}(x'_j)$, and for each $x\in B_{\epsilon r_j}(x'_j)$ let $\beta_j(x)$ be 
the largest integer such that $V^\delta_{r_{\beta_j}/4}(x)\geq \ln|\Gamma_j|-\delta$.  Let $\beta_j=\max \beta_j(x)$ with $x_j$ 
the corresponding point.  Note that for $\epsilon(v,\delta)$ sufficiently small, we have 
$B_{2r_{\beta_j}}(x_j)\subseteq B_{\epsilon r_j}(x'_j)$, and in particular,  for every $x\in B_{2r_{\beta_j}}(x_j)$ 
\begin{align}
V^\delta_{r_{\beta_j}/8}(x)< \ln|\Gamma_j|-\delta\, .
\end{align}

Consider the collection of balls $\{B_{r_j}(x_j)\}$.  Clearly, by construction, conditions (1) and (3) are satisfied. 
If $x\in B_1(p)\setminus \{B_{r_j}(x_j)\}$ then since $\{B_{2r_j}(x_j)\}$ cover $B_1(p)$ we have that for some $x_j$ that 
$x\in A_{r_j,2r_j}(x_j)$, which implies $r_h(x)\geq r_0(v,\delta)$, as claimed.
\end{proof}

By the previous lemma, the regions between necks, namely $B_1(p)\setminus \bigcup_j B_{r_{\alpha_j}}(x_j)$, 
can be written as the union of a definite number of balls of definite size, on  which there is definite geometric control. \\

We are nearly in a position to prove Theorem \ref{t:main_finite_diffeo}.  To do so we
 will in fact prove the following stronger result, 
which is the {\it bubble tree} decomposition of $M^4$.

\begin{theorem}\label{t:bubble_tree}
Let $M^4$ satisfy $|\Ric_{M^4}|\leq 3$, 
$\Vol(M)\geq \rv>0$ and $\diam(M)\leq D$. Then there exists a decomposition of $M^4$
\begin{align}
M^4 = \cB^1\cup\bigcup_{j_2=1}^{N_2}
\cN^2_{j_2}\cup\bigcup_{j_2=1}^{N_2}\cB^2_{j_2}\cup\cdots\cup
\bigcup_{j_k=1}^{N_k} \cN^k_{j_k}\cup\bigcup_{j_k=1}^{N_k} \cB^k_{j_k}\, ,
\end{align}
into open sets which satisfy the following:
\begin{enumerate}
\item If $x\in \cB^\ell_j$ then $r_h(x)>r_0(n,\rv,D)\cdot \diam(\cB^\ell_j)$.
\vskip1mm

\item Each neck $\cN^\ell_j$ is diffeomorphic 
to $\dR\times S^3/\Gamma^\ell_j$ for some $\Gamma^\ell_j\subset O(4)$.
\vskip1mm

\item $\cN^\ell_j\cap \cB^\ell_j$ is diffeomorphic to $\dR\times S^3/\Gamma^\ell_j$.

\item $\cB^{\ell-1}_{j'}\cap \cN^\ell_j$ are either empty or 
diffeomorphic to $\dR\times S^3/\Gamma^\ell_j$.
\vskip1mm

\item $N_\ell\leq N(\rv,D)$ and $k\leq k(\rv,D)$.
\end{enumerate}
\end{theorem}

\begin{proof}
Let us remark first, that if $p\in M^n$, then by volume ratio monotonicity, we have for every $r\leq 1$ that
\begin{align}
\Vol(B_r(p))\geq \frac{\Vol_{-1}(B_r)}{\Vol_{-1}(B_D)}\Vol(B_D(p)) \geq C(n,D)^{-1}\Vol(M^4)r^n
\geq C^{-1}\rv r^n=: \rv' r^n\, .
\end{align}

Let $\epsilon<\epsilon(\rv')$ from Lemma \ref{l:neck} with $\delta(\rv,D,\epsilon)$ sufficiently small to satisfy 
Theorem \ref{t:annulus} and  Lemmas \ref{l:neck}, \ref{l:order_drop_diffeo}, \ref{l:body_diffeo}.  
After rescaling, it is sufficient to consider a Riemannian manifold $(M^4,g)$ with $|\Ric_{M^4}|\leq 3\delta$, 
$\diam(M) \leq D\delta^{-2}=D'$ and $\Vol(B_1(p))> \rv'>0$ for every $p\in M$.\\

Let us begin by efficiently covering $M^4$ by balls $\{B_1(x^0_j)\}$ such that
the balls in $\{B_{1/4}(x^0_j)\}$ are disjoint.  
By the usual doubling argument,  there are at most $N(n,D,\rv)$ such balls.  For each such ball, we 
apply Lemma \ref{l:body_diffeo} in order to produce a collection of balls $\{B_{r^1_j}(x^1_j)\}_1^{N_1}$ 
such that $r^1_j=r_{\alpha^1_j}>\bar r(\rv,D)$, $N_1\leq N(\rv',D')$, $T^\delta_{\alpha^1_j}(x^1_j)=0$, and 
such that if $x\in M^4\setminus \bigcup_j B_{r^1_j}(x^1_j)$ then $r_h(x)>\bar r$.  Furthermore, if we denote
by $\Gamma^2_j$,
the group associated to $B_{r^1_j}(x^1_j)$, then if $\beta^1_j$ is the largest integer such that 
$V^\delta_{r_{\beta^1_j}/2}(x^1_j)\geq \ln |\Gamma_j^2|-\delta$, then for all 
$x\in B_{2r_{\beta^1_j}}(x^1_j)$ we have 
\begin{align}
\cV^\delta_{r_{\beta^1_j}/4}(x^1_j)< \ln |\Gamma_j^2|-\delta\, .
\end{align}

Define
\begin{align}
\cB^1=: M^4\setminus \bigcup B_{r_j}(x_j)\, ,
\end{align}
as the first body region. Then we can write
\begin{align}\label{e:tree_decomp:1}
M^4=\cB^1 \bigcup B_{2r^1_j}(x^1_j)\, ,
\end{align}
where by using Theorem \ref{t:annulus}, we have that $B_{2r^1_j}(x^1_j)\cap \cB^1$ is 
diffeomorphic to $\dR\times S^3/\Gamma^1_j$.\\

Now to prove the theorem, let us inductively build a decomposition of $M^4$ 
\begin{align}
M^4 = \cB^1\cup\bigcup_{j_2=1}^{N_2}\cN^2_{j_2}\bigcup_{j_2=1}^{N_2}\cB^2_{j_2}\cup\cdots
\cup\bigcup_{j_k=1}^{N_k} \cN^k_{j_k}\bigcup_{j_k=1}^{N_k} \cB^k_{j_k}
\cup\bigcup_{a=1}^{N_{k+1}} B_{2r^k_{a}}(x_a)\, ,
\end{align}
with the following properties:
\begin{enumerate}
\item If $x\in \cB^\ell_j$ then $r_h(x)>r_0(n,\rv,D)\cdot \diam(\cB^\ell_j)$.
\vskip1mm

\item Each neck $\cN^\ell_j$ is diffeomorphic to 
$\dR\times S^3/\Gamma^\ell_j$ for some $\Gamma^\ell_j\subset O(4)$.
\vskip1mm

\item Each $\cN^\ell_j\cap \cB^\ell_j$ is diffeomorphic to $\dR\times S^3/\Gamma^\ell_j$.  Each $\cN^\ell_j\cap \cB^{\ell-1}_{j'}$ are either empty or diffeomorphic to $\dR\times S^3/\Gamma^\ell_j$. 
\vskip1mm

\item $N_\ell\leq N(n,\rv,D)$.
\vskip1mm

\item If $\cN^{\ell+1}_a\cap \cB^\ell_j\neq \emptyset$, then $|\Gamma^\ell_a|\leq |\Gamma^\ell_j|-1$.
\vskip1mm

\item We have $r^k_a=r_{\alpha^k_a}$ with $T^\delta_{\alpha^k_a}=0$, 
and $\cB^k_j\cap B_{r^k_a}(x_a)\subseteq A_{r^k_a/2,r^k_a}(x_a)$.
\vskip1mm

\item If $\beta^k_a$ is the largest integer such that $\cV^\delta_{r_{\beta^k_a}/4}(x_a)\geq \ln |\Gamma^k_a|-\delta$, 
then for every $x\in B_{2r_{\beta^k_a}}(x_a)$ we have  $\cV^\delta_{r_{\beta^k_a}/8}(x)<\ln |\Gamma^k_a|-\delta$.
\end{enumerate}

Before building the inductive decomposition, let us note that once we have it,
we will have finished the proof.
 In fact, all we really need to see is that for some $k\leq k(n,\rv,D)$, there are no balls
 $\{B_{r^k_{a}}(x_a)\}$ in the decomposition.  
To see this, observe that by the lower volume bound we have the upper order bound 
$|\Gamma^2_j|\leq C(\rv,D)$.  
By condition (5) above we have by iteration that for each $j$ that there is some $j_2$ 
such that
\begin{align}
0\leq |\Gamma^k_j|\leq |\Gamma^2_{j_2}|-(k-2)\leq C(\rv,D)-(k-2)\, ,
\end{align}
and in particular this immediately implies the upper bound
\begin{align}
k\leq k(\rv,D)\, .
\end{align}

To prove the inductive decomposition, we begin by noting 
that \eqref{e:tree_decomp:1} provides the basic
 case.  So let us assume that the 
decomposition has been constructed for some $k$, 
and let us build the decomposition for $k+1$.  

First, we use condition (7) and Lemma \ref{l:neck} to see that there 
exists an open set
\begin{align}
A_{r_{\beta^k_a}/2,2r_{\alpha^k_a}}(x_a)\subseteq \cN^{k+1}_a
\subseteq A_{(1-\epsilon)r_{\beta^k_a}/2,2(1+\epsilon)r_{\alpha^k_a}}(x_a)\, ,
\end{align}
and a diffeomorphism $\Phi^{k+1}_a:\cN^{k+1}_a
\to A_{r_{\beta^k_a}/2,2r_{\alpha^k_a}}(0)$ with $0\in \dR^4/\Gamma^k_a$.  
By Lemma \ref{l:order_drop_diffeo},  there exists a radius 
$r_{a}=\bar r(\rv,\delta)r_{\beta^k_a}$ such that
\begin{align}
\cV^\delta_{r_{a}}(x)<\ln\big(|\Gamma^{k+1}_a|-1\big)+\delta\, ,
\end{align}
for every $x\in B_{2r_{\beta^k_a}}(x_{a})$.  

Pick some 
efficient covering 
$\{B_{r_a}(x_{aj}\}$ of $B_{2r_{\beta^k_a}}(x_a)$ such that the balls in
 $\{B_{r_a/4}(x_{aj}\}$ disjoint.  Now apply Lemma \ref{l:body_diffeo} to each 
ball $\{B_{r_a}(x_{aj}\}$ in order to construct 
a collection of balls $\{B_{r^{k+1}_{ab}}(x^{k+1}_{ab})\}$ with $r^{k+1}_{ab}
=r_{\alpha^{k+1}_{ab}}>\bar r(\rv,\delta) r_{\beta^k_a}$.  
Observe that since there are at most $N(\rv,D)$ balls in the collection $\{B_{r_a/4}(x_{aj}\}$,
 and the application of Lemma \ref{l:body_diffeo} 
produces at most $N(\rv,D)$ balls for each of these, we have at most $N(\rv,D)$ such balls in total.  

If we put
\begin{align}
\cB^{k+1}_a\equiv \ B_{2r_{\beta^k_a}}(x_a)\setminus \cup B_{r_{\alpha^{k+1}_{ab}}}(x_{ab})\, ,
\end{align}
 we see that $\cB^{k+1}_a$ and the collection $\{B_{r^{k+1}_{ab}}(x^{k+1}_{ab})\}$ satisfy
 the inductive conditions.  Specifically, 
what is left to check is condition (5).  However, by construction, we have 
\begin{align}
\ln(|\Gamma^k_a|-1)+\delta > \cV^\delta_{r^{k+1}_{ab}}(x_{ab})\geq \ln|\Gamma^{k+1}_{ab}|-\delta\, ,
\end{align}
which for $\delta(\rv)$ sufficiently small implies $|\Gamma^{k+1}_{aj}|<|\Gamma^k_a|$.  In particular, 
the decomposition
\begin{align}
M^n \equiv \cB^1\cup\bigcup_{j_2=1}^{N_2}\cN^2_{j_2}\bigcup_{j_2=1}^{N_2}\cB^2_{j_2}\cup\cdots\cup
\bigcup_{j_k=1}^{N_k} \cN^k_{j_k}\bigcup_{j_k=1}^{N_k} \cB^k_{j_k}\cup\bigcup_{a=1}^{N_{k+1}} \cN^{k+1}_a
 \bigcup \cB^{k+1}_a \bigcup B_{2r_{\alpha^{k+1}_{ab}}}(x_{ab})\, ,
\end{align}
satisfies the inductive hypothesis as well, which completes the proof.

\end{proof}

Now that we have constructed the bubble tree in Theorem \ref{t:bubble_tree} 
let us finish the proof of Theorem \ref{t:main_finite_diffeo}:

\begin{proof}[Proof of Theorem \ref{t:main_finite_diffeo}]

Let $M^4$ satisfy $|\Ric_{M^4}|\leq 3$, $\Vol(M)>\rv>0$ and $\diam(M^4)\leq D$.  
Then using Theorem \ref{t:bubble_tree}, we can write
\begin{align}
M^4 \equiv \cB^1\cup\bigcup_{j_2=1}^{N_2}\cN^2_{j_2}\cup\bigcup_{j_2=1}^{N_2}\cB^2_{j_2}\cup\cdots
\cup\bigcup_{j_k=1}^{N_k} \cN^k_{j_k}\cup\bigcup_{j_k=1}^{N_k} \cB^k_{j_k}\, .
\end{align}

First we will analyze each body region $\cB^k_j$.  Indeed, by (1) and Theorem \ref{t:harmonic_rad_finite_diffeo}, 
it follows that there are at most $C(\rv,D)$-diffeomorphism types for each $\cB^k_j$.  
By (4) there 
are at most $C(\rv,D)$ such body regions, each of
which has at most $C(\rv,D)$ boundary components.
By (2), (3) and Lemma \ref{l:neck}, we can suppose that $\epsilon(\rv)$ is
so small that for each neck, the induced
attaching map between boundary components of the corresponding
pair of bodies is sufficiently close to being an isometry of
$S^3/\Gamma_\alpha$, that it is isotopic to such an isometry.
Since the group of isometries of a compact manifold 
has finitely many components,
it follows that for each neck, 
there are only finitely many possible isotopy classes
of such attaching maps. As a consequence, there are
at most $C(\rv,D)$ diffeomorphism 
types that can arise by attaching together the body regions
by the various necks.  This proves the theorem.
\end{proof}

\subsection{$L^2$ Curvature Estimates}\label{ss:L2_curvature_dimfour}

We begin with the following, whose proof is essentially the same as that of
Theorem \ref{t:main_finite_diffeo} of the previous subsection:

\begin{theorem}\label{t:local_finite_diffeo}
There exists $\delta(\rv)>0$ such that if $M^4$ satisfies $|\Ric_{M^4}|<2\delta$, $\Vol(B_1(p))>\rv>0$, 
 and 
$T^\delta_0(p)=0$, then there exists $B_{1}(p)\subseteq U\subseteq B_{2}(p)$ such that $U$ has at most $C(\rv)$ 
diffeomorphism types.  Further, $U$ can be chosen so that it's boundary $\partial U$ is diffeomorphic to $S^3/\Gamma$ 
and satisfies the second fundamental form estimate $|A|\leq C(v)$.
\end{theorem}
\begin{proof}
The proof is the same as that of Theorem \ref{t:main_finite_diffeo}, except for the 
second fundamental form estimate on the boundary.  
To see this estimate, we use $T^\delta_0(p)=0$ and Theorem \ref{t:annulus} to find a 
diffeomorphism $\Phi:A_{1/2,2}(0)\to B_1(p)$ 
onto its image, such that if $g_{ij}=\Phi^*g$ is the pullback metric then
\begin{align}
||g_{ij}-\delta_{ij}||_{C^{0}} + ||\partial_k g_{ij}||_{C^{0}}<\epsilon\, .
\end{align}
In particular, we can choose $U$ so that its boundary is $\partial U=\partial B_{3/2}(0)$ in these coordinates.  
The $C^1$ estimates on $g$ give rise to the appropriate second fundamental form estimates on $\partial U$.
\end{proof}

With this in hand we are in a position to finish the proof of Theorem \ref{t:main_dimfour}:

\begin{proof}[Proof of Theorem \ref{t:main_dimfour}]
Let $M^4$ satisfy $|\Ric_{M^4}|\leq 3$ and $\Vol(B_1(p))>\rv>0$.  Using volume monotonicity,
 we have for every $x\in B_1(p)$ and $r\leq 1$, 
\begin{align}
\Vol(B_r(x))\geq \frac{\Vol_{-1}(B_r)}{\Vol_{-1}(B_2)}\Vol(B_2(x))\geq c(n)\Vol(B_1(p))\,r^4\geq c\rv\,r^4\, . 
\end{align}
Let $\delta(\rv)$ be as in Theorem \ref{t:local_finite_diffeo}. By Lemma \ref{l:bad_scale_estimate}, we have that
 for each $x\in B_1(p)$, there exists a radius, $r_{\alpha_x}=2^{-\alpha_x}\in [C(\rv)\delta^3, \delta^2]$, such that 
$T^\delta_{\alpha_x}(x)=0$.  Let $\{B_{r_{i}}(x_i)\}$ be a subcovering such that 
the balls in $\{B_{r_{i}/4}(x_i)\}$ are disjoint, where
 $r_i=r_{\alpha_{x_i}}$.  Since $r_i>\bar r(\rv)$, 
we have by the usual doubling estimates that there are at most $C(\rv)$ balls in this covering.\\

Note that, for each ball $B_{r_i}(x_i)$, we can apply Theorem \ref{t:local_finite_diffeo} in order to get a subset 
$U_i\supseteq B_{r_i}(x_i)$ with bounded diffeomorphism type and uniform boundary control.  Now recall in dmiension
$4$, the 
Chern-Guass-Bonnet formula can be written as
\begin{align}
\chi(U_i) = \frac{1}{32\pi^2}\int_{U_i}|\Rm|^2-4|\Ric|^2+R^2 + \int_{\partial U_i} \Psi\, ,
\end{align}
where $\Psi=\Psi(A)$ is a function of the second fundamental form.  By reorganizing, we obtain the bound
\begin{align}
\int_{U_i}|\Rm|^2 &\leq  32\pi^2|\chi(U_i)|+4\int_{U_i}|\Ric|^2+C\int_{U_i}|\Psi|\, ,\notag\\
&\leq C(\rv)\, ,
\end{align}
where we have used the bound on the diffeomorphism type, the Ricci bound, and the second fundamental form bound 
from 
Theorem \ref{t:local_finite_diffeo}.  By summing over $i$, we get
\begin{align}
\fint_{B_1(p)}|\Rm|^2\leq C(\rv)\sum \int_{U_i}|\Rm|^2 \leq C(\rv)\, ,
\end{align}
as claimed.
\end{proof}

In view of  \cite{A89}, \cite{BKN89},  \cite{T90}, \cite{Anderson_Einstein}, the
$L^2$ curvature bound in Theorem \ref{t:local_finite_diffeo} has the following consequence.

\begin{corollary}
Let $(M^4_j,d_j,p_j)\stackrel{d_{GH}}{\longrightarrow} (X,d,p)$ be a Gromov-Hausdorff limit of manifolds 
with $|\Ric_{M^4_j}|\leq n-1$ and $\Vol(B_1(p_j))>\rv>0$.  Then $X$ is a Riemannian orbifold with at most
$c(v)$ singular points.
\end{corollary}

Similary, we get
\begin{corollary}
Let $M^4$ be a complete noncompact Ricci flat manifold with Euclidean volume growth. Then $M^4$ is an ALE space.
\end{corollary}

\section{Open Questions}\label{s:conjectures}
In this section, we briefly remark on some possible extensions of the results of this paper. 
 To begin with, we recall that one of the main applications of this paper 
was to combine the codimension $4$ estimates of Theorem \ref{t:main_codim4} 
with the ideas of quantitative stratification in order to show for all $q<2$ that 
$\fint_{B_1(p)}|\Rm|^q$ is uniformly bounded when $M^n$ is a noncollapsed manifold 
with bounded Ricci curvature.  Furthermore, in dimension $4$
 we were able to improve
 this to show a bound on $\fint_{B_1(p)}|\Rm|^2$.  We conjecture that this  holds in any dimension.

\begin{conjecture}\label{conj:L2}
 There exists $C=C(n,\rv)>0$ such that
if $M^n$ satisfies  $|\Ric_{M^n}|\leq n-1$ and $\Vol(B_1(p))>\rv>0$, then 
\begin{align}
\fint_{B_1(p)}|\Rm|^2 \leq C\, .
\end{align}
\end{conjecture}

In a different direction, another main result of the paper was to show
 that in dimension $4$, noncollapsed manifolds with bounded diameter and Ricci curvature
 have finite diffeomorphism type.  In higher dimensions, this is too much to hope for; 
see for instance \cite{HeinNaber_TangentCones} where noncollapsed Calabi-Yau manifolds of real dimension $\geq 6$ are constructed with 
unbounded third Betti number.  Nonetheless, it interesting to ask if under the assumption of bounded Ricci curvature, if one should expect a bound on the second Betti number?

\begin{question}
Does there exist $C=C(n,\rv,D)$ such that if 
$M^n$ satisfies $|\Ric_{M^n}|\leq n-1$, $\diam(M^n)\leq D$, 
and $\Vol(B_1(p))>\rv>0$, then $b_2(M^n)\leq C$.
\end{question}

Examples of Perelman show that if the $2$-sided bound on the Ricci tensor is weakened to a lower bound, then the answer is negative.

\bibliographystyle{plain}

\end{document}